\documentclass[final]{siamart190516}
\usepackage{amsmath}
\usepackage{amsfonts}
\usepackage{algorithm}
\usepackage{algorithmic}
\usepackage{multicol}
\usepackage{capt-of}
\usepackage{bm}
\usepackage{placeins}
\usepackage{booktabs}

\DeclareMathOperator*{\argmin}{arg\,min}

\DeclareMathOperator{\sgn}{sgn}
\DeclareMathOperator{\dom}{dom}

\newcommand{\bvec}[1]{\bar{\bm{#1}}}
\newcommand{\tvec}[1]{\tilde{\bm{#1}}}
\newcommand{\grad}[1]{#1'}
\renewcommand{\exp}[1]{\operatorname{exp}\left({#1}\right)}

\newcommand{\tcol}[1]{\multicolumn{2}{c|}{#1}}
\newcommand{\bd}[1]{\textbf{#1}}
\newcommand{\inrange}[1]{\in \{1, ... , #1\}}

\newcommand{\TheTitle}{Gradient Methods with Memory for Minimizing Composite Functions}
\newcommand{\TheAuthors}{Mihai I. Florea}

\title{\TheTitle}

\author{Mihai I. Florea\thanks{Department of Mathematical Engineering, Catholic University of Louvain, Belgium; \newline{} Department of Electronics and Computers, Transilvania University of Brasov, Romania. \newline{} E-mail: \email{mihai.florea@uclouvain.be}}}

\headers{Efficient GMMs for Minimizing Composite Functions}{\TheAuthors}

\begin{document}

\maketitle

\begin{abstract}
The recently introduced Gradient Methods with Memory use a subset of the past oracle information to create an accurate model of the objective function that enables them to surpass the Gradient Method in practical performance. The model introduces an overhead that is substantial on all problems but the smooth unconstrained ones. In this work, we introduce several Gradient Methods with Memory that can solve composite problems efficiently, including unconstrained problems with non-smooth objectives. The auxiliary problem at each iteration still cannot be solved exactly but we show how to alter the model and how to initialize the auxiliary problem solver to ensure that this inexactness does not degrade the convergence guarantees. Moreover, we dynamically increase the convergence guarantees as to provably surpass those of their memory-less counterparts. These properties are preserved when applying acceleration and the containment of inexactness further prevents error accumulation. Our methods are able to estimate key geometry parameters to attain state-of-the-art worst-case rates on many important subclasses of composite problems, where the objective smooth part satisfies a strong convexity condition or a relaxation thereof. In particular, we formulate a near-optimal restart strategy applicable to optimization methods with sublinear convergence guarantees of any order. We support the theoretical results with simulations.
\end{abstract}



\section{Introduction}

The class of large-scale composite problems can be used to model an outstandingly vast array of applications, spanning multiple unrelated areas of mathematics, statistics, computer science, signal processing, optimal control, economics, business administration etc. The term \emph{composite} refers to the structure of the objective in these problems, which can be expressed as the sum between a smooth but expensive to compute function and a simple but possibly non-differentiable and/or non-finite regularizer. The smooth part is generally a data fidelity term that contains most if not all the problem data collected from observations. The regularizer usually imposes certain characteristics on the solution that are not evident from the observations, such as sparsity within a certain space. \emph{Large-scale} denotes a number of variables of the order of millions or more. Under these conditions, the only tractable operations are the computation of function values and gradient-like quantities (first-order operations), and these only at specific points. Second order information such as the Hessian is considered to be unavailable due to limitations in both memory and processing power.

The small number of tractable operations justifies the application of the black-box model to this problem class. Under the black-box assumption, optimization algorithms are only able to attain information on the problem by means of oracle functions, which have the same syntax for all problems within the class. One important advantage of this paradigm lies in the possibility of designing optimization algorithms that can be deployed without modification to solve the entire range of their applications.

Generally, black-box first-order methods fall into two categories: fixed-point schemes such as the Gradient Method and accelerated algorithms such Nesterov's Fast Gradient Method.\footnote{For a detailed exposition see, e.g., \cite{ref_014}. In this work we do not assume the existence of a dual problem formulation and thus do not consider primal-dual schemes.} The former use gradient-type information obtained at the latest test point while the latter aggregate past oracle information in the form of an estimate function or a similar construct, with a memory footprint comparable to the output of a single gradient call.

Considering that oracle information is expensive to obtain, it is worth inquiring how to most effectively utilize the information already collected, beyond an aggregate the size of one oracle call. The recently introduced Gradient Methods with Memory (GMM) attempt to address the issue of information scarcity by maintaining a subset of past gradient calls (also referred to as a bundle) in memory~\cite{ref_018}. New iterates are obtained by minimizing a regularized piece-wise linear model of the objective based on this bundle.

The complexity of the bundle implies that GMMs are nested schemes and we need to distinguish between the outer iterations of the main scheme used to solve the outer (original) problem and the inner (auxiliary) problem solved by the inner iterations of the inner scheme. The structure of the inner problem increases the outer performance without sacrificing generality but the bundle is also responsible for two weaknesses in the original GMMs. First, the inner problem cannot be solved exactly for non-trivial bundle sizes. This inexactness has a negative impact on the convergence guarantees and leads to error accumulation when acceleration is applied~\cite{ref_004}. Second, the inner problem is itself a large-scale problem in most cases and its difficulty may even approach in certain situations the one of the outer problem.

The Exact Gradient Methods with Memory have been recently proposed to address the effects of inexactness~\cite{ref_005}. By employing a two-stage minimization procedure, the errors induced by the bundle are removed from the convergence analysis. This strategy also prevents the acceleration induced accumulation of errors. Moreover, the two-stage system retains the generality of the original GMMs, particularly the ability to handle an extended class of composite problems defined in terms of relative smoothness. However, to ensure generality, the inner problem overhead remains high for many constrained problems or problems with non-differentiable objectives.

In this work, we construct Gradient Methods with Memory for the most common class of composite problems, where the smooth part has a Lipschitz gradient defined in terms of the standard Euclidean norm. Our methods are able not only to contain the inexactness in the model but actually remove it from the convergence analysis altogether. As a direct consequence, they can be accelerated without error accumulation. And most importantly, we redesign the model to ensure that the inner problem has an anti-dual in the form of a Quadratic Program (QP) whose dimensionality is given by the bundle size without any dependence on the outer problem size. Combined with an appropriate model update strategy, we show that the total overhead induced by the model becomes negligible even for constrained problems with non-smooth objectives.

We propose both fixed-point schemes as well as accelerated methods. Unlike the original GMMs, where inexactness degrades the convergence guarantees, we show that by carefully initializing the inner optimization schemes, not only can we preserve the guarantees but, using the dynamic adjustment procedure, our methods provably exceed the corresponding traditional Gradient Methods in terms of worst-case rates. We construct two Accelerated Gradient Methods with Memory, one for non-strongly convex objectives and a generalization that can take into account a known strong convexity parameter. Applying restarts to the former allows it to utilize the broader quadratic functional growth (QFG) property~\cite{ref_011}. We further show how the restart technique can estimate an unknown growth parameter. The benefits of restarting are not limited to our methods, and we formulate a framework that is general enough to be utilized by optimization methods with convergence guarantees of any order. All our methods are by design equipped with a dynamic Lipschitz constant estimation (line-search) procedure.

The theoretical findings are supported by simulations on a collection of synthetic problems that covers a wide range of fields.

\subsection{Organization}

In Section~\ref{label_005} we propose a fixed-point Gradient Method with Memory. We discuss how the model can be constructed and the inner problem can be initialized to negate the effects of inexactness and to ensure that the model overhead becomes negligible for any composite problem. We study the convergence of this algorithm on the entire class as well as on strongly-convex problems and on the more general class of problems with quadratic functional growth. In Section~\ref{label_076} we combine acceleration with our model to obtain an Accelerated Gradient Method with Memory (AGMM) that not only prevents the accumulation of errors but also increases the convergence guarantees dynamically at runtime. We formulate a version that converges faster than the state-of-the-art on strongly convex problems with known strong convexity parameter. For objectives endowed with the QFG property, we discuss how AGMM can be restarted to attain a linear rate. We consider both the case when the parameter known and when it is not. For estimating an unknown QFG parameter, we propose an general adaptive restart scheme. In Section~\ref{label_171} we test our methods on a collection of synthetic problems spanning a wide range of composite problems. We discuss the importance of our theoretical and practical results in Section~\ref{label_185}.

\subsection{Problem setup}

We consider the class of composite problems, given by

\begin{equation} \label{label_001}
\min_{\bm{x} \in \mathbb{R}^n} \{ F(\bm{x}) \overset{\operatorname{def}}{=} f(\bm{x}) + \psi(\bm{x}) \},
\end{equation}
where the function $f$ is full-domain differentiable with Lipschitz gradient and the regularizer $\Psi$ is a proper closed extended value convex function.
The Lipschitz gradient property of $f$ is defined in terms of the standard Euclidean norm $\| . \|_2$ with constant $L_f$ as
\begin{equation} \label{label_002}
f(\bm{y}) \leq f(\bm{x}) + \langle \grad{f}(\bm{x}), \bm{y} - \bm{x} \rangle + \frac{L_f}{2} \| \bm{x} - \bm{y} \|_2^2, \quad \bm{x},\bm{y} \in \mathbb{R}^n .
\end{equation}
The value of $L_f$ need not be known.

The regularizer $\Psi$ is considered proximable (or simple) in the sense that the proximal operator,
defined as
\begin{equation}
\operatorname{prox}_{\tau \Psi}(\bm{x}) \overset{\operatorname{def}}{=} \displaystyle \argmin_{\bm{z} \in \mathbb{R}^n}\left\{\Psi(\bm{z}) + \frac{1}{2\tau}\|\bm{z} - \bm{x}\|_2^2 \right\},
\end{equation}
can be computed with complexity $\mathcal{O}(n)$ for all $\bm{x} \in \mathbb{R}^n$ and $\tau > 0$.

We also assume that the optimization problem in \eqref{label_001} has a non-empty set $X^*$ of optimal points and that the regularizer $\Psi$ is infinite outside the feasible set $X$, which is thus given by $X \overset{\operatorname{def}}{=} \left\{ \bm{x} \ \middle| \ \Psi(\bm{x}) < + \infty \right\}$. We define the projection onto the optimal set and the distance to the optimal set, respectively, as
\begin{equation}
\bm{o}(\bm{x}) \overset{\operatorname{def}}{=} \displaystyle \argmin_{\bm{z} \in X^*}\|\bm{z} - \bm{x}\|_2, \quad d(\bm{x}) \overset{\operatorname{def}}{=} \displaystyle \| \bm{x} - \bm{o}(\bm{x}) \|_2, \quad \bm{x} \in \mathbb{R}^n .
\end{equation}

We also consider two important subclasses of composite problems. The first is the class where the objective displays the quadratic functional growth (QFG) property. Our definition here is a slight generalization over the original one provided in \cite{ref_011}. For a growth parameter $\mu \geq 0$, the QFG property is given by
\begin{equation} \label{label_003}
F(\bm{x}) \geq F^* + \frac{\mu}{2}d(\bm{x})^2, \quad \bm{x} \in \mathbb{R}^n .
\end{equation}
In the second subclass, the objective components possess strong convexity parameters. A function $\phi$ that is sub-differentiable on its entire domain has a strong convexity parameter $\mu_{\phi}$ if
\begin{equation} \label{label_004}
\phi(\bm{y}) \geq \phi(\bm{x}) + \langle \bm{\xi}, \bm{y} - \bm{x} \rangle + \frac{\mu_{\phi}}{2} \| \bm{y} - \bm{x} \|_2^2, \quad \bm{x}, \bm{y} \in \dom \phi ,
\end{equation}
where $\bm{\xi}$ is any subgradient of $\phi$ at $\bm{x}$, i.e., $\bm{\xi} \in \delta F(\bm{x})$. Note that every convex function has a strong convexity parameter of zero. We designate the objective $F$ as strongly convex if both $f$ and $\Psi$ have strong convexity parameters $\mu_f \geq 0$ and $\mu_{\Psi} \geq 0$, respectively, with $\mu \overset{\operatorname{def}}{=} \mu_f + \mu_{\Psi}$. When regular convex functions need to be excluded, strong convexity entails $\mu > 0$. Note that a strongly convex $F$ satisfies the QFG condition with growth parameter $\mu$ (see also \cite{ref_011} for proof).

\section{An Efficient Gradient Method with Memory} \label{label_005}

\subsection{Motivation}

The original Gradient Method with Memory proposed in \cite{ref_018} generates a new iterate $\bm{x}_{+}$ from a previous iterate $\bvec{x}$ in the following way:
\begin{equation} \label{label_006}
\bm{x}_{+} = \displaystyle \argmin_{\bm{y} \in \mathbb{R}^n} \left\{ l_{f}(\bm{y}) + \psi(\bm{y}) + \frac{L}{2} \| \bm{y} - \bvec{x} \|_2^2 \right\} .
\end{equation}
Here $L$ is a parameter dynamically estimated by the algorithm using a line-search procedure and $l_{f}$ is a piece-wise linear approximation of the function $f$ using a collection $\mathcal{Z}$ of test points. The number of test points is $| \mathcal{Z} | = m$ and $\mathcal{Z}$ must include $\bm{z}_1 = \bvec{x}$. The piece-wise linear model $l_{f}$ is given by
\begin{equation} \label{label_007}
l_{f}(\bm{y}) \overset{\operatorname{def}}{=} \displaystyle \max_{\bm{z}_i \in \mathcal{Z}} \{ f(\bm{z}_i) + \langle \grad{f}(\bm{z}_i), \bm{y} - \bm{z}_i \rangle \} , \quad \bm{y} \in \mathbb{R}^n .
\end{equation}
For brevity, we introduce the quantities $\bm{H}_{f}(\mathcal{Z}) \in \mathbb{R}^m$ and $\bm{G}_{f}(\mathcal{Z}) \in \mathbb{R}^n \times \mathbb{R}^m$ as
\begin{align}
(\bm{H}_{f}(\mathcal{Z}))_i &= f(\bm{z}_i) - \langle \grad{f}(\bm{z}_i), \bm{z}_i \rangle, \quad i \in \{ 1, ..., m \}, \\
\bm{G}_{f}(\mathcal{Z}) &= (\grad{f}(\bm{z}_1), \grad{f}(\bm{z}_2), ..., \grad{f}(\bm{z}_m)).
\end{align}
The model $l_{f}$ can be equivalently expressed as
\begin{equation} \label{label_008}
l_{f}(\bm{y}) = \displaystyle \max_{\bm{\lambda} \in \Delta_m} \langle \bm{\lambda}, \bm{H}_{f}(\mathcal{Z}) + \bm{G}_{f}^T(\mathcal{Z}) \bm{y} \rangle, \quad \bm{y} \in \mathbb{R}^n ,
\end{equation}
where $\Delta_m$ is the $m$-dimensional simplex, given by:
\begin{equation}
\Delta_m = \left\{ \bm{\lambda} \in \mathbb{R}^m \ \middle\vert \ \displaystyle \sum_{i = 1}^{m} \bm{\lambda}_i = 1 \mbox{ and } \bm{\lambda}_i \geq 0, \ i \in \{ 1, ..., m \} \right\}.
\end{equation}
The line-search condition that needs to be satisfied by the parameter $L$ is
\begin{equation} \label{label_009}
f(\bm{x}_{+}) \leq l_{f}(\bm{x}_{+}) + \frac{L}{2} \| \bm{x}_{+} - \bvec{x} \|_2^2 .
\end{equation}
From the Lipschitz gradient property of $f$ in \eqref{label_002} we have that the condition in \eqref{label_009} holds in the worst-case for $L > L_f$.

The iterate generation in \eqref{label_006} is a non-trivial optimization problem, partly because it has the same dimensionality as the outer problem \eqref{label_001}. Its anti-dual (the dual taken with the minus sign) has only $m$ variables, gathered in the vector $\bm{\lambda}$, and is written as
\begin{equation} \label{label_010}
\displaystyle \min_{\bm{\lambda} \in \Delta_m} \left\{ \alpha_L(\bm{\lambda}) \overset{\operatorname{def}}{=} \Phi_L(L \bvec{x} - \bm{G}_{f}(\mathcal{Z}) \bm{\lambda}) - \langle \bm{\lambda}, \bm{H}_{f}(\mathcal{Z}) \rangle \right\} ,
\end{equation}
where
\begin{equation} \label{label_011}
\Phi_{L}(\bm{x}) \overset{\operatorname{def}}{=} \displaystyle \max_{\bm{y} \in \mathbb{R}^n} \left\{ \langle \bm{x}, \bm{y} \rangle - \psi(\bm{y}) - \frac{L}{2} \|\bm{y}\|_2^2 \right\} .
\end{equation}
Lemma~2 in \cite{ref_018} establishes that $\Phi_{L}$ is differentiable with Lipschitz gradient for all $L > 0$ with Lipschitz constant $1 / L$ and its derivative is given by
\begin{equation} \label{label_012}
\grad{\Phi_{L}}(\bm{x}) = \operatorname{prox}_{\frac{1}{L} \Psi}\left(\frac{1}{L} \bm{x}\right), \quad \bm{x} \in \mathbb{R}^m .
\end{equation}
Consequently, the problem in \eqref{label_010} has a differentiable objective, with the gradient given by
\begin{equation} \label{label_013}
\alpha'_L(\bm{\lambda}) = -\bm{H}_{f}(\mathcal{Z}) - \bm{G}_{f}^T(\mathcal{Z}) \operatorname{prox}_{\frac{1}{L} \Psi}\left(\bvec{x} - \frac{1}{L} \bm{G}_{f}(\mathcal{Z}) \bm{\lambda}\right).
\end{equation}
The black-box model stipulates that the inner problem \eqref{label_010} can only be solved by first-order methods. Any such method needs to compute at every \emph{inner} iteration the gradient in \eqref{label_013}, which entails $\mathcal{O}(m n)$ floating point operations (FLOP) due to the reliance on the product $\bm{G}_{f}(\mathcal{Z}) \bm{\lambda}$. However, when the objective is smooth unconstrained, which equates to $\Psi = 0$, the gradient in \eqref{label_013} can be refactored to compute only $\bm{Q}_f(\mathcal{Z}) \bm{\lambda}$, where $\bm{Q}_f(\mathcal{Z}) \overset{\operatorname{def}}{=} \bm{G}_{f}^T(\mathcal{Z}) \bm{G}_{f}(\mathcal{Z})$ is an $m \times m$ matrix. Further, the standard Euclidean norm assumption allows the efficient update of $\bm{Q}_f(\mathcal{Z})$ only \emph{once} for every outer iteration, incurring $\mathcal{O}(m n)$ FLOPs, with every inner iteration requiring only $\mathcal{O}(m^2)$ FLOPs. The question arises whether we can benefit from this computational optimization for the common class of composite problems considered in this work.

\subsection{Turning to the composite gradient mapping}

We are able to answer this question affirmatively by relying on the \emph{composite gradient mapping} mathematical construct introduced by Nesterov in \cite{ref_014}. The composite gradient mapping is defined as follows
\begin{equation} \label{label_014}
g_L(\bm{x}) \overset{\operatorname{def}}{=} L (\bm{x} - T_L(\bm{x})), \quad \bm{x} \in \mathbb{R}^n,
\end{equation}
where
\begin{equation} \label{label_015}
T_L(\bm{x}) \overset{\operatorname{def}}{=} \argmin_{\bm{y} \in \mathbb{R}^n} \left\{ f(\bm{x}) + \langle \grad{f}(\bm{x}), \bm{y} - \bm{x} \rangle + \frac{L}{2}\| \bm{y} - \bm{x} \|_2^2 + \Psi(\bm{y}) \right\} ,
\quad \bm{x} \in \mathbb{R}^n ,
\end{equation}
and $L$ is a positive parameter.
The proximal gradient operator $T_L$ can be expressed in terms of the oracle functions as
\begin{equation} \label{label_016}
T_{L}(\bm{x}) = \operatorname{prox}_{\frac{1}{L} \Psi} \left(\bm{x} - \frac{1}{L} \grad{f}(\bm{x}) \right), \ \bm{x} \in \mathbb{R}^n, \ L > 0.
\end{equation}
Based on \eqref{label_016} we can define a step size for the algorithm, given by $1 / L$.

From the first-order optimality conditions in \eqref{label_015} we have that there exists a function $\bm{\xi}_L$ such that $\bm{\xi}_L(\bm{x})$ is a subgradient of $\Psi$ at $T_L(\bm{x})$ for all $\bm{x} \in \mathbb{R}^n$ and
\begin{equation} \label{label_017}
g_L(\bm{x}) = \grad{f}(\bm{x}) + \bm{\xi}_L(\bm{x}), \quad \bm{x} \in \mathbb{R}^n .
\end{equation}
From \eqref{label_017} we see that the composite gradient mapping is able to embed information from both $f$ and $\Psi$. In the sequel we show how to remove the regularizer $\Psi$ from the anti-dual objective and thus bring \eqref{label_010} to a much simpler and less computationally demanding form. Key to our endeavor is the following well-known result.
\begin{lemma} \label{label_018}
If for a certain $\bm{x} \in \mathbb{R}^n$ and $L > 0$ the following local upper bound rule holds
\begin{equation} \label{label_019}
f(T_L(\bm{x})) \leq f(\bm{x}) + \langle \grad{f}(\bm{x}), T_L(\bm{x}) - \bm{x} \rangle + \frac{L}{2}\| T_L(\bm{x}) - \bm{x} \|_2^2 ,
\end{equation}
then the composite objective $F$ is globally lower bounded as
\begin{equation} \label{label_020}
F(\bm{y}) \geq F(T_L(\bm{x})) + \frac{1}{2 L} \| g_L(\bm{x})\|_2^2 + \langle g_L(\bm{x}), \bm{y} - \bm{x} \rangle, \quad \bm{y} \in \mathbb{R}^n .
\end{equation}
\end{lemma}
\begin{proof}
See Appendix C in \cite{ref_008} as well as the proof of Lemma~1 in \cite{ref_007} with the condition $\mu = 0$ applied.
\end{proof}
The lower bound \eqref{label_020} in Lemma~\ref{label_018} suggests an alternative means of storing the past oracle information, this time using both a set $\mathcal{Z}$ of past test points $\bm{z}_i$ as well as a corresponding collection $\mathcal{L}$ of $L(\bm{z}_i)$ parameters, all satisfying \eqref{label_019}. We first define for every $i \in \{ 1, ..., m \}$ the composite gradient mappings $g_F(\bm{z}_i) \overset{\operatorname{def}}{=} g_{L(\bm{z}_i)}(\bm{z}_i)$ and scalars
\begin{equation} \label{label_021}
h_F(\bm{z}_i) \overset{\operatorname{def}}{=} F(T_{L(\bm{z}_i)}(\bm{z}_i)) + \frac{1}{2 L(\bm{z}_i)} \| g_{L(\bm{z}_i)}(\bm{z}_i)\|_2^2 - \langle g_{L(\bm{z}_i)}(\bm{z}_i), \bm{z}_i \rangle ,
\end{equation}
which we concatenate into $\bm{H}_{F}(\mathcal{Z}) \in \mathbb{R}^m$ and $\bm{G}_{F}(\mathcal{Z}) \in \mathbb{R}^n \times \mathbb{R}^m$ as
\begin{align}
(\bm{H}_F(\mathcal{Z}))_i &\overset{\operatorname{def}}{=} h_F(\bm{z}_i), \quad i \in \{ 1, ..., m \}, \label{label_022}\\
\bm{G}_F(\mathcal{Z}) &\overset{\operatorname{def}}{=} ( g_{L(\bm{z}_1)}(\bm{z}_1), g_{L(\bm{z}_2)}(\bm{z}_2), ..., g_{L(\bm{z}_m)}(\bm{z}_m) ) \label{label_023}.
\end{align}

This new model can be used to generate a piece-wise linear global lower bound on the entire composite objective, not just $f$, as
\begin{equation} \label{label_024}
F(\bm{y}) \geq l_{F}(\bm{y}) \overset{\operatorname{def}}{=} \max \left\{ \bm{H}_F(\mathcal{Z}) + \bm{G}_F(\mathcal{Z})^T \bm{y} \right \}, \quad \bm{y} \in \mathbb{R}^n .
\end{equation}
Here $\max$ denotes the vector maximum.

We can thus formulate a new iterate generation procedure by substituting the global lower bound $l_{f} + \Psi$ on $F$ in \eqref{label_006} with $l_{F}$ in \eqref{label_024}. The new procedure is then given by
\begin{equation} \label{label_025}
\bm{x}_{+} = \displaystyle \argmin_{\bm{y} \in \mathbb{R}^n} \left\{ l_{F}(\bm{y}) + \frac{1}{2 a} \| \bm{y} - \bvec{x} \|_2^2 \right\} ,
\end{equation}
where $a = 1 / L$ is the algorithm step size. Note that $a$ is obtained separately from the values $L(\bm{z}_i)$ needed for populating the model.

\subsection{Efficient main iterate generation} \label{label_026}

The obvious starting point for an efficient method is \eqref{label_025}. However, at every iteration we choose a more general model made up of $\bm{H} \in \mathbb{R}^p$ and $\bm{G} \in \mathbb{R}^n \times \mathbb{R}^p$. The size $p$ can vary from one iteration to another, but it is upper bounded by the memory capacity $m$ imposed on the algorithm at startup. The components of the model can either be drawn directly from $\bm{H}_F(\mathcal{Z})$ and $\bm{G}_F(\mathcal{Z})$, can be any convex combination of those components or can even include problem specific global lower bound information, such as objective non-negativity. The advantage of using a convex combination of lower bounds will become more evident when we formulate an accelerated method.
At this point, the only condition we impose on the model for our method is that $(\bm{H})_1 = h_F(\bm{z}_1)$ and $(\bm{G})_1 = g_F(\bm{z}_1)$. Recall that by convention, the previous iterate $\bvec{x}$ is placed first in our model and denoted by $\bm{z}_1$.

The iterate generation procedure in \eqref{label_025} becomes
\begin{equation} \label{label_027}
\bm{x}_{+} = \displaystyle \argmin_{\bm{y} \in \mathbb{R}^n} \left\{ \max_{\bm{\lambda} \in \Delta_p} \langle \bm{\lambda} , \bm{H} + \bm{G}^T \bm{y} \rangle + \frac{1}{2 a} \| \bm{y} - \bvec{x} \|_2^2 \right\} .
\end{equation}
This problem is identical to the inner problem of the original Gradient Method with Memory under the special case of unconstrained minimization under the standard Euclidean setup. The solution is therefore given by
\begin{equation}
\bm{x}_{+} = \bvec{x} - a \bm{G} \bvec{\lambda},
\end{equation}
where, for non-trivial $p$, $\bvec{\lambda}$ is an approximate solution of the anti-dual problem
\begin{equation} \label{label_028}
\min_{\bm{\lambda} \in \Delta_p} \left\{d_{\bm{Q},\bm{B}}(\bm{\lambda}, a) \overset{\operatorname{def}}{=} \frac{a}{2} \langle \bm{\lambda}, \bm{Q} \bm{\lambda} \rangle - \langle \bm{B}, \bm{\lambda} \rangle \right\},
\end{equation}
where $\bm{Q} \overset{\operatorname{def}}{=} \bm{G}^T \bm{G}$ and $\bm{B} = \bm{G}^T \bvec{x} + \bm{H}$.
The anti-dual problem in \eqref{label_028} is a Quadratic Program (QP) with dimensionality $p$ that can be solved with minimal computational overhead (see~\cite{ref_018,ref_005} for a detailed discussion with computational results). We denote a generic QP solver that produces an approximate solution to problem~\eqref{label_028} as $\bm{\lambda}_{\bm{Q}, \bm{B}}(a)$. We place no requirements on the quality of this solution.

It may appear at first glance that altering the lower bounds would suffice in designing a new method. This is not the case because the worst-case results pertaining to line-search in \eqref{label_009} do not apply to $l_{F}$ and we cannot adopt the line-search procedure of the original Gradient Method with Memory. However, the nature of our model provides a simple alternative: searching for a step size $a$ that satisfies $F(\bm{x}_+) \leq -d_{\bm{Q},\bm{B}}(\bar{\bm{\lambda}}, a)$ and reverting to a fallback iterate $\bm{x}_{-}$ if the step size decreases below a threshold $\tau_{-}$. The iterate generation, which simultaneously produces the new iterate $\bm{x}_{+}$ and the corresponding step size $\tau$, becomes
\begin{equation} \label{label_029}
(\bm{x}_{+}, \tau) = \left\{\begin{array}{ll}
(\bvec{x} - a \bm{G} \bvec{\lambda}, a), &\mbox{if }
F(\bvec{x} - a \bm{G} \bvec{\lambda}) \leq -d_{\bm{Q},\bm{B}}(\bar{\bm{\lambda}}, a) \mbox{ and } a > \tau_{-} \\
(\bm{x}_{-}, \tau_{-}), & \mbox{otherwise}
\end{array}\right.,
\end{equation}
where $ \bar{\bm{\lambda}} = \bm{\lambda}_{\bm{Q}, \bm{B}}(a)$. In the current context, the obvious values of $\bm{x}_{-}$ and $\tau_{-}$ are $\bm{z}_1 = T_{L(\bvec{x})}(\bvec{x})$ and $1 / L(\bvec{x})$, respectively. However, in the case of strongly convex objectives, we can formulate bounds that are tighter than in Lemma~\ref{label_018}. Their structure slightly alters the fallback step sizes $\tau_{-}$, as we shall see in Subsubsection~\ref{label_051}.

Now we are ready to proceed with the stipulation of our method.

\subsection{Building the method}

When designing each iterative scheme we use the following convention. Every iteration is indexed by $k \geq 0$ with $k = 0$ marking the first iteration. Quantities already existing at the beginning of iteration $k$ are indexed with $k$ while those that are produced during iteration $k$ are indexed by $k + 1$.

We distinguish within a single outer iteration two independent line-search procedures. The first, the Lipschitz search, produces a Lipschitz constant estimate $L_{k + 1} > 0$ and a new test point $\tvec{x}_{k + 1} = T_{L_{k + 1}}(\bm{x}_k)$. The second procedure, the step size search, computes the actual step size $a_{k + 1} > 0$ that is taken by the algorithm. The strength of the piece-wise linear model lies in the ability of the algorithm to produce in practice step sizes $a_{k + 1}$ far exceeding the value $\tau_{k + 1} = 1 / L_{k + 1}$ of the Gradient Method. Another characteristic is that the second procedure may fail, in which case we can simply revert to $\bm{x}_{-}= \tvec{x}_{k + 1}$ and $\tau_{-} = \tau_{k + 1}$ while retaining the convergence gains and guarantees obtained at previous iterations.

For simplicity, we use an Armijo-style backtracking search for each procedure, both with the same parameters $r_u > 1$ and $r_d \leq 1$. To optimize memory usage, our method updates at every iteration a model of size $p_{k + 1} = \min\{k + 1, m\}$ with variables $\bm{H}_{k + 1} \in \mathbb{R}^{m_{k + 1}}$ and $\bm{G}_{k + 1} \in \mathbb{R}^n \times \mathbb{R}^{p_{k + 1}}$. The resulting efficient Gradient Method with Memory is listed in Algorithm~\ref{label_030}.

\begin{algorithm}[h!]
\caption{An Efficient Gradient Method with Memory for Composite Problems}
\label{label_030}
\begin{algorithmic}[1]
\STATE \bd{Input:} $\bm{x}_0 \in \mathbb{R}^n$, $L_0 > 0$, $r_u > 1 \geq r_d > 0$, $T > 0$.
\FOR{$k = 0,\ldots{},T-1$}
\STATE $L_{k + 1} := r_d L_k$

\LOOP
\STATE $\tvec{x}_{k + 1} := T_{L_{k + 1}}(\bm{x}_k)$
\IF {$f(\tvec{x}_{k + 1}) \leq f(\bm{x}_k) + \langle \grad{f}(\bm{x}_k), \tvec{x}_{k + 1} - \bm{x}_k \rangle + \frac{L_{k + 1}}{2} \| \tvec{x}_{k + 1} - \bm{x}_k \|_2^2$}
\STATE Break from loop
\ELSE
\STATE $L_{k + 1} := r_u L_{k + 1}$
\ENDIF
\ENDLOOP
\STATE $\tau_{k + 1} = 1 / L_{k + 1}$ \label{label_031}
\STATE $\bvec{g}_{k + 1} = L_{k + 1} (\bm{x}_k - \tvec{x}_{k + 1})$
\STATE $\bar{h}_{k + 1} = F(\tvec{x}_{k + 1}) + \frac{1}{2 L_{k + 1}} \|\bvec{g}_{k + 1}\|_2^2 - \langle \bvec{g}_{k + 1}, \bm{x}_k \rangle$ \label{label_032}
\STATE Generate $\bm{G}_{k + 1}$ and $\bm{H}_{k + 1}$ to include $\bvec{g}_{k + 1}$ and $\bar{h}_{k + 1}$
\STATE Generate $\bm{Q}_{k + 1}$ to equal $\bm{G}_{k + 1}^T \bm{G}_{k + 1}$
\STATE $\bm{B}_{k + 1} = \bm{H}_{k + 1} + \bm{G}_{k + 1}^T \bm{x}_k$
\STATE $a_{k + 1} = a_k / r_d$
\LOOP
\IF{$a_{k + 1} < \tau_{k + 1}$}
\STATE $a_{k + 1} = \tau_{k + 1}$
\STATE $\bm{x}_{k + 1} = \tvec{x}_{k + 1}$
\STATE Break from loop
\ENDIF
\STATE $\bm{\lambda}_{k + 1} := \bm{\lambda}_{\bm{Q}_{k + 1}, \bm{B}_{k + 1}}(a_{k + 1})$
\STATE $\bm{x}_{k + 1} := \bm{x}_k - a_{k + 1} \bm{G}_{k + 1} \bm{\lambda}_{k + 1}$
\IF {$F(\bm{x}_{k + 1}) \leq \langle \bm{B}_{k + 1}, \bm{\lambda}_{k + 1} \rangle - \frac{a_{k + 1}}{2} \langle \bm{\lambda}_{k + 1}, \bm{Q}_{k + 1} \bm{\lambda}_{k + 1} \rangle$}
\STATE Break from loop

\ENDIF
\STATE $a_{k + 1} = a_k / r_u$
\ENDLOOP
\ENDFOR
\end{algorithmic}
\end{algorithm}

\subsection{Convergence analysis}

Before we analyze the convergence of our method we need the following result.

\begin{theorem} \label{label_033}
Let the point $\bm{x}_+$ and the step size $\tau$ be generated by one iteration of our efficient Gradient Method with Memory as in \eqref{label_029} with $\bm{x}_{-} = T_{L(\bvec{x})}(\bvec{x})$ and $\tau_{-} = 1 / L(\bvec{x})$. Then, for any $\bm{y} \in \mathbb{R}^n$
\begin{equation}
\frac{1}{2}\| \bm{x}_+ - \bm{y} \|_2^2 \leq \frac{1}{2}\| \bvec{x} - \bm{y} \|_2^2 + \tau \left( F(\bm{y}) - F(\bm{x}_+) \right) .
\end{equation}
\end{theorem}
\begin{proof}
We treat each case separately. First, if $F(\bm{x}_+) \leq -d_{\bm{Q},\bm{B}}(\bar{\bm{\lambda}}, a)$ and $a < 1 /L(\bvec{x})$, we define the following function built around $\bar{\bm{\lambda}}$:
\begin{equation} \label{label_034}
u_{\bar{\bm{\lambda}}}(\bm{y}) \overset{\operatorname{def}}{=} \langle \bar{\bm{\lambda}}, \bm{H} + \bm{G}^T \bm{y} \rangle + \frac{1}{2 a} \| \bm{y} - \bvec{x} \|_2^2 , \quad \bm{y} \in \mathbb{R}^n .
\end{equation}
It follows that the new iterate is the optimum of $u_{\bar{\bm{\lambda}}}$, namely
\begin{equation}
\bm{x}_+ = \bvec{x} - a \bm{G} \bar{\bm{\lambda}} = \argmin_{\bm{y} \in \mathbb{R}^n} u_{\bar{\bm{\lambda}}}(\bm{y}) .
\end{equation}
The first term in the function $u_{\bar{\bm{\lambda}}}$ is a global lower bound on the composite objective $F$ and therefore we have
\begin{equation} \label{label_035}
u_{\bar{\bm{\lambda}}}(\bm{y}) \leq F(\bm{y}) + \frac{1}{2 a} \| \bm{y} - \bvec{x} \|_2^2, \quad \bm{y} \in \mathbb{R}^n .
\end{equation}

The optimal value of $u_{\bar{\bm{\lambda}}}$ matches the one of the dual
\begin{equation} \label{label_036}
u_{\bar{\bm{\lambda}}}^* = u_{\bar{\bm{\lambda}}}(\bm{x}_+) = -d_{\bm{Q},\bm{B}}(\bar{\bm{\lambda}}, a) \geq F(\bm{x}_+) .
\end{equation}
The function $u_{\bar{\bm{\lambda}}}$ is a parabola (i.e. a quadratic function whose Hessian is a positive multiple of the identity matrix) and can also we written as
\begin{equation} \label{label_037}
u_{\bar{\bm{\lambda}}}(\bm{y}) = u_{\bar{\bm{\lambda}}}(\bm{x}_+) + \frac{1}{2 a} \| \bm{y} - \bm{x}_{+} \|_2^2 , \quad \bm{y} \in \mathbb{R}^n .
\end{equation}
Combining \eqref{label_036} with \eqref{label_037} we obtain
\begin{equation} \label{eq:u_F_x_+}
u_{\bar{\bm{\lambda}}}(\bm{y}) \geq F(\bm{x}_+) + \frac{1}{2 a} \| \bm{y} - \bm{x}_{+} \|_2^2, \quad \bm{y} \in \mathbb{R}^n .
\end{equation}
Putting together \eqref{label_035} with \eqref{eq:u_F_x_+} and rearranging terms using $a > 0$ gives
\begin{equation} \label{label_038}
\frac{1}{2}\| \bm{x}_+ - \bm{y} \|_2^2 \leq \frac{1}{2}\| \bvec{x} - \bm{y} \|_2^2 + a \left( F(\bm{y}) - F(\bm{x}_+) \right) , \quad \bm{y} \in \mathbb{R}^n .
\end{equation}

In the second case, expanding \eqref{label_020} applied to $\bm{x}_+$ using \eqref{label_014} we get
\begin{equation} \label{label_068mma_x_plus}
F(\bm{y}) \geq F(\bm{x}_+) + \frac{L(\bvec{x})}{2} \| \bvec{x} - \bm{x}_+ \|_2^2 + L(\bvec{x}) \langle \bvec{x} - \bm{x}_+ , \bm{y} - \bvec{x} \rangle, \quad \bm{y} \in \mathbb{R}^n .
\end{equation}
Rearranging terms in \eqref{label_068mma_x_plus} yields
\begin{equation} \label{label_040}
\frac{1}{2}\| \bm{x}_+ - \bm{y} \|_2^2 \leq \frac{1}{2}\| \bvec{x} - \bm{y} \|_2^2 + \frac{1}{ L(\bvec{x})} \left( F(\bm{y}) - F(\bm{x}_+) \right) , \quad \bm{y} \in \mathbb{R}^n .
\end{equation}
Finally, combining \eqref{label_038} in the first case with \eqref{label_040} in the second case gives the desired result.
\end{proof}
Using Theorem~\ref{label_033} we can formulate a convergence rate for our method.
\begin{theorem} \label{label_041}
The iterates produced by Algorithm~\ref{label_030} satisfy
\begin{equation}
F(\bm{x}_k) - F^* \leq \frac{\| \bm{x}_0 - \bm{x}^*\|_2^2}{2 A_k} \leq \frac{L_u \| \bm{x}_0 - \bm{x}^*\|_2^2}{2 k}, \quad k \geq 1 , \quad \bm{x}^* \in X^*,
\end{equation}
where $A_k$ is the convergence guarantee, in this case given by the sum of all step sizes, namely
\begin{equation} \label{label_042}
A_k \overset{\operatorname{def}}{=} \sum_{i = 1}^{k} a_i, \quad k \geq 0,
\end{equation}
and $L_u \overset{\operatorname{def}}{=} \max \{ r_u L_f, r_d L_0 \}$ is the worst-case Lipschitz constant estimate.
\end{theorem}
\begin{proof}
Taking Theorem~\ref{label_033} with $\bm{y} = \bvec{x}$ we obtain that Algorithm~\ref{label_030} produces iterates with monotonically decreasing objective function values, namely $F(\bm{x}_{k + 1}) \leq F(\bm{x}_k)$ for all $k \geq 1$.
Applying Theorem~\ref{label_033} with $\bm{y} = \bm{x}^*$ and $\tau = a_{k + 1}$ yields
\begin{equation} \label{label_043}
\frac{1}{2} \| \bm{x}_{i + 1} - \bm{x}^* \|_2^2 \leq \frac{1}{2} \| \bm{x}_i - \bm{x}^* \|_2^2 + a_{k + 1} ( F^* - F(\bm{x}_{i + 1}) ), \quad i \geq 0 .
\end{equation}
Summing up \eqref{label_043} for every $i \in \{0, ..., k - 1\}$ and canceling matching terms gives
\begin{equation} \label{label_044}
\frac{1}{2} \| \bm{x}_k - \bm{x}^* \|_2^2 \leq \frac{1}{2} \| \bm{x}_0 - \bm{x}^* \|_2^2 + \displaystyle \sum_{i = 1}^{k} a_i (F(\bm{x}^*) - F(\bm{x}_i)) , \quad k \geq 1.
\end{equation}
Taking the monotonicity of Algorithm~1 into account and refactoring \eqref{label_044}, we obtain the first inequality in our theorem.
The line-search conditions ensure that $L(\bm{x}_k) \leq L_u$ and $a_{k} \geq 1 / L(\bm{x}_{k})$ for all $k \geq 1$, which gives the second inequality.
\end{proof}

Note that the convergence guarantees in Theorem~\ref{label_041} are not influenced by the quality of the inner problem solutions, a considerable improvement over the original Gradient Method with Memory, whose convergence analysis critically depends on the inner problem objective accuracy $\delta$ \cite{ref_018}. This allows inner solvers to terminate when a certain amount of computing resources have been expended (see also \cite{ref_005}).

\subsubsection{Quadratic functional growth}
\begin{proposition}
If $F$ has the QFG property with parameter $\mu > 0$, then Algorithm~1 exhibits a linear convergence rate measured both in iterates and in objective function values, given for every $k \geq 1$ by
\begin{align}
d^2(\bm{x}_k) &\leq \displaystyle \left( \prod_{i = 1}^{k} \frac{1}{1 + a_i \mu} \right) d^2(\bm{x}_0) \leq \left(1 + \frac{\mu}{L_u} \right)^{-k} d^2(\bm{x}_0) , \label{label_045} \\
F(\bm{x}_k) - F^* &\leq \displaystyle \left( \prod_{i = 1}^{k - 1} \frac{1}{1 + a_i \mu} \right) \frac{d^2(\bm{x}_0)}{2 a_k} \leq \left(1 + \frac{\mu}{L_u} \right)^{1-k} \frac{L_u d^2(\bm{x}_0)}{2} . \label{label_046}
\end{align}
\end{proposition}
\begin{proof}
Theorem~\ref{label_033} with $\bm{y} = \bm{o}(\bm{x}_i)$ for all $i \geq 0$ yields
\begin{equation} \label{label_047}
\frac{1}{2} \| \bm{x}_{i + 1} - \bm{o}(\bm{x}_i) \|_2^2 \leq \frac{1}{2} \| \bm{x}_i - \bm{o}(\bm{x}_i) \|_2^2 - a_{i + 1} (F(\bm{x}_{i + 1}) - F^*) .
\end{equation}
From the definition of the projection operation $o$ we have that
\begin{equation} \label{label_048}
\| \bm{x}_{i + 1} - \bm{o}(\bm{x}_{i + 1}) \|_2 \leq \| \bm{x}_{i + 1} - \bm{x}^* \|_2, \quad \bm{x}^* \in X^* .
\end{equation}
Combining \eqref{label_048} using $\bm{x}^* = \bm{o}(\bm{x}_i)$, \eqref{label_047} and the QFG condition in \eqref{label_003} we obtain
\begin{equation} \label{eq:ssrk_i+1_i}
\frac{1}{2} \| \bm{x}_{i + 1} - \bm{o}(\bm{x}_{i + 1}) \|_2^2 \leq \frac{1}{2} \| \bm{x}_{i + 1} - \bm{o}(\bm{x}_{i}) \|_2^2 \leq \frac{1}{2} \| \bm{x}_i - \bm{o}(\bm{x}_i) \|_2^2 - \frac{\mu a_{i + 1}}{2} \| \bm{x}_{i + 1} - \bm{o}(\bm{x}_{i + 1}) \|_2^2 .
\end{equation}
Refactoring \eqref{eq:ssrk_i+1_i} yields
\begin{equation} \label{label_049}
\frac{1}{2} \| \bm{x}_{i + 1} - \bm{o}(\bm{x}_{i + 1}) \|_2^2 \leq \frac{1}{1 + \mu a_{i + 1}} \| \bm{x}_{i} - \bm{o}(\bm{x}_{i}) \|_2^2 .
\end{equation}
Iterating \eqref{label_049} for $i \in \{0, ..., k - 1\}$ and using $a_{i + 1} \geq 1 / L_u$ we get \eqref{label_045}.

On the other hand, \eqref{label_047} with $i = k - 1$ also implies that
\begin{equation} \label{label_050}
F(\bm{x}_{k}) - F^* \leq \frac{1}{2 a_{k}} \| \bm{x}_{k - 1} - \bm{o}(\bm{x}_{k - 1}) \|_2^2, \quad k \geq 1.
\end{equation}
Applying \eqref{label_045} to \eqref{label_050} gives \eqref{label_046}.

\end{proof}

\subsubsection{Strong convexity} \label{label_051}

We have seen that the quadratic functional growth property is sufficient to obtain a linear convergence rate for Algorithm~\ref{label_030}. We consider the case when $f$ and $\Psi$ have strong convexity parameters $\mu_f \geq 0$ and $\mu_{\Psi} \geq 0$, respectively, and derive a much better rate.

We first formulate a tighter lower bound than the one stipulated in Lemma~\ref{label_018}. To do so, we extend the definition of the composite gradient mapping in \eqref{label_014} as
\begin{equation} \label{label_052}
g_L(\bm{x}) \overset{\operatorname{def}}{=} (L + \mu_{\Psi}) (\bm{x} - T_L(\bm{x})), \quad \bm{x} \in \mathbb{R}^n ,
\end{equation}
with the definition of $T_L(\bm{x})$ unchanged.
\begin{lemma} \label{label_053}
If the objective is strongly convex and the condition in \eqref{label_019} holds, then we can formulate a global lower bound on the objective $F$ for all $\bm{y} \in \mathbb{R}^n$ as
\begin{equation} \label{label_054}
F(\bm{y}) \geq F(T_L(\bm{x})) + \frac{1}{2 (L + \mu_{\Psi})} \| g_{L}(\bm{x})\|_2^2 + \langle g_{L}(\bm{x}), \bm{y} - \bm{x} \rangle + \frac{\mu_f + \mu_{\Psi}}{2} \| \bm{y} - \bm{x} \|_2^2.
\end{equation}
\end{lemma}
\begin{proof}
See the proof of Lemma~1 in \cite{ref_007}.
\end{proof}
Lemma~\ref{label_053} confirms that the extended definition of the composite gradient preserves its gradient-like properties previously present in the \eqref{label_014} form as long as we define the fallback step size as $\tau_{-} = 1 / (L + \mu_{\Psi})$. These properties are the descent update (gradient step) and the descent rule (see also \cite{ref_008}), respectively given for all $\bm{x} \in \mathbb{R}^n$ by
\begin{equation}
T_L(\bm{x}) = \bm{x} - \frac{1}{L + \mu_{\Psi}} \cdot g_L(\bm{x}), \quad
F(T_L(\bm{x})) \leq F(\bm{x}) - \frac{1}{2 (L + \mu_{\Psi})} \| g_L(\bm{x}) \|_2^2.
\end{equation}

The form of the lower bound also suggests a simple way of extending the scalars introduced in \eqref{label_021} as
\begin{equation} \label{label_055}
h_F(\bm{z}_i) \overset{\operatorname{def}}{=} F(T_{L(\bm{z}_i)}(\bm{z}_i)) + \frac{1}{2 (L(\bm{z}_i) + \mu_{\Psi})} \| g_{L(\bm{z}_i)}(\bm{z}_i)\|_2^2 - \langle g_{L(\bm{z}_i)}(\bm{z}_i), \bm{z}_i \rangle ,
\end{equation}
with \eqref{label_022} and \eqref{label_023} as well as all other definitions now referencing \eqref{label_052} and \eqref{label_055}.

Using this tighter bound we can now obtain a stronger version of Theorem~\ref{label_041} in the form of Proposition~\ref{label_033_sc}. Moreover, strong convexity is not a necessary condition for this version. Within the scope of Proposition~\ref{label_033_sc}, $\mu_f$ and $\mu_{\Psi}$ do not need to satisfy \eqref{label_004} for $f$ and $\Psi$, respectively. Their properties can be limited to those mentioned explicitly in within the scope of the proposition. Acknowledging the potential of the following results, we leave the application of this observation to a more general context as a topic for future work.
\begin{proposition} \label{label_033_sc}
If $\bm{x}_+$ and $\tau$ have been generated using \eqref{label_029} with $\bm{x}_{-} = T_{L(\bvec{x})}(\bvec{x})$ and $\tau_{-} = 1 / ({L(\bvec{x})} + \mu_{\Psi})$ where the first entry in the model is given by $(\bm{H})_1 = h_F(\bvec{x})$, $(\bm{G})_1 = g_F(\bvec{x})$ and the following property holds with parameter $\mu \geq 0$ for a certain $\bm{x}^* \in X^*$ :
\begin{equation} \label{label_057}
F^* \geq \max \{ \bm{H} + \bm{G}^T \bm{x}^* \} + \frac{\mu}{2}\| \bvec{x} - \bm{x}^* \|_2^2,
\end{equation}
then we have
\begin{equation}
\frac{1}{2}\| \bm{x}_+ - \bm{x^*} \|_2^2 \leq \frac{1}{2}(1 - \tau \mu)\| \bvec{x} - \bm{x}^* \|_2^2 + \tau \left( F^* - F(\bm{x}_+) \right).
\end{equation}
\end{proposition}
\begin{proof}
If $F(\bm{x}_+) \leq -d_{\bm{Q},\bm{B}}(\bar{\bm{\lambda}}, a)$ and $a < \tau_{-}$, the condition in \eqref{label_057} implies for any $\bvec{\lambda} \in \Delta_m$ that
\begin{equation} \label{label_058}
u_{\bar{\bm{\lambda}}}(\bm{x}^*) = \langle \bar{\bm{\lambda}}, \bm{H} + \bm{G}^T \bm{x}^* \rangle + \frac{1}{2 a} \| \bm{x}^* - \bvec{x} \|_2^2 \leq F^* + \left(\frac{1}{2a} - \frac{\mu}{2}\right) \| \bm{x}^* - \bvec{x} \|_2^2 .
\end{equation}
In this case \eqref{label_036} still holds and applying \eqref{label_037} we get
\begin{equation} \label{label_059}
u_{\bar{\bm{\lambda}}}(\bm{x}^*) = u_{\bar{\bm{\lambda}}}^* + \frac{1}{2a} \| \bm{x}^* - \bvec{\bm{x}}_+ \|_2^2 \geq F(\bm{x}_+) + \frac{1}{2a} \| \bm{x}^* - \bvec{\bm{x}}_+ \|_2^2.
\end{equation}
Combining \eqref{label_058} with \eqref{label_059} and rearranging terms using $a > 0$ gives
\begin{equation} \label{label_060}
\frac{1}{2}\| \bm{x}_+ - \bm{x^*} \|_2^2 \leq \frac{1}{2}(1 - a \mu)\| \bvec{x} - \bm{x}^* \|_2^2 + a \left( F^* - F(\bm{x}_+) \right) .
\end{equation}
If the step size search condition is not satisfied, using \eqref{label_057} and the assumption $(\bm{H})_1 = h_F(\bvec{x})$ and $(\bm{G})_1 = g_F(\bvec{x})$, we obtain that
\begin{equation} \label{label_061}
F^* \geq F(\bm{x}_+) + \frac{1}{2 \tau_{-}} \| \bvec{x} - \bm{x}_{+} \|_2^2 + \frac{1}{\tau_{-}}\langle \bvec{x} - \bm{x}_{+}, \bm{x}^* - \bvec{x} \rangle + \frac{\mu}{2}\| \bvec{x} - \bm{x}^* \|_2^2 .
\end{equation}
Expanding and rearranging terms in \eqref{label_061} using $\tau_{-} > 0$ gives
\begin{equation} \label{label_062}
\frac{1}{2}\| \bm{x}_+ - \bm{x}^* \|_2^2 \leq \frac{1}{2} (1 - \mu \tau_{-} ) \| \bvec{x} - \bm{x}^* \|_2^2 + \tau_{-} ( F(\bm{x}^*) - F(\bm{x}_+) ) .
\end{equation}
Putting together \eqref{label_060} and \eqref{label_062} completes the proof.
\end{proof}
The conditions in Proposition~\ref{label_033_sc} relate to strong convexity in the following way.
\begin{proposition} \label{label_063}
If $\mu_f$ and $\mu_{\psi}$ are strong convexity parameters satisfying \eqref{label_004} for $f$ and $\Psi$, respectively, and the model is built exclusively by incorporating convex combinations of past iterates, then the conditions in Proposition~\ref{label_033_sc} are satisfied for every $\bm{x} \in X^*$.
\end{proposition}
\begin{proof}
From \eqref{label_054} in Lemma~\ref{label_018} we have for all $\bm{y} \in \mathbb{R}^n$ that
\begin{equation} \label{label_064}
F(\bm{y}) \geq h_F(\bm{z}_i) + \langle g_F(\bm{z}_i), \bm{y} \rangle + \frac{\mu}{2} \| \bm{y} - \bm{z}_i \|_2^2, \quad i \in \{ 1, ..., m \} .
\end{equation}

Because each $\bm{z}_i$ is either the current or a previous iterate, we can iteratively apply Theorem~\ref{label_033} with $\bm{y} \in X^*$ and obtain
\begin{equation} \label{label_065}
\|\bm{x}_+ - \bm{x}^*\|_2 \leq \|\bm{z}_i - \bm{x}^*\|_2, \quad \bm{x}^* \in X^*.
\end{equation}
Using \eqref{label_065} in \eqref{label_064} with $\bm{y} \in X^*$ yields
\begin{equation} \label{label_066}
F^* \bm{1}_m \succeq H_F(\mathcal{Z}) + G_F(\mathcal{Z}) \bm{x}^* + \left(\frac{\mu}{2} \|\bm{x}_+ - \bm{x}^*\|_2^2\right) \bm{1}_m , \quad \bm{x}^* \in X^*,
\end{equation}
where $\bm{1}_m$ is the vector of all ones of size $m$.

By assumption, our model is built using a matrix with non-negative entries $\bm{S} \in \mathbb{R}_+^{p \times m}$, $p \leq m$, with $\bm{S}_{1,1} = 1$ and $\bm{S} \bm{1}_m = \bm{1}_p$, such that $\bm{H} = \bm{S} H_F(\mathcal{Z})$ and $\bm{G} = \bm{S} G_F(\mathcal{Z})$. Multiplying both sides of \eqref{label_066} by $\bm{S}$ and applying the vector maximum yields the desired result.
\end{proof}
Proposition~\ref{label_063} provides us with a means of modifying Algorithm~1 to attain a faster rate on composite objectives when \eqref{label_057} holds. Specifically, we alter lines~\ref{label_031} through \ref{label_032} to, respectively, become:
\begin{align}
\tau_{k + 1} &= 1 / (L_{k + 1} + \mu_{\psi}), \label{label_067}\\
\bvec{g}_{k + 1} &= (L_{k + 1} + \mu_{\psi}) (\bm{x}_k - \tvec{x}_{k + 1}), \\
\bar{h}_{k + 1} &= F(\bm{x}_{k + 1}) + \frac{1}{2 (L_{k + 1} + \mu_{\psi})} \|\bvec{g}_{k + 1}\|_2^2 - \langle \bvec{g}_{k + 1}, \bm{x}_k \rangle . \label{label_068}
\end{align}
Note that for $\mu_{\Psi} = 0$, Algorithm~1 is not altered. Therefore, our modification is merely a generalization meant to take advantage of a relaxed form of strong convexity in $\Psi$, if present.

\begin{theorem} \label{label_069}
If the objective $F$ and the model update procedure in the generalized version of Algorithm~1, incorporating \eqref{label_067} through \eqref{label_068}, satisfy the conditions in Proposition~\ref{label_033_sc}, then the produced iterates exhibit a linear convergence rate in function values, given by
\begin{equation} \label{label_070}
F(x_{k}) - F^* \leq \frac{\| \bm{x}_0 -\bm{x}^* \|_2^2}{2 \sum_{i = 1}^{k} a_i \pi_i}
\leq \frac{\mu \left(1 - q_u \right)^k \| \bm{x}_0 -\bm{x}^* \|_2^2}
{1 - \left(1 - q_u \right)^k} , \quad k \geq 1,
\end{equation}
where $\pi_k \overset{\operatorname{def}}{=} \displaystyle \prod_{j = 1}^{k} \frac{1}{1 - a_j \mu}$ for all $k \geq 1$ and $q_u \overset{\operatorname{def}}{=} \frac{\mu}{L_u + \mu_{\Psi}}$.
\end{theorem}
\begin{proof}
Taking Proposition~\ref{label_033_sc} with $\bvec{x} = \bm{x}_i$ for $i \geq 0$ we obtain
\begin{equation} \label{label_071}
\frac{1}{2} \| \bm{x}_{i + 1} -\bm{x}^* \|_2^2 \leq \frac{1}{2}(1 - a_{i + 1} \mu) \| \bm{x}_i -\bm{x}^* \|_2^2 + a_{i + 1} (F^* - F(\bm{x}_{i + 1}) ) .
\end{equation}
Multiplying both sides of \eqref{label_071} with $\pi_{i + 1}$ produces
\begin{equation} \label{label_072}
\frac{\pi_{i + 1}}{2} \| \bm{x}_{i + 1} -\bm{x}^* \|_2^2 \leq \frac{\pi_i}{2} \| \bm{x}_i -\bm{x}^* \|_2^2 + a_{i + 1} \pi_{i + 1} (F^* - F(\bm{x}_{i + 1}) ) .
\end{equation}
Adding together \eqref{label_072} for $i \in \{0, ..., k - 1 \}$ and canceling matching terms yields
\begin{equation} \label{label_073}
\left( \sum_{i = 1}^{k} a_i \pi_i \right) (F(\bm{x}_k) - F^*) \leq \frac{1}{2} \| \bm{x}_{0} -\bm{x}^* \|_2^2 - \frac{\pi_k}{2} \| \bm{x}_{k} -\bm{x}^* \|_2^2 .
\end{equation}
Discarding the $(\pi_k / 2) \| \bm{x}_{k} -\bm{x}^* \|_2^2$ term gives the first inequality in \eqref{label_070}.
The $a_i \pi_i$ terms are increasing in $a_i$ and applying $a_i \geq \tau_u \overset{\operatorname{def}}{=} 1 / (L_u + \mu_{\Psi})$ we get
\begin{equation} \label{label_074}
a_i \pi_i \geq \tau_u \left( 1 - \mu \tau_u \right)^{-i}, \quad i \geq 1.
\end{equation}
Summing up \eqref{label_074} for $i \in \{1, ..., k - 1 \}$ we obtain the sum of a geometric progression which applied to the former inequality in \eqref{label_070} yields the latter one.
\end{proof}

\subsection{Model overhead} \label{label_075}

In Algorithm~\ref{label_030}, the model overhead comprises the inner problem setup followed by the generation of approximate solutions for several values of the step size $a_{k + 1}$, decreasing until fit. The inner problem setup at iteration $k \geq 0$ entails the update of the model $\bm{Z}_{k + 1}$ and $\bm{H}_{k + 1}$ as well as of the compound quantities $\bm{Q}_{k + 1}$ and $\bm{B}_{k + 1}$. The model can be updated directly from the output of the first state and hence the complexity of the setup lies in the update of the compounds. As discussed in our precursor work \cite{ref_006}, the update of $\bm{Q}_{k + 1}$ can be performed recursively and in place. The corresponding computational load is dominated by the matrix-vector product $\bm{G}_k \bar{\bm{g}}_{k + 1}$ with complexity $\mathcal{O}(p_{k+ 1} n)$ FLOPs. The update of $\bm{B}_{k + 1}$ is also expensive because of the need to compute $\bm{G}_{k + 1} \bm{x}_{k}$, also with complexity $\mathcal{O}(p_{k+ 1} n)$.

To obtain an approximate solution $\bm{\lambda}_{\bm{Q}_{k + 1}, \bm{B}_{k + 1}}(A, \delta)$, which is a $p_{k + 1}$ dimensional QP, one may resort to any optimization scheme. The accuracy of the algorithm output does not affect the worst-case rate of Algorithm~\ref{label_030}. Therefore, we assume that the complexity of the inner problem is bounded by a small constant.

However, every inner problem instance is accompanied by the recomputation of $F(\bm{x}_{k + 1})$, needed to evaluate the step size search stopping criterion. Therefore, each step size search iteration incurs an additional oracle cost aside from the otherwise negligible cost of the inner problem, which is notably separate from the backtrack overhead in the Lipschitz search.

\section{Accelerating the Efficient Gradient Method with Memory} \label{label_076}

We have seen in Theorem~\ref{label_041} that for fixed-point schemes, a convergence guarantee can be computed directly from the step sizes. Higher rates in accelerated methods can be obtained by maintaining the convergence guarantee explicitly as part of an \emph{estimate function}.\footnote{The results in this section pertaining to the most general composite problem class have been previously presented in \cite{ref_006}. More detailed derivations can be found therein.}

\subsection{The general case} \label{sec:algo-nsc}

\subsubsection{Estimate functions}
At least two different forms of estimate functions have been proposed in the literature (e.g., \cite{ref_013,ref_014}). To accommodate models with memory, we propose a modification based on both forms, which allows the starting point to be infeasible as in \cite{ref_014} while retaining the estimate sequence property according to \cite{ref_013}, as follows:
\begin{equation} \label{label_077}
\psi_{k}(\bm{x}) = W_k(\bm{x}) + \frac{1}{2 A_k} \| \bm{x} - \bm{x}_0 \|_2^2, \quad \bm{x} \in \mathbb{R}^n,
\end{equation}
where $\psi_{k}$ is the estimate function at the beginning of iteration $k \geq 1$. Here, $W_k$ can be any expression that satisfies $W_k(\bm{x}^*) \leq F(\bm{x}^*)$, with $\bm{x}^*$ being a specific point in the optimal set, which we considered fixed. A sufficient condition is that $W_k$ is a global lower bound on the composite objective $F$. The quantity $A_k$ takes on the role convergence guarantee if the following estimate sequence property is maintained at every iteration:
\begin{equation} \label{label_078}
F(\bm{x}_k) \leq \psi_k^* \overset{\operatorname{def}}{=} \min_{\bm{x} \in \mathbb{R}^n} \psi_k(\bm{x}), \quad k \geq 1 .
\end{equation}
The proof follows directly from the definition as
\begin{equation} \label{label_079}
\begin{gathered}
F(\bm{x}_k) - F(\bm{x}^*) \leq \psi_k^* - F(\bm{x}^*) \leq \psi_k(\bm{x}^*) - F(\bm{x}^*) \\
\leq F(\bm{x}^*) + \frac{1}{2 A_k} \| \bm{x}^* - \bm{x}_0 \|_2^2 - F(\bm{x}^*) = \frac{1}{2 A_k} \| \bm{x}^* - \bm{x}_0 \|_2^2 .
\end{gathered}
\end{equation}
The structure of the estimate function in \eqref{label_077} reveals a simple way of improving accelerated algorithms: the tightening of $W_k$ around $\bm{x}_k$, which permits the increase of the guarantee $A_k$ without violating the estimate sequence property in \eqref{label_078}.

\subsubsection{Piece-wise linear lower bounds} Gradient Methods with Memory maintain a piece-wise linear lower bound on the objective function leading to a marked performance increase over the Gradient Method (GM). This bound also has the tightness that we seek. Therefore, we propose an estimate function built around a piece-wise linear lower model on the objective $F$, determined by $\bm{H}$ and $\bm{G}$, and further parametrized by the convergence guarantee $A > 0$ of the form
\begin{equation} \label{label_080}
\tilde{\psi}(\bm{x}) = \max\left\{ \bm{H} + \bm{G}^T \bm{x} \right\} + \frac{1}{2 A} \| \bm{x} - \bm{x}_0 \|_2^2, \quad \bm{x} \in \mathbb{R}^n .
\end{equation}
The constraints imposed of $\bm{H}$ and $\bm{G}$ will be discussed in Subsubsection~\ref{label_076_model}.
The estimate sequence property in \eqref{label_078} becomes
\begin{equation} \label{label_078_opt_prob}
F(\bm{x}_{+}) \leq \tilde{\psi}^* = \min_{\bm{x} \in \mathbb{R}^n} \max_{\bm{\lambda} \in \Delta_m} \left\{ \langle \bm{\lambda}, \bm{H} + \bm{G}^T \bm{x} \rangle + \frac{1}{2 A} \| \bm{x} - \bm{x}_0 \|_2^2 \right\}.
\end{equation}
The right-hand side is an optimization problem. It's anti-dual is given by
\begin{equation} \label{label_082}
\min_{\bm{\lambda} \in \Delta_m} \left\{d_{\bm{Q},\bm{C}}(\bm{\lambda}, A) \overset{\operatorname{def}}{=} \frac{A}{2} \langle \bm{\lambda}, \bm{Q} \bm{\lambda} \rangle - \langle \bm{C}, \bm{\lambda} \rangle \right\} ,
\end{equation}
with $\bm{Q} \overset{\operatorname{def}}{=} \bm{G}^T \bm{G}$ and $\bm{C} \overset{\operatorname{def}}{=} \bm{H} + \bm{G}^T \bm{x}_0$. The anti-dual is a QP with a structure identical to \eqref{label_028}. An approximate solution is given by $\bm{\lambda}_{\bm{Q}, \bm{C}}(A)$. Because an exact solution of \eqref{label_082} is most often unattainable, we cannot reliably compute the optimal value of the estimate function $\tilde{\psi}^*$ in \eqref{label_080} to enforce the estimate sequence property in \eqref{label_078_opt_prob}. Instead, we can define a simpler estimate function around an approximate solution $\bm{\lambda} = \bm{\lambda}_{\bm{Q}, \bm{C}}(A)$ as
\begin{equation} \label{label_083}
\psi(\bm{x}) = \langle \bm{\lambda}, \bm{H} + \bm{G}^T \bm{x} \rangle + \frac{1}{2 A} \| \bm{x} - \bm{x}_0 \|_2^2, \quad \bm{x} \in \mathbb{R}^n ,
\end{equation}
with $\psi^* = -d_{\bm{Q},\bm{C}}(\bm{\lambda}, A)$.

\subsubsection{The middle problem} Having defined the lower bound and estimate function, we need to find a means of increasing the convergence guarantee. As opposed to the outer problem in \eqref{label_001} (where the variable is $\bm{x}$) and the inner problem in \eqref{label_078_opt_prob} (with the variable given by $\bm{\lambda}$), this constitutes a middle problem (with variable $A$). To facilitate derivations, we temporarily assume that the inner problem can be solved exactly, with the solution written as $\bvec{\lambda}(A)$, and then extend the middle scheme to the inexact case. The task becomes a root finding problem for the function
\begin{equation} \label{label_084}
\Gamma(A) \overset{\operatorname{def}}{=} -d_{\bm{Q},\bm{C}}(\bvec{\lambda}(A), A) - F(\bm{x}_{+}), \quad A > 0.
\end{equation}
An efficient method for solving this problem is the Newton method, with the update at inner iteration $t \geq 0$ given by $A^{(t + 1)} = A^{(t)} - \Gamma(A^{(t)}) / \Gamma'(A^{(t)}) .$
Our exactness assumption allows us to apply Danskin's lemma and obtain
$\Gamma'(A) = -\frac{1}{2} \langle \bvec{\lambda}(A), \bm{Q} \bvec{\lambda}(A) \rangle$. The Newton iterations can be written in closed form as
\begin{equation} \label{label_085}
A^{(t + 1)} = A^{(t)} + 2 \frac{\Gamma(A^{(t)})}{\langle \bm{\lambda}(A^{(t)}), \bm{Q} \bm{\lambda}(A^{(t)}) \rangle} .
\end{equation}
In the exact case, Newton's method has a quadratic convergence rate. Inexactness degrades these guarantees but does not affect the validity of the procedure as long as we ensure feasibility. Namely, we enforce the estimate sequence property at every Newton iteration, that is $\Gamma(A^{(t)}) \geq 0$ for $t \geq 0$. The denominator is non-negative, with a value of zero meaning an exact solution to the outer problem has already been found. Therefore, a Newton update will only lead to an \emph{increase} in convergence guarantees and the estimate sequence property provides and effective stopping criterion for the algorithm (aside from the zero denominator check). The entire middle method is summarized in Algorithm~\ref{label_086}, with parameters $\bm{Q} \in \mathbb{R}^{p \times p}$, $\bm{C} \in \mathbb{R}^{p}$, $F(\bm{x}_+) < \infty$, $\bm{\lambda}^{(0)} \in \mathbb{R}^{p}$ and $A^{(0)} > 0$.

\begin{algorithm}[h!]
\caption{The middle method \newline{}
Newton($\bm{Q}, \bm{C}, F(\bm{x}_+), \bm{\lambda}^{(0)}, A^{(0)} )$}
\label{label_086}
\begin{algorithmic}[1]
\STATE $\bm{\lambda}_{\operatorname{valid}} := \bm{\lambda}^{(0)}$
\STATE $A_{\operatorname{valid}} := A := A^{(0)}$
\FOR{$t = 0, \ldots{}, N - 1$}
\STATE $\bm{\lambda} := \bm{\lambda}_{\bm{Q}, \bm{C}}(A)$ with starting point $\bm{\lambda}^{(0)}$
\STATE $\psi^* := \langle \bm{C}, \bm{\lambda} \rangle - \frac{A}{2} \langle \bm{\lambda}, \bm{Q} \bm{\lambda} \rangle$
\IF {$\psi^* < F(\bm{x}_+)$}
\STATE Break from loop
\ENDIF
\STATE $\bm{\lambda}_{\operatorname{valid}} := \bm{\lambda}$
\STATE $A_{\operatorname{valid}} := A$
\STATE $A := A + 2 \frac{\psi^* - F(\bm{x}_+)}{\langle \bm{\lambda}, \bm{Q} \bm{\lambda} \rangle}$
\ENDFOR
\RETURN $\bm{\lambda}_{\operatorname{valid}}, A_{\operatorname{valid}}$
\end{algorithmic}
\end{algorithm}

With the middle method formulated, it remains to construct a model and determine a way to initialize the Newton method with a large enough but feasible value of $A$.

\subsubsection{Test points and weights} Fixed-point schemes impose two unnecessary constraints on the iterates and weights: the iterates are always chosen among the points where the oracle is called (test points) and the weights $a_{k + 1}$ with $k \geq 0$, correspond to step sizes taken using the results of those oracle calls. The Fast Gradient Method (FGM) attains the asymptotically optimal rate (up to a constant) on the composite problem class by discarding these assumptions. Instead, each test point is given by the weighted averages between an iterate and an estimate function optimal point. When dealing with non-strongly convex objectives, we have
\begin{equation} \label{label_087}
\bm{y}_{k + 1} = \frac{1}{A_{k} + a_{k + 1}}(A_k \bm{x}_k + a_{k + 1} \bm{v}_k), \quad k \geq 0.
\end{equation}
where $\bm{v}_k = \argmin_{\bm{z} \in \mathbb{R}^n} \psi_k(\bm{z})$, $k \geq 1$. The initial estimate function is undefined and we set $\bm{v}_0 = \bm{x}_0$. Each estimate function $W_{k + 1}$ term, $k \geq 0$, is likewise obtained as the weighted average of the previous $W_k$ term and $w_{k + 1}$, a simple lower bound on $F$ at the point $\bm{x}^*$, with weights $A_k$ and $a_{k + 1} > 0$, respectively. The lower bound $w_{k + 1}$ is simple because it is constructed using the oracle output obtained during a single iteration. Note that accelerated methods differ substantially from the Gradient Method because they maintain an aggregate lower bound of \emph{all} past information (without necessarily storing a bundle) as opposed to using only the more recently acquired simple bound. The values $a_{k + 1}$ no longer correspond to step sizes, but instead become weights applied to simple lower bounds.

The algorithm that produces the largest weights and highest worst-case convergence guarantees currently known on the composite problem class is the Accelerated Composite Gradient Method (ACGM)~\cite{ref_008,ref_009}. ACGM not only displays excellent theoretical and practical performance, but includes several accelerated first-order methods as particular cases. Due to its aggressive weight update, generality and performance, we choose this method as a template when designing our methods.

ACGM has the convergence guarantee given by $A_{k + 1} = A_k + a_{k + 1}$, $k \geq 0$, just as in Theorem~\ref{label_041}, with the weight given by
\begin{equation} \label{label_088}
L_{k + 1} a_{k + 1}^2 = a_{k + 1} + A_{k}, \quad k \geq 0.
\end{equation}
We set the update rule in \eqref{label_088} as the starting point for the convergence guarantee update and mandate that
\begin{equation} \label{label_089}
A_{k + 1} \geq A_k + a_{k + 1}, \quad k \geq 0.
\end{equation}

\subsubsection{Model} \label{label_076_model}

Now that we have identified the conditions on the weights and test points required to achieve a state-of-the-art worst-case rate, we can formulate the model.

Before we begin, we need to address the special case that arises in the first iteration, corresponding to $k = 0$. Acceleration cannot be applied at this point because the oracle call history is empty. A bundle cannot be utilized either. Therefore, it only remains to perform a simple GM step, which for the same reasons coincides with the first step of FGM. We thus have $\bm{v}_0 = \bm{x}_0$, $\bm{y}_1 = \bm{x}_0$, $\bm{x}_1 = T_{L_1}(\bm{x}_0)$. The model becomes $\bm{H}_1 = h_1$, $\bm{G}_1 = g_1$. The inner problem has the unique exact solution $\lambda_1 = 1$. We set $A_0 = 0$ because $F(\bm{x}_0)$ may be unbounded and cannot be subject to any guarantee. The preliminary choice $A_1 = a_1 = 1 / L_1$ leads to $\psi_1^* = \bar{h}_1 + \langle \bvec{g}_1, \bm{x}_0 \rangle- (A_1 / 2) \| \bvec{g}_1 \|_2^2 = F(\bm{x}_1)$. This value of $A_1$ is an exact root of $\Gamma$ and cannot be increased. Therefore, the middle scheme in Algorithm~\ref{label_086} is not called.

At the start of the second and latter iterations ($k \geq 1$), an estimate function $\psi_k$ is already available, having been constructed using \eqref{label_083}. This function is uniquely determined by 3 parameters: $A_k$, $h_k \overset{\operatorname{def}}{=} \langle \bm{\lambda}_k, \bm{H}_k \rangle$ and $\bm{g}_k \overset{\operatorname{def}}{=} \bm{G}_k \bm{\lambda}_k$. We have seen that the middle scheme requires a feasible starting point. The lower bound weighted average condition suggests some simple constraints that can be imposed on our model to achieve this. The constraints need to affect only \emph{two} entries in our model (consequently, we only call the middle method when $p_{k + 1} \geq 2$) and the starting point of the inner optimization scheme. Specifically, we impose for all $k \geq 1$ that
\begin{equation} \label{label_091}
\begin{aligned}
&(\bm{H}_{k + 1})_1 = h_k, & \quad & (\bm{G}_{k + 1})_1 = \bm{g}_k, \\
&(\bm{H}_{k + 1})_2 = h_F(\bm{y}_{k + 1}), & \quad & (\bm{G}_{k + 1})_2 = g_F(\bm{y}_{k + 1}) .
\end{aligned}
\end{equation}
The starting point of the innermost scheme $\bm{\lambda}_{k + 1}^{(0)}$ is set as
\begin{equation} \label{label_092}
(\bm{\lambda}^{(0)}_{k + 1})_i = \frac{1}{A_k + a_{k + 1}}
\left\{ \begin{array}{ll} A_k ,& i = 1, \\ a_{k + 1} ,& i = 2, \\ 0, & i \in \{ 3, ..., p \}. \end{array} \right.,\quad k \geq 1.
\end{equation}
The above assumptions with be validated in the convergence analysis.

With the outer, middle and inner problems defined, we construct an efficient Accelerated Gradient Method with Memory using the same Armijo-style line-search as in the first stage of Algorithm~\ref{label_030} iterations. Note that the middle scheme takes on the role performed by the step size search procedure in Algorithm~\ref{label_030} but benefits from the parameter-free quality of Newton's rootfinding algorithm. Our method is listed in Algorithm~\ref{label_093}.

\begin{algorithm}[h!]
\caption{An Efficient Accelerated Gradient Method with Memory for Composite Problems}
\label{label_093}
\begin{algorithmic}[1]
\STATE \bd{Input:} $\bm{x}_0 \in \mathbb{R}^n$, $L_0 > 0$, $r_u > 1 \geq r_d > 0$, $T > 0$
\STATE $\bm{v}_0 = \bm{x}_0$, $A_0 = 0$
\FOR{$k = 0,\ldots{},T-1$}
\STATE $L_{k + 1} := r_d L_k$
\LOOP
\STATE $a_{k + 1} := \frac{1 + \sqrt{1 + 4 L_{k + 1} A_k}}{2 L_{k + 1}}$ \label{label_094}
\STATE $\bm{y}_{k + 1} := \frac{1}{A_k + a_{k + 1}}(A_k \bm{x}_k + a_{k + 1} \bm{v}_k)$ \label{label_095}
\STATE $\bm{x}_{k + 1} := T_{L_{k + 1}}(\bm{y}_{k + 1})$ \label{label_096} \\[1mm]

\IF {$f(\bm{x}_{k + 1}) \leq f(\bm{y}_{k + 1}) + \langle \grad{f}(\bm{y}_{k + 1}), \bm{x}_{k + 1} - \bm{y}_{k + 1} \rangle + \frac{L_{k + 1}}{2} \| \bm{x}_{k + 1} - \bm{y}_{k + 1} \|_2^2$} \label{label_097}
\STATE Break from loop
\ELSE
\STATE $L_{k + 1} := r_u L_{k + 1}$
\ENDIF
\ENDLOOP
\STATE $\bar{\bm{g}}_{k + 1} = L_{k + 1} (\bm{y}_{k + 1} - \bm{x}_{k + 1})$
\STATE $\bar{h}_{k + 1} = F(\bm{x}_{k + 1}) + \frac{1}{2 L_{k + 1}} \|\bar{\bm{g}}_{k + 1}\|_2^2 - \langle \bar{\bm{g}}_{k + 1}, \bm{y}_{k + 1} \rangle$
\IF {$k = 0$}
\STATE $\bm{H}_1 = \bar{h}_{1}$, $\bm{G}_1 = \bvec{g}_1$, $\bm{Q}_1 = \| \bvec{g}_1 \|_2^2$, $\bm{C}_1 = \bar{h}_{1} + \langle \bvec{g}_1, \bm{x}_0 \rangle$, $\bm{\lambda}_{1} = 1$, $A_1 = a_1$
\ELSE
\STATE Generate $\bm{H}_{k + 1}$ and $\bm{G}_{k + 1}$ to satisfy \eqref{label_091}
\STATE Generate $\bm{Q}_{k + 1}$ to equal $\bm{G}_{k + 1}^T \bm{G}_{k + 1}$
\STATE Generate $\bm{C}_{k + 1}$ to equal $\bm{H}_{k + 1} + \bm{G}_{k + 1}^T \bm{x}_0$
\STATE Generate $\bm{\lambda}^{(0)}_{k + 1}$ according to \eqref{label_092}
\STATE $\bm{\lambda}_{k + 1}, A_{k + 1} = \operatorname{Newton}(\bm{Q}_{k + 1}, \bm{C}_{k + 1}, F( \bm{x}_{k + 1}), \bm{\lambda}^{(0)}_{k + 1}, A_k + a_{k + 1})$
\ENDIF
\STATE $h_{k + 1} = \langle \bm{H}_{k + 1}, \bm{\lambda}_{k + 1} \rangle$
\STATE $\bm{g}_{k + 1} = \bm{G}_{k + 1} \bm{\lambda}_{k + 1}$
\STATE $\bm{v}_{k + 1} = \bm{x}_0 - A_{k + 1} \bm{G}_{k + 1}^T \bm{\lambda}_{k + 1}$
\ENDFOR
\end{algorithmic}
\end{algorithm}

\subsubsection{Convergence analysis} \label{label_098}

We have seen in \eqref{label_079} that the estimate functions ensure a worst-case convergence rate as long as the estimate sequence property in \eqref{label_078} is maintained at every iteration by Algorithm~\ref{label_093}. For $k = 1$, \eqref{label_078} is satisfied with equality. It remains to show that \eqref{label_078} carries over to subsequent iterations. The estimate function is outputted by the middle scheme, which is able to preserve the estimate sequence property only if it holds in its initial state. Therefore, we must first verify the feasibility of the starting conditions of the middle scheme, for any algorithmic state.

We use the following notation for the estimate function at iteration $k \geq 1$ just before the middle scheme is invoked:
\begin{equation} \label{label_099}
\bar{\psi}_{k + 1}(\bm{x}) \overset{\operatorname{def}}{=} \langle \bm{\lambda}_{k + 1}^{(0)}, \bm{H}_{k + 1} + \bm{G}_{k + 1}^T \bm{x} \rangle + \frac{1}{2 (A_k + a_{k + 1})}\| \bm{x} - \bm{x}_0 \|_2^2, \quad \bm{x} \in \mathbb{R}^n.
\end{equation}
We prove the feasibility condition as follows.
\begin{theorem} \label{label_100}
If at iteration $k \geq 1$ we have that $\psi_k^* \geq F(\bm{x}_k)$, then $\bar{\psi}_{k + 1}^* \geq F(\bm{x}_{k + 1})$.
\end{theorem}
\begin{proof}
Throughout this proof we consider every $\bm{x} \in \mathbb{R}^n$ and $k \geq 1$. Our conditions on the model in \eqref{label_091} and \eqref{label_092} allow us to expand \eqref{label_099} as
\begin{align}
(A_k + a_{k + 1}) \bar{\psi}_{k + 1}(\bm{x})
&= A_k (h_k + \langle \bm{g}_k, \bm{x} \rangle ) + a_{k + 1} ( \bar{h}_{k + 1} + \langle \bar{\bm{g}}_{k + 1}, \bm{x} \rangle ) + \frac{1}{2} \| \bm{x} - \bm{x}_0 \|_2^2 \nonumber \\
&= A_k \psi_k(\bm{x}) + a_{k + 1} ( \bar{h}_{k + 1} + \langle \bar{\bm{g}}_{k + 1}, \bm{x} \rangle ) . \label{label_101}
\end{align}
The previous estimate function $\psi_k$ can be written as a parabola, given by
\begin{equation} \label{label_102}
\psi_k(\bm{x}) = \psi_k^* + \frac{1}{A_k} \| \bm{x} - \bm{v}_k \|_2^2 .
\end{equation}
Expanding $\psi_k(\bm{x})$ in \eqref{label_101} using \eqref{label_102} and taking the minimum on both sides yields
\begin{align}
(A_k + a_{k + 1}) \bar{\psi}_{k + 1}^*
& = A_k \psi_k^* + \min_{\bm{x} \in \mathbb{R}^n} \left\{ a_{k + 1} (\bar{h}_{k + 1} + \langle \bar{\bm{g}}_{k + 1}, \bm{x} \rangle ) + \frac{1}{2} \| \bm{x} - \bm{v}_k \|_2^2 \right\} \nonumber \\
& = A_k \psi_k^* + a_{k + 1} (\bar{h}_{k + 1} + \langle \bar{\bm{g}}_{k + 1}, \bm{v}_k \rangle ) - \frac{a_{k + 1}^2}{2} \| \bar{\bm{g}}_{k + 1} \|_2^2 . \label{label_103}
\end{align}
Using the assumption along with \eqref{label_020} we obtain
\begin{equation} \label{label_104}
\psi_k^* \geq F(\bm{x}_k) \geq \bar{h}_{k + 1} + \langle \bar{\bm{g}}_{k + 1}, \bm{x}_k \rangle .
\end{equation}
Applying \eqref{label_104} in \eqref{label_103} we get that
\begin{equation} \label{label_105}
(A_k + a_{k + 1})\bar{\psi}_{k + 1}^* \geq (A_k + a_{k + 1}) \bar{h}_{k + 1}
+ \langle \bar{\bm{g}}_{k + 1}, A_k \bm{x}_k + a_{k + 1} \bm{v}_k \rangle - \frac{a_{k + 1}^2}{2} \| \bar{\bm{g}}_{k + 1} \|_2^2 .
\end{equation}
We complete the proof by employing both \eqref{label_087} and \eqref{label_088} in \eqref{label_105}
as follows:
\begin{align}
(A_k + a_{k + 1})\bar{\psi}_{k + 1}^* \geq & (A_k + a_{k + 1}) \bar{h}_{k + 1}
+ \langle \bar{\bm{g}}_{k + 1}, (A_k + a_{k + 1}) \bm{y}_{k + 1} \rangle
\nonumber \\
& - \frac{A_k + a_{k + 1}}{2 L_{k + 1}} \| \bar{\bm{g}}_{k + 1} \|_2^2 = (A_k + a_{k + 1}) F(\bm{x}_{k + 1}).
\end{align}
\end{proof}
With a feasible starting point, the middle scheme will only produce iterates that satisfy the estimate sequence property in \eqref{label_078}. The Newton iterations will only increase the convergence guarantees, as much as possible, thereby ensuring \eqref{label_089}. We can now apply \eqref{label_079} to obtain
\begin{equation} \label{label_106}
F(\bm{x}_k) - F(\bm{x}^*) \leq \frac{1}{2 A_k} \| \bm{x}_0 - \bm{x}^* \|_2^2, \quad k \geq 1.
\end{equation}
The growth rate of the convergence guarantees is governed by \eqref{label_088} and \eqref{label_089}. These conditions are compatible with the worst-case results in \cite[Appendix~E]{ref_008}, stated as
\begin{equation} \label{label_107}
A_k \geq \frac{(k + 1)^2}{4 L_u} , \quad k \geq 1,
\end{equation}
yielding
\begin{equation}
F(x_k) - F(x^*) \leq \frac{2 L_u \| \bm{x}_0 - \bm{x}^* \|_2^2}{(k + 1)^2} , \quad k \geq 1.
\end{equation}

\subsection{Strong convexity}

Accelerated first-order methods can be modified to exhibit linear convergence when the objective is strongly convex. This procedure is described in detail for ACGM in \cite{ref_008,ref_009}. ACGM retains on this subclass the best known convergence guarantees and displays excellent practical performance. For these reasons, we use ACGM as a starting point for extending the efficient Accelerated Gradient Method with Memory to handle strongly convex objectives.

We therefore consider the special case stipulated in Lemma~\ref{label_053} as well as in Proposition~\ref{label_063}, wherein $f$ and $\Psi$ have strong convexity parameters $\mu_f \geq 0$ and $\mu_{\Psi} \geq 0$, respectively, with $\mu = \mu_f + \mu_{\Psi}$.

When constructing our method, the general form of the estimate functions remains the same as in \eqref{label_077}, with strong convexity affecting only the structure of the lower bounds, which according to \eqref{label_054} in Lemma~\ref{label_053} become
\begin{equation} \label{label_108}
F(\bm{y}) \geq h_F(\bm{z}_i) + \langle g_F(\bm{z}_i), \bm{y} \rangle + \frac{\mu}{2}\|\bm{y} - \bm{z}_i \|_2^2, \quad \bm{y} \in \mathbb{R}^n, \quad i \in \{ 1, ..., m \} ,
\end{equation}
with $h_F$ and $g_F$ generalized to handle strong convexity using \eqref{label_055} and \eqref{label_052}, respectively. To simplify the analysis, we center the quadratic term in each bound on the same point, the initial $\bm{x}_0$, and define for $i \in \{ 1, ..., m \}$ the quantities
\begin{equation} \label{label_109}
\bar{h}_F(\bm{z}_i) = h_F(\bm{z}_i) + \frac{\mu}{2} \left( \| \bm{z}_i \|_2^2 - \| \bm{x}_0 \|_2^2 \right) , \qquad
\bar{g}_F(\bm{z}_i) = g_F(\bm{z}_i) + \mu (\bm{x}_0 - \bm{z}_i) .
\end{equation}
The lower bounds in \eqref{label_108} can be expressed using \eqref{label_109} as
\begin{equation} \label{label_110}
F(\bm{y}) \geq \bar{h}_F(\bm{z}_i) + \langle \bar{g}_F(\bm{z}_i) , \bm{y} \rangle + \frac{\mu}{2}\|\bm{y} - \bm{x}_0 \|_2^2, \quad \bm{y} \in \mathbb{R}^n, \ i \in \{ 1, ..., m \} .
\end{equation}
The form of the lower bounds in \eqref{label_110} allows us to generalize the ideal estimate function in \eqref{label_080} as
\begin{equation} \label{label_111}
\tilde{\psi}(\bm{x}) = \max\left\{ \bm{H} + \bm{G}^T \bm{x} \right\} + \frac{1}{2}\left(\frac{1}{A} + \mu\right) \| \bm{x} - \bm{x}_0 \|_2^2, \quad \bm{x} \in \mathbb{R}^n .
\end{equation}
The necessary constraints on $\bm{H}$ and $\bm{G}$ require additional notation and will be stated in the sequel.

We see from \eqref{label_111} that strong convexity replaces $A$ in \eqref{label_080} with the quantity
$\sigma(A) \overset{\operatorname{def}}{=} A / (1 + \mu A)$.
The anti-dual problem arising from the estimate sequence property is now given by $d_{\bm{Q}, \bm{C}}(\bm{\lambda}, \sigma(A))$ with $d$ having the same expression as in \eqref{label_082}. The estimate function built around the approximate solution $\bm{\lambda} = \bm{\lambda}_{\bm{Q}, \bm{C}}(\sigma(A))$, with $\bm{Q} \overset{\operatorname{def}}{=} \bm{G}^T \bm{G}$ and $\bm{C} \overset{\operatorname{def}}{=} \bm{H} + \bm{G}^T \bm{x}_0$, is
\begin{equation} \label{label_112}
\psi(\bm{x}) = \langle \bm{\lambda}, \bm{H} + \bm{G}^T \bm{x} \rangle + \frac{1}{2 \sigma(A)} \| \bm{x} - \bm{x}_0 \|_2^2, \quad \bm{x} \in \mathbb{R}^n ,
\end{equation}
with $\psi^* = -d_{\bm{Q}, \bm{C}}(\bm{\lambda}, \sigma(A))$.

We proceed to constructing the middle method under the same assumptions as those in Subsection~\ref{sec:algo-nsc}. The estimate sequence gap $\Gamma(A) = \psi^*(A) - F(\bm{x}_+)$ now has the derivative given by
\begin{equation} \label{label_113}
\Gamma'(A) = -\frac{1}{2 (1 + \mu A)^2} \langle \bm{\lambda}(A), Q \bm{\lambda}(A) \rangle .
\end{equation}
We notice from \eqref{label_113} that the Newton iterations ensure an increase the convergence guarantees in the strongly convex case as well. We can now state the middle method in Algorithm~\ref{label_086-sc-method}.

\begin{algorithm}[h!]
\caption{The middle method for strongly convex objectives \newline{}
Newton-SC($\bm{Q}, \bm{C}, F(\bm{x}_+), \mu, \bm{\lambda}^{(0)}, A^{(0)} )$}
\label{label_086-sc-method}
\begin{algorithmic}[1]
\STATE $\bm{\lambda}_{\operatorname{valid}} := \bm{\lambda}^{(0)}$
\STATE $A_{\operatorname{valid}} := A := A^{(0)}$
\FOR{$t = 0, \ldots{}, N - 1$}
\STATE $\bm{\lambda} := \bm{\lambda}_{\bm{Q}, \bm{C}}\left(\frac{A}{1 + \mu A}\right)$ with starting point $\bm{\lambda}^{(0)}$
\STATE $\psi^* := \langle \bm{C}, \bm{\lambda} \rangle - \frac{A}{2 (1 + \mu A)} \langle \bm{\lambda}, \bm{Q} \bm{\lambda} \rangle$
\IF {$\psi^* < F(\bm{x}_+)$}
\STATE Break from loop
\ENDIF
\STATE $\bm{\lambda}_{\operatorname{valid}} := \bm{\lambda}$
\STATE $A_{\operatorname{valid}} := A$
\STATE $A := A + 2 (1 + \mu A)^2 \cdot{} \frac{ \psi^* - F(\bm{x}_+) }{\langle \bm{\lambda}, \bm{Q} \bm{\lambda} \rangle}$
\ENDFOR
\RETURN $\bm{\lambda}_{\operatorname{valid}}, A_{\operatorname{valid}}$
\end{algorithmic}
\end{algorithm}

With ACGM as our basis for algorithm design, we generate the new weights and test points as
\begin{gather}
(L_{k + 1} + \mu_{\Psi}) a^2_{k + 1} = \bar{A}_{k + 1} \bar{\gamma}_{k + 1}, \label{label_114}\\
\bm{y}_{k + 1} = \frac{1}{A_k \bar{\gamma}_{k + 1} + a_{k + 1} \gamma_k}(A_k \bar{\gamma}_{k + 1} \bm{x}_k + a_{k + 1} \gamma_k \bm{v}_k ) , \label{label_115}
\end{gather}
where
\begin{equation}
\bar{A}_{k + 1} \overset{\operatorname{def}}{=} A_k + a_{k + 1}, \quad \bar{\gamma}_{k + 1} \overset{\operatorname{def}}{=} 1 + \mu (A_k + a_{k + 1}), \quad
\gamma_k \overset{\operatorname{def}}{=} 1 + \mu A_k, \quad k \geq 0 .
\end{equation}
Next, by rearranging and canceling matching terms in \eqref{label_109}, we simplify the expression of the new model entries and obtain
\begin{align}
\bar{\bm{g}}_{k + 1} &= (L_{k + 1} - \mu_f) \bm{y}_{k + 1} - (L_{k + 1} + \mu_{\Psi}) \bm{x}_{k + 1} + \mu \bm{x}_0, \label{label_116} \\
\bar{h}_{k + 1} &= F(\bm{x}_{k + 1}) - \frac{L_{k + 1} - \mu_f}{2} \| \bm{y}_{k + 1} \|_2^2 + \frac{L_{k + 1} + \mu_{\Psi}}{2} \| \bm{x}_{k + 1} \|_2^2 - \frac{\mu}{2} \| \bm{x}_0 \|_2^2 . \label{label_117}
\end{align}
We can now write down the constraints on the model for all $k \geq 1$, using the updated definitions $h_k \overset{\operatorname{def}}{=} \langle \bm{\lambda}_k, \bm{H}_k \rangle$ and $\bm{g}_k \overset{\operatorname{def}}{=} \bm{G}_k \bm{\lambda}_k$, as
\begin{equation} \label{label_118}
\begin{aligned}
&(\bm{H}_{k + 1})_1 = h_k, & \quad & (\bm{G}_{k + 1})_1 = g_k, \\
&(\bm{H}_{k + 1})_2 = \bar{h}_F(\bm{y}_{k + 1}), & \quad & (\bm{G}_{k + 1})_2 = \bar{h}_F(\bm{y}_{k + 1}) .
\end{aligned}
\end{equation}

The starting point of the innermost scheme $\bm{\lambda}_{k + 1}^{(0)}$ is set according to \eqref{label_092}.

The model is updated and the middle method in Algorithm~\ref{label_086-sc-method} is called. Upon termination, the new estimate function optimum point in \eqref{label_112} is given by
\begin{equation}
\bm{v}_{k + 1} = \bm{x}_0 - \frac{A_{k + 1}}{1 + \mu A_{k + 1}} \bm{G}_{k + 1}^T \bm{\lambda}_{k + 1} .
\end{equation}

With all components in place, we state the Efficient Accelerated Gradient Method with Memory for strongly convex composite problems in Algorithm~\ref{label_119}.

\begin{algorithm}[h!]
\caption{An Efficient Accelerated Gradient Method with Memory for Strongly Convex Composite Problems}
\label{label_119}
\begin{algorithmic}[1]
\STATE \bd{Input:} $\bm{x}_0 \in \mathbb{R}^n$, $\mu_f \geq 0$, $\mu_{\Psi} \geq 0$, $L_0 > 0$, $r_u > 1 \geq r_d > 0$, $T > 0$
\STATE $\bm{v}_0 = \bm{x}_0$, $A_0 = 0$, $\mu = \mu_f + \mu_{\Psi}$
\FOR{$k = 0,\ldots{},T-1$}
\STATE $\gamma_k = 1 + \mu A_k$
\STATE $L_{k + 1} := r_d L_k$
\LOOP
\STATE $a_{k + 1} := \frac{\gamma_k + A_k \mu}{2 (L_{k + 1} - \mu_f)} \left(1 + \sqrt{1 + \frac{4 (L_{k + 1} - \mu_f) A_k \gamma_k}{(\gamma_k + A_k \mu)^2}} \right)$\\[1mm]
\STATE $\bar{\gamma}_{k + 1} := 1 + \mu (A_k + a_{k + 1})$\\[1mm]
\STATE $\bm{y}_{k + 1} := \frac{1}{A_k \bar{\gamma}_{k + 1} + a_{k + 1} \gamma_k}(A_k \bar{\gamma}_{k + 1} \bm{x}_k + a_{k + 1} \gamma_k \bm{v}_k)$\\[1mm]
\STATE $\bm{x}_{k + 1} := T_{L_{k + 1}}(\bm{y}_{k + 1})$ \\[1mm]

\IF {$f(\bm{x}_{k + 1}) \leq f(\bm{y}_{k + 1}) + \langle \grad{f}(\bm{y}_{k + 1}), \bm{x}_{k + 1} - \bm{y}_{k + 1} \rangle + \frac{L_{k + 1}}{2} \| \bm{x}_{k + 1} - \bm{y}_{k + 1} \|_2^2$}
\STATE Break from loop
\ELSE
\STATE $L_{k + 1} := r_u L_{k + 1}$
\ENDIF
\ENDLOOP
\STATE $\bar{\bm{g}}_{k + 1} = (L_{k + 1} - \mu_f) \bm{y}_{k + 1} - (L_{k + 1} + \mu_{\Psi}) \bm{x}_{k + 1} + \mu \bm{x}_0$\\[1mm]
\STATE $\bar{h}_{k + 1} = F(\bm{x}_{k + 1}) - \frac{L_{k + 1} - \mu_f}{2} \| \bm{y}_{k + 1} \|_2^2 +
\frac{L_{k + 1} + \mu_{\Psi}}{2} \| \bm{x}_{k + 1} \|_2^2 - \frac{\mu}{2} \| \bm{x}_0 \|_2^2$\\[1mm]
\IF {$k = 0$}
\STATE $\bm{H}_1 = \bar{h}_{1}$, $\bm{G}_1 = \bvec{g}_1$, $\bm{Q}_1 = \| \bvec{g}_1 \|_2^2$, $\bm{C}_1 = \bar{h}_{1} + \langle \bvec{g}_1, \bm{x}_0 \rangle$, $\bm{\lambda}_{1} = 1$, $A_1 = a_1$
\ELSE
\STATE Generate $\bm{H}_{k + 1}$ and $\bm{G}_{k + 1}$ to satisfy \eqref{label_091}
\STATE Generate $\bm{Q}_{k + 1}$ to equal $\bm{G}_{k + 1}^T \bm{G}_{k + 1}$
\STATE Generate $\bm{C}_{k + 1}$ to equal $\bm{H}_{k + 1} + \bm{G}_{k + 1}^T \bm{x}_0$
\STATE Generate $\bm{\lambda}^{(0)}_{k + 1}$ according to \eqref{label_092}
\STATE $\bm{\lambda}_{k + 1}, A_{k + 1} := \operatorname{Newton-SC}(\bm{Q}_{k + 1}, \bm{C}_{k + 1}, F( \bm{x}_{k + 1}), \mu, \bm{\lambda}^{(0)}_{k + 1}, A_k + a_{k + 1})$
\ENDIF
\STATE $h_{k + 1} = \langle \bm{H}_{k + 1}, \bm{\lambda}_{k + 1} \rangle$
\STATE $\bm{g}_{k + 1} = \bm{G}_{k + 1} \bm{\lambda}_{k + 1}$
\STATE $\bm{v}_{k + 1} = \bm{x}_0 - \frac{A_{k + 1}}{1 + \mu A_{k + 1}} \bm{g}_{k + 1}$
\ENDFOR
\end{algorithmic}
\end{algorithm}

\subsubsection{Convergence analysis} \label{label_120}

As we have seen in Subsubsection~\ref{label_098}, a worst-case convergence rate can be guaranteed if the estimate sequence property in \eqref{label_078} can be shown to hold for all $k \geq 1$.
Following the non-strongly convex case, the first iteration of Algorithm~\ref{label_119} is a descent step. We have $\bm{y}_1 = \bm{v}_0 = \bm{x}_0$, $\bm{x}_1 = T_{L_1}(\bm{x}_0)$. The model is determined by $\bm{H}_1 = \bar{h}_1$ and $\bm{G}_1 = \bvec{g}_1$ which directly gives $\bm{\lambda}_1 = 1$ without employing the inner method. Because $A_1 = a_1 = 1 / (L_1 - \mu_f)$ we have that $\psi_1^* = \bar{h}_1 + \langle \bvec{g}_1, \bm{x}_0 \rangle- \frac{A_1}{2 (1 + \mu A_1)} \| \bvec{g}_1 \|_2^2 = F(\bm{x}_1)$. Consequently, the middle method is not called and the value of $A_1$ stays $1 / (L_1 - \mu_f)$.

It remains to prove by induction that the estimate sequence property is preserved from one iteration to another \emph{before} the middle method is called thus ensuring the feasibility of the middle problem. Interestingly, this preservation does not require $f$ to be strongly convex which is why we first present the following result, which notably differs from Lemma~\ref{label_053}.

\begin{lemma} \label{label_121}
If $\Psi$ has a strong convexity parameter $\mu_{\Psi}$ and the condition in \eqref{label_019} holds, then $F$ is lower bounded for all $\bm{x}, \bm{y} \in \mathbb{R}^n$ as
\begin{equation} \label{label_122}
F(\bm{y}) \geq F(T_L(\bm{x})) + \frac{1}{2 (L + \mu_{\Psi})} \| g_{L}(\bm{x})\|_2^2 + \langle g_{L}(\bm{x}), \bm{y} - \bm{x} \rangle + \frac{\mu_{\Psi}}{2} \| \bm{y} - \bm{x} \|_2^2 .
\end{equation}
\end{lemma}
\begin{proof}
Herein we consider all $\bm{x}, \bm{y} \in \mathbb{R}^n$.
From the convexity of $f$ at $\bm{y}$ we have
\begin{equation} \label{label_123}
f(\bm{y}) \geq f(\bm{x}) + \langle f'(\bm{x}) , \bm{y} - \bm{x} \rangle .
\end{equation}
The first-order optimality conditions in \eqref{label_015} imply that there exists a subgradient $\bm{\xi}(\bm{x})$ of $\Psi$ at $T_L(\bm{x})$ such that
\begin{equation} \label{label_124}
L(\bm{x} - T_L(\bm{x})) = \bm{\xi}(\bm{x}) + f'(\bm{x}) .
\end{equation}
The definition of the strong convexity parameter in \eqref{label_004} applied for $\bm{\xi}(\bm{x})$ gives
\begin{equation} \label{label_125}
\Psi(\bm{y}) \geq \Psi(T_L(\bm{x})) + \langle \bm{\xi}(\bm{x}), \bm{y} - T_L(\bm{x}) \rangle + \frac{\mu_{\Psi}}{2} \| \bm{y} - T_L(\bm{x}) \|_2^2 .
\end{equation}
Adding together \eqref{label_019}, \eqref{label_124}, \eqref{label_125}, and canceling matching terms leads to
\begin{equation} \label{label_126}
F(\bm{x}) \geq F(T_L(\bm{x})) + \frac{\mu_{\Psi}}{2} \| \bm{y} - T_L(\bm{x}) \|_2^2 - \frac{L}{2} \| T_L(\bm{x}) - \bm{x} \|_2^2 | + L \langle \bm{x} - T_L(\bm{x}), \bm{y} - T_L(\bm{x}) \rangle .
\end{equation}
Rearranging terms in \eqref{label_126} yields \eqref{label_122}.
\end{proof}
Now we can proceed to the main result.
\begin{theorem} \label{label_127}
For a composite objective $F$ wherein we only know that $\Psi$ has a strong convexity parameter $\mu_{\Psi} \geq 0$, if at iteration $k \geq 1$ we have that $\psi_k^* \geq F(\bm{x}_k)$, then $\bar{\psi}_{k + 1}^* \geq F(\bm{x}_{k + 1})$, where $\bar{\psi}_{k + 1}$ is now given for all $\bm{x} \in \mathbb{R}^n$ and $k \geq 1$ by
\begin{equation} \label{label_128}
\bar{\psi}_{k + 1}(\bm{x}) \overset{\operatorname{def}}{=} \langle \bm{\lambda}_{k + 1}^{(0)}, \bm{H}_{k + 1} + \bm{G}_{k + 1}^T \bm{x} \rangle + \frac{\bar{\gamma}_{k + 1}}{2 \bar{A}_{k + 1}}\| \bm{x} - \bm{x}_0 \|_2^2.
\end{equation}
\end{theorem}
\begin{proof}
The results in this proof hold for all $\bm{x} \in \mathbb{R}^n$ and $k \geq 1$. The constraints on the model in \eqref{label_118} along with the middle scheme starting conditions in \eqref{label_092} enable us to write down the estimate function in \eqref{label_128} as
\begin{equation} \label{label_129}
\begin{aligned}
(A_k + a_{k + 1}) \bar{\psi}_{k + 1}(\bm{x})
& = A_k (h_k + \langle \bm{g}_k, \bm{x} \rangle ) + a_{k + 1} ( \bar{h}_{k + 1} + \langle \bvec{g}_{k + 1}, \bm{x} \rangle ) \\
& \quad + \frac{1 + \mu (A_k + a_{k + 1})}{2} \| \bm{x} - \bm{x}_0 \|_2^2 .
\end{aligned}
\end{equation}
The previous estimate function can be expressed in two ways. First, using \eqref{label_112} as
\begin{equation} \label{label_130}
\psi_k(\bm{x}) = h_k + \langle \bm{g}_k, \bm{x} \rangle + \frac{1}{2}\left(\mu + \frac{1}{A_k} \right) \| \bm{x} - \bm{x}_0 \|_2^2,
\end{equation}
and second using the canonical form as
\begin{equation} \label{label_131}
\psi_k(\bm{x}) = \psi_k^* + \frac{\gamma_k}{A_k} \| \bm{x} - \bm{v}_k \|_2^2 .
\end{equation}
The definition in \eqref{label_128} allows us to bring that estimate function also to a canonical form, given by
\begin{equation} \label{label_132}
\bar{\psi}_{k + 1}(\bm{x}) = \bar{\psi}_{k + 1}^* + \frac{\bar{\gamma}_{k + 1}}{\bar{A}_{k + 1}} \| \bm{x} - \bvec{v}_{k + 1} \|_2^2 .
\end{equation}
Moreover, new model entries in \eqref{label_116} and \eqref{label_117} by definition satisfy
\begin{equation} \label{label_133}
\begin{gathered}
\bar{h}_{k + 1} + \langle \bvec{g}_{k + 1}, \bm{x} \rangle + \frac{\mu}{2} \| \bm{x} - \bm{x}_0 \|_2^2 = F(\bm{x}_{k + 1}) + \frac{1}{2 \bar{L}_{k + 1}} \| \bm{c}_{k + 1} \|_2^2 \\+ \langle \bm{c}_{k + 1} , \bm{x} - \bm{y}_{k + 1} \rangle + \frac{\mu}{2} \| \bm{x} - \bm{y}_{k + 1} \|_2^2,
\end{gathered}
\end{equation}
where we define the extended Lipschitz constant estimate $\bar{L}_{k + 1} \overset{\operatorname{def}}{=} L_{k + 1} + \mu_{\Psi}$ and the composite gradient $\bm{c}_{k + 1} \overset{\operatorname{def}}{=} (L_{k + 1} + \mu_{\Psi})( \bm{y}_{k + 1} - \bm{x}_{k + 1} )$ to simplify notation.
Applying \eqref{label_133}, \eqref{label_132} and substituting \eqref{label_130} with \eqref{label_131} in \eqref{label_129} gives
\begin{equation} \label{label_134}
\begin{gathered}
\bar{A}_{k + 1} \bar{\psi}_{k + 1}^* + \frac{\bar{\gamma}_{k + 1}}{2} \| \bm{x} - \bvec{v}_{k + 1} \|_2^2 = A_k \psi_k^* + \frac{\gamma_k}{2} \| \bm{x} - \bm{v}_{k} \|_2^2 + a_{k + 1} F(\bm{x}_{k + 1}) \\ + a_{k + 1} \left( \frac{1}{2 \bar{L}_{k + 1}} \| \bm{c}_{k + 1} \|_2^2
+ \langle \bm{c}_{k + 1} , \bm{x} - \bm{y}_{k + 1} \rangle + \frac{\mu}{2} \| \bm{x} - \bm{y}_{k + 1} \|_2^2 \right) .
\end{gathered}
\end{equation}
Differentiating \eqref{label_134} with respect to $\bm{x}$ we obtain
\begin{equation} \label{label_135}
\bar{\gamma}_{k + 1} (\bm{x} - \bvec{v}_{k + 1}) = \gamma_k (\bm{x} - \bm{v}_{k}) + a_{k + 1} \bm{c}_{k + 1} + \mu (\bm{x} - \bm{y}_{k + 1}) .
\end{equation}
Taking \eqref{label_135} with $\bm{x} = \bm{y}_{k + 1}$, applying the Euclidean norm squared and dividing both sides by $2 \bar{\gamma}_{k + 1}$ yields
\begin{equation} \label{label_136}
\frac{\bar{\gamma}_{k + 1}}{2} \|\bm{x} - \bvec{v}_{k + 1}\|_2^2 = \frac{\gamma_k^2}{2 \bar{\gamma}_{k + 1}} \| \bm{x} - \bm{v}_{k} \|_2^2 + \frac{a_{k + 1}^2}{2 \bar{\gamma}_{k + 1}} \| \bm{c}_{k + 1} \|_2^2 + \frac{a_{k + 1} \gamma_k}{\bar{\gamma}_{k + 1}} \langle \bm{c}_{k + 1}, \bm{x} - \bm{v}_{k} \rangle .
\end{equation}
Further, we combine the assumption $\psi_k^* \geq F(\bm{x}_k)$ with Lemma~\ref{label_121} applied at $\bm{y} = \bm{x}_k$ and $\bm{x} = \bm{y}_{k + 1}$ we get
\begin{equation} \label{label_137}
\psi_k^* \geq F(\bm{x}_k) \geq F(\bm{x}_{k + 1}) + \frac{1}{2 \bar{L}_{k + 1}} \| \bm{c}_{k + 1} \|_2^2 + \langle \bm{c}_{k + 1}, \bm{x}_k - \bm{y}_{k + 1} \rangle + \frac{\mu_{\Psi}}{2} \| \bm{x}_k - \bm{y}_{k + 1} \|_2^2.
\end{equation}
Taking \eqref{label_134} with $\bm{x} = \bm{y}_{k + 1}$, applying \eqref{label_136} and \eqref{label_137} and grouping terms together produces
\begin{equation} \label{label_138}
\begin{gathered}
\bar{A}_{k + 1} \bar{\psi}_{k + 1}^* \geq \bar{A}_{k + 1} F(\bm{x}_{k + 1}) + \left( \frac{\bar{A}_{k + 1}}{2 \bar{L}_{k + 1}} - \frac{a_{k + 1}^2}{2 \bar{\gamma}_{k + 1}} \right) \| \bm{c}_{k + 1} \|_2^2 \\
+ \left\langle \bm{c}_{k + 1}, A_k \bm{x}_k + \frac{a_{k + 1} \gamma_k}{\bar{\gamma}_{k + 1}} \bm{v}_k - \left( A_k + \frac{a_{k + 1} \gamma_k}{\bar{\gamma}_{k + 1}}\right) \bm{y}_{k + 1} \right\rangle \\
+ \left(\frac{\gamma_k}{2} - \frac{\gamma_k^2}{2 \bar{\gamma}_{k + 1}} \right) \| \bm{y}_{k + 1} - \bm{v}_k \|_2^2 + \frac{A_k \mu_{\Psi}}{2} \| \bm{x}_k - \bm{y}_{k + 1} \|_2^2 .
\end{gathered}
\end{equation}
The updates in \eqref{label_114} and \eqref{label_115}, respectively, imply that the second and third terms in the right-hand side of \eqref{label_138} are null whereas $\bar{\gamma}_{k + 1} > \gamma_k$ implies the fourth term is non-negative and the last term is obviously non-negative. The desired result follows directly.
\end{proof}

Based on Theorem~\ref{label_127}, using the same reasoning as in Subsubsection~\ref{label_098}, we arrive at \eqref{label_106}. Considering that the growth rate of the convergence guarantees is never lower than that of ACGM when the Lipschitz constant estimates are the same, we have that the Theorem~2 in \cite{ref_009} holds, which gives the following worst-case rate:
\begin{equation}
F(\bm{x}_k) - F(\bm{x}^*) \leq \min\left\{\frac{4}{(k + 1)^2}, (1 - \sqrt{q_u})^{k - 1}\right\}
\frac{L_u - \mu_f}{2} \| \bm{x}_0 - \bm{x}^* \|_2^2 \quad k \geq 1.
\end{equation}

\subsection{Quadratic functional growth}

Good linear convergence rates can also be obtained for more relaxed notions of strong convexity such as the quadratic functional growth (QFG) condition. However, the quadratic lower bounds defining this class of problems, namely those in \eqref{label_003}, do not incorporate local gradient-like information. For this reason, we cannot adopt the approach used in the previous subsection. Instead, we improve overall performance by employing \emph{restart strategies}.

The idea of restarting a first-order method has been introduced in the original formulation of the Fast Gradient Method in \cite{ref_012} to deal with strong convexity. Interestingly, this idea predates the use of lower bounds that take into account strong convexity, either in composite form \cite{ref_014} or as smooth quadratic functions \cite{ref_013}, also found in Algorithm~\ref{label_119}. The notion of quadratic functional growth in \cite{ref_011} was proposed from the beginning alongside a restarting procedure, albeit only applicable when $\Psi$ is an indicator function of the feasible set. This strategy relies either on the a priori knowledge of both $L_f$ and $\mu$ or, alternatively, of the value $F^*$. Many composite problems do not fall into either category. A very recent breakthrough work in \cite{ref_001} has managed to formulate a restart strategy that is able to estimate an unknown growth parameter. This adaptive restart endows any first-order method with a near-optimal rate on composite objectives possessing the QFG property. Due to the importance of this work, we adopt the notation used therein when formulating our restarting procedure. However, we note that even the strategy proposed in \cite{ref_001} assumes that $L_f$ is known (with no strong convexity assumption on $\Psi$) and places constraints on the minimum number of iterations performed by the first-order scheme thereby ignoring to a certain extent the progress made beyond the worst-case guarantees, which is central to our methods.

In this work, we consider both the case when $\mu$ is known and when $\mu$ is not known. Under both assumptions, we formulate a restart strategy that is amenable to any optimization scheme possessing sublinear worst-case convergence guarantees. Afterwards, we analyze the implementation and performance of these restart strategies when taking into account the particularities of our AGMM, in this case the strong convexity agnostic version in Algorithm~\ref{label_093}.

First we need to introduce some notation. A restart scheme can be considered a wrapper around an existing optimization method $\mathcal{R}$. We require $\mathcal{R}$ to maintain at every iteration $k$ a convergence guarantee $A_k$ satisfying
\begin{equation} \label{label_139}
F(\bm{x}_k) - F^* \leq \frac{1}{2 A_k} d^2(\bm{x}_0).
\end{equation}
The worst-case rate sequence $\{A_k\}_{k \geq 0}$ has to increase at every iteration ($A_{k + 1} > A_k$ for all $k \geq 0$), has to start at $A_0 = 0$ and has to become arbitrarily large given enough iterations $\left(\displaystyle \lim_{k \rightarrow \infty}A_k = \infty\right)$. Throughout our analysis, it is desirable to maintain a one-to-one correspondence between iterations and convergence guarantees. To this end we extend the sequence $\{A_k\}_{k \geq 0}$ to take real-valued arguments instead of discrete indexes $k$, thereby ensuring that the resulting real-valued function is smooth and invertible. Let $A : \mathbb{R}_+ \rightarrow \mathbb{R}_+ $, $A(0) = 0$ and $A(k) = A_k$ for every integer $k \geq 1$. The value at all other points can be obtained either by extending to all positive real numbers the analytical expression of $A_k$ in integer $k$, if available, either by interpolation or through a combination of the two approaches. The resulting function $A$ must be a strictly increasing and must satisfy $\displaystyle \lim_{x \rightarrow \infty}A(x) = \infty$. These properties render $A$ invertible and we denote its inverse by $A^{-1} : \mathbb{R}_+ \rightarrow \mathbb{R}_+ $.

Every instance of the restarted scheme $\mathcal{R}$ needs to make meaningful progress beyond the starting point $\bm{x}_0$, at least as good as one iteration of the Gradient Method, namely $F(\bm{x}_k) \leq F\left(T_{L(\bm{x_0})}(\bm{x}_0)\right)$, $k > 0$, given an estimate $L(x_0)$ satisfying \eqref{label_019} with $\bm{x} = \bm{x}_0$.

The wrapper scheme successively calls $\mathcal{R}$, feeding the output iterate of one $\mathcal{R}$ instance as the input to the next one. An iteration of the wrapper scheme thus needs the include the following update:
\begin{equation} \label{label_140}
(\bm{r}_{j + 1}, U_{j + 1}, n_{j + 1}) = \mathcal{R}(\bm{r}_j, \bar{U}_j, E_j),
\end{equation}
where $j$ is the wrapper iteration counter, taking the values of $0$ through $J - 1$, with $J > 0$. Note that the wrapper scheme can be easily made online by setting $J = +\infty$. At every wrapper iteration $j$, the method $\mathcal{R}$ takes as parameters the starting point $\bm{r}_j$ (which becomes $\bm{x}_0$ within $\mathcal{R}$), a threshold value of the convergence guarantee upon termination $\bar{U}_j$ and an optional additional termination criterion $E_j$. The condition $E_j$ is sufficient for the termination of $\mathcal{R}$, and except for the first wrapper iteration, of the wrapper scheme as well. For instance, $E_j$ can mark the depletion of a predetermined computation budget, as a certain \emph{global total} number of $\mathcal{R}$ iterations. The output of $\mathcal{R}$ must contain the last iterate, the final convergence guarantee and the number of expended $\mathcal{R}$ iterations, known to the wrapper scheme as $\bm{r}_{j + 1}$, $U_{j + 1}$ and $n_{j + 1}$, respectively. A template for the restarted scheme can be found in Algorithm~\ref{label_141}. The algorithmic state of $\mathcal{R}$ at iteration $k$, denoted by $\operatorname{State}_k$, can include the estimate function, auxiliary points, results of oracle calls (possibly including higher of order ones) and any other parameters $\mathcal{R}$ may maintain at runtime. The implementation of $\operatorname{\mathcal{R}\_initialization}$ and $\operatorname{\mathcal{R}\_iteration}$ define the method $\mathcal{R}$.

\begin{algorithm}[h!]
\caption{A template for the restarted scheme call \newline{} $(\bm{r}_{j + 1}, U_{j + 1}, n_{j + 1}) = \mathcal{R}(\bm{r}_j, \bar{U}_j, E_j)$}
\label{label_141}
\begin{algorithmic}[1]
\STATE $\bm{x}_0 = \bm{r}_j$
\STATE $A_0 = 0$
\STATE $k := 0$
\STATE $\operatorname{State}_0 = \operatorname{\mathcal{R}\_initialization}(\bm{x}_0, A_0)$
\REPEAT
\STATE $(x_{k + 1}, A_{k + 1}, \operatorname{State}_{k + 1}) = \operatorname{\mathcal{R}\_iteration}(x_k, A_k, \operatorname{State}_k)$
\STATE $k := k + 1$
\UNTIL $A_{k} \geq \bar{U}_j$ or $E_j$ holds
\STATE $\bm{r}_{j + 1} = \arg\max\left\{F(\bm{x}_{k}), F\left(T_{L(\bm{x}_0)}(\bm{x}_0)\right)\right\}$
\STATE $U_{j + 1} = A_{k}$
\STATE $n_{j + 1} = k$
\end{algorithmic}
\end{algorithm}

Given the quadratic functional growth assumption, the outputs of Algorithm~\ref{label_141} satisfy
\begin{equation} \label{label_142}
\begin{gathered}
F(\bm{r}_{j + 1}) - F^* \leq \frac{1}{2 U_{j + 1}} d^2(\bm{r}_j) \leq \frac{1}{\mu U_{j + 1}} (F(\bm{r}_j) - F^*) \\
\overset{\mbox{if}\ E_j = \mbox{False}}\leq \frac{1}{\mu \bar{U}_{j}}(F(\bm{r}_j) - F^*), \quad j \geq 0.
\end{gathered}
\end{equation}
Iterating \eqref{label_142} we obtain that
\begin{equation} \label{label_143}
F(\bm{r}_j) - F^* \leq \frac{1}{C_j} (F(\bm{r}_0) - F^*), \quad j \geq 1,
\end{equation}
where the wrapper scheme convergence guarantee $C_j$ is given by $C_j = \mu^j \prod_{i = 1}^{j} U_{i}$. The accuracy of our estimate $\bm{r}_j$ in terms of objective value can be upper bounded by a \emph{known} quantity (see \cite{ref_001}). Specifically, rearranging terms in \eqref{label_143} yields
\begin{equation} \label{label_144}
F(\bm{r}_j) - F^* \leq \frac{1}{C_j - 1} (F(\bm{r}_0) - F(\bm{r}_j)), \quad j \geq 1 .
\end{equation}
Note that our assumptions do not ensure that the starting point $\bvec{r}_0$ of the wrapper scheme is feasible. If we set $\bm{r}_0 = \bm{x}_0$, then \eqref{label_143} and \eqref{label_144} become meaningless if $F(\bm{x}_0)$ is unbounded. For simplicity, throughout this section we choose $\bm{r}_0 = \operatorname{prox}_{\tau_0}(\bm{x}_0)$, with a suitably chosen $\tau_0$ such as $1 / L_0$ for AGMM. The complexity of this additional step is by assumption negligible on the scale of the wrapper scheme.

While starting point feasibility is necessary for every instance of $\mathcal{R}$, successive instances need not be linked by a single point. A restart must ensure that \eqref{label_139} is maintained, which entails the reset of the estimate function, if employed. However, for AGMM in Algorithm~\ref{label_093} the bundle itself need not be emptied at restart. Among the many possibilities of adjusting the model during restarts we consider in this work the two extremes: a \emph{soft restart}, where the bundle is maintained in its entirety and a \emph{hard restart} where the bundle is completely emptied. For AGMM, resetting the estimate function at the end of the wrapper iteration $j$ amounts to setting $\bm{v}_0 = \bm{r}_{j + 1}$, with $A_0 = 0$. Hard restart entails a full reset of the estimate function, and consequently the assignment $\psi_0^* = F(\bm{v}_0)$, whereas a soft restart does not alter the estimate function optimal value.

Going back to the general form of $\mathcal{R}$, we first consider that the guarantee $A_k$ for $\mathcal{R}$ is well approximated by a term $c \cdot k^p$ with constants $c > 0$ and $p > 0$. In this case we have that $C_j = (\mu c)^j \left( \prod_{i = 1}^{j} n_{i} \right)^p$, $j \geq 1$. For a given global total number of $\mathcal{R}$ iterations $N_j = \sum_{i = 1}^{j} n_{i}$, the inequality of arithmetic and geometric means implies that the largest values of $C_j$ are obtained when all $n_i$, $i \inrange{j}$, border $\bar{n} = N_j / j$. Generalizing to the case when the scale of $A_k$ only needs to be sublinear and independent of the problem dimensionality $n$, we obtain the same result and thus conclude that it is near optimal to obtain after every call to $\mathcal{R}$ an improvement by a constant $D \in (0, 1)$, stated as
\begin{equation} \label{label_145}
F(\bm{r}_{j + 1}) - F^* \leq D (F(\bm{r}_j) - F^*) , \quad j \geq 0 .
\end{equation}

\subsubsection{Known growth parameter}

When the exact value of $\mu$ is available to the algorithm, we can easily determine $D$ and $\bar{n}$ to maximize the overall worst-case rate $C_j$. We thus need to maximize the expression $\left( \mu A\left(\frac{N}{j}\right) \right)^j$ for any given $N$. By changing the variable to $\bar{n} = N / j$, the task reduces to determining the value $\bar{n}^*$ that maximizes $\left( \mu A(\bar{n}) \right)^{\frac{1}{\bar{n}}}$ with the wrapper scheme parameters $\bar{U}_j$ and $D$ directly given by $\bar{U}_j = A(\bar{n}^*)$ for all $j \geq 0$ and $D = \frac{1}{\mu A(\bar{n}^*)}$.

When $A(x)$ is well approximated by $c x^p$ we can obtain simple closed form expressions for $\bar{n}$, $D$ and $j$, respectively, given by
\begin{equation} \label{label_146}
\bar{n}^* = e \cdot (\mu c)^{-\frac{1}{p}}, \quad D = \frac{1}{\mu c (\bar{n}^*)^p} = e^{-p} , \quad j = \frac{N}{\bar{n}^*} =\frac{1}{e}(\mu c)^{\frac{1}{p}} N.
\end{equation}
Under this setup we have that $C_j = C^{-j} = e^{\frac{p}{e} (\mu c)^\frac{1}{p} N}$ for every $N = j \bar{n}^*$, $j \geq 1$. While the results in \eqref{label_146} accurately describe the asymptotic behavior, we need $\bar{n}$ and $j$ to be integers. In practical implementations we need to restart after a worst-case $\tilde{n} = \lceil \bar{n}^* \rceil$ iterations of $\mathcal{R}$, where $\lceil x \rceil$ is the smallest integer no less than $x$. The estimate $\bm{x}_N$ obtained after $N$ iterations of $\mathcal{R}$ satisfies for all $N$ that are \emph{positive multiples} of $\tilde{n}$ the following:
\begin{equation} \label{label_147}
\begin{gathered}
F(\bm{x}_N) - F^* \leq e^{ -\frac{p N}{\tilde{n}} } (F(\bm{r}_0) - F^*) \leq \exp{ -\frac{p}{e} (\mu c)^\frac{1}{p} N } (F(\bm{r}_0) - F^*), \\
F(\bm{x}_N) - F^* \leq \frac{1}{\exp{\frac{p}{e} (\mu c)^\frac{1}{p} N } - 1} (F(\bm{r}_0) - F(\bm{x}_N)) .
\end{gathered}
\end{equation}
Note that the worst-case rate in \eqref{label_147} \emph{does not} generally hold for all values of $N$ not multiples of $\tilde{n}$ because we cannot guarantee that $A_k = A(k) \geq A(\tilde{n})^{\frac{k}{\tilde{n}}}$ with $k \in \{ 1, ..., \tilde{n} - 1 \}$ for all possible parameter choices. However, in most common applications, \eqref{label_147} can be shown to hold for every positive integer $N$.

Our AGMM also fits this case and by setting $\mathcal{R}$ to be Algorithm~\ref{label_093} we can formulate a restart strategy according to \eqref{label_146} using $c = \frac{1}{ 4L_u}$ and $p = 2$. AGMM employs the bundle to attain a rate that is in practice much better then the worst-case rate and instead we only use $D = e^{-2}$ and set the threshold convergence guarantees $\bar{U}_j$ at $\bar{U}_j = 1 / (\mu D) = e^2 / \mu$ for all $j \geq 0$. The worst-case results in \eqref{label_147} also hold, for both soft and hard restart strategies.

\subsubsection{Unknown growth parameter} \label{label_148}

We can see in \eqref{label_146} that when the worst-case guarantees are dominated by a monomial term,
the constant $D$ \emph{does not depend on the growth parameter $\mu$}. This suggests that an efficient restart scheme can be implemented without knowledge of $\mu$. However, the condition in \eqref{label_145} cannot verified at runtime because the value of $F^*$ is usually not known. The estimation procedure introduced in \cite{ref_001} brings \eqref{label_145} to a form where all objective function values are known. As we have seen in \eqref{label_144}, \eqref{label_145} can be refactored to become
\begin{equation} \label{label_149}
F(\bm{r}_{j + 1}) - F(\bm{x}) \leq F(\bm{r}_{j + 1}) - F^* \leq \frac{D}{1 - D} (F(\bm{r}_j) - F(\bm{r}_{j + 1})) , \ j \geq 0, \ \bm{x} \in \mathbb{R}^n .
\end{equation}
For the progress made in \eqref{label_149} to be meaningful, we need to have $0 < D < \frac{1}{2}$. Note that \cite{ref_001} enforces $D =1 / (e + 1)$ whilst in this work we determine the value of $D$ that maximizes the asymptotic worst-case rate or alternatively, as we shall elaborate in the sequel, maximizes the iteration efficiency.

\paragraph{A simple adaptive restart scheme}
First, we formulate a restart strategy that can utilize any algorithm $\mathcal{R}$ adhering to the template in Algorithm~\ref{label_141}. The first call to $R$ needs to differ from the subsequent ones because we do not have an adequate threshold value $\bar{U}_0$ at hand.
An estimate can be obtained by choosing a suitable criterion $E_0$. A good choice was proposed in \cite{ref_001} also based on \eqref{label_149}. We generalize it to our setup as
\begin{equation} \label{label_150}
F(\bm{x}_m) - F(\bm{x}_k) \leq \frac{D}{1 - D} ( F(\bm{x}_0) - F(\bm{x}_m) ),
\end{equation}
where $m = \left\lceil A^{-1}\left( \frac{A(k)}{s} \right) \right\rceil$. For \eqref{label_150} to be a valid stopping criterion, we must first prove that it terminates within a reasonable amount of $\mathcal{R}$ iterations.

We define the maximal convergence guarantee as $U_{\operatorname{max}} \overset{\operatorname{def}}{=} \frac{1}{\mu D}$. When the iteration counter $k$ of the $\mathcal{R}$ instance called at $\bm{r}_j$ reaches a level where $A_k = A(k) \geq U_{\operatorname{max}}$, using the same reasoning as in the derivation of \eqref{label_149}, we obtain that within $\mathcal{R}$ for all $\bm{x} \in \mathbb{R}^n$ the following holds:
\begin{equation} \label{label_151}
\begin{gathered}
F(\bm{x}_k) - F(\bm{x}) \leq \frac{1}{\mu A(k) - 1} ( F(\bm{x}_0) - F(\bm{x}_k) ) \\ \leq \frac{1}{\mu U_{\operatorname{max}} - 1} ( F(\bm{x}_0) - F(\bm{x}_k) ) = \frac{D}{1 - D} ( F(\bm{x}_0) - F(\bm{x}_k) ) .
\end{gathered}
\end{equation} Based on this observation, we can formulate a limit on $n_1$ as follows.

\begin{lemma} \label{label_152}
The instance $\mathcal{R}(r_0, +\infty, E_0)$ terminates within $\bar{n}_1 \overset{\operatorname{def}}{=} $
$\left\lceil A^{-1}\left( \frac{s}{\mu D} \right) \right\rceil$ iterations.
\end{lemma}
\begin{proof}
If the iteration counter $k$ in $\mathcal{R}(r_0, +\infty, E_0)$ reaches $\bar{n}_1$, the condition in \eqref{label_150} is evaluated at $k = \bar{n}_1$ and $m = \bar{m}_1$, where $\bar{m}_1 \overset{\operatorname{def}}{=} \left\lceil A^{-1}\left( \frac{A(\bar{n}_1)}{s} \right) \right\rceil$. Because $A$ is monotonically increasing, we have that
\begin{equation}
\begin{gathered}
A(\bar{m}_1) = A\left( \left\lceil A^{-1}\left( \frac{A(\bar{n}_1)}{s} \right) \right\rceil \right) \geq A\left( A^{-1}\left( \frac{A(\bar{n}_1)}{s} \right) \right) = \frac{A(\bar{n}_1)}{s} \\
= \frac{A\left(\left\lceil A^{-1}\left( \frac{s}{\mu D} \right) \right\rceil\right)}{s} \geq \frac{A\left(A^{-1}\left( \frac{s}{\mu D} \right)\right)}{s} = \frac{\frac{s}{\mu D}}{s} = \frac{1}{\mu D} = U_{\operatorname{max}}.
\end{gathered}
\end{equation}
Consequently, taking into account \eqref{label_151}, we have that $E_0$ in \eqref{label_150} is satisfied and $\mathcal{R}(r_0, +\infty, E_0)$ terminates.
\end{proof}
Having proven the validity of the starting conditions, we are now ready to formulate the restart strategy.

At every wrapper iteration $j \geq 1$ \emph{when $E_j$ is not triggered}, we distinguish two cases. If the geometric decrease condition
\begin{equation} \label{label_153}
F(\bm{r}_{j}) - F(\bm{r}_{j + 1}) \leq \frac{D}{1 - D} ( F(\bm{r}_{j - 1}) - F(\bm{r}_{j}) ),
\end{equation}
holds, then we proceed to the next instance of $\mathcal{R}$ without need to adjust any parameters to attain a linear rate. If the condition in \eqref{label_153} is not satisfied, then an adjustment of the wrapper scheme occurs, which involves the increase of the threshold value $\bar{U}_j$. These increases eventually stop when $\bar{U}_j \geq U_{\operatorname{max}}$, whereby \eqref{label_151} implies \eqref{label_153}. We thus have $\bar{U}_j < s U_{\operatorname{max}}$ and at every wrapper step $j$, we automatically obtain an estimate of $\mu$ in the form
\begin{equation}
\hat{\mu}_j \overset{\operatorname{def}}{=} \frac{1}{D \bar{U}_j} > \frac{1}{D (sU_{\operatorname{max}})} = \frac{\mu}{s}, \quad j \geq 0.
\end{equation}
The condition $E_j$ ends the restarts if triggered for any $j\geq 1$. The entire wrapper scheme is listed in Algorithm~\ref{label_154}.

\begin{algorithm}[h!]
\caption{A restart strategy for the sublinear scheme $\mathcal{R}$ that attains a near-optimal linear convergence rate}
\label{label_154}
\begin{algorithmic}[1]
\STATE \bd{Input:} $\bvec{r}_0 \in \mathbb{R}^n$, $0 < D < \frac{1}{2}$, $s > 1$, $J \inrange{+\infty}$
\STATE $\bm{r}_0 = \operatorname{prox}_{\tau_0}(\bvec{r}_0)$
\STATE $(\bm{r}_1, U_1, n_1) = \mathcal{R}(\bm{r}_0, +\infty, E_0)$ \hspace{1em} \#Reference step \label{label_154_1}
\STATE $\bar{U}_1 := U_1$
\FOR{$j = 1, \ldots{}, J - 1$}
\STATE $(\bm{r}_{j + 1}, U_{j + 1}, n_{j + 1}) = \mathcal{R}(\bm{r}_j, \bar{U}_j, E_j)$
\IF {$E_j$ holds}
\STATE \bd{break}
\ENDIF
\IF {$F(\bm{r}_{j}) - F(\bm{r}_{j + 1}) > \frac{D}{1 - D} (F(\bm{r}_{j - 1}) - F(\bm{r}_{j}))$}
\STATE $\bar{U}_{j + 1} := s \cdot \bar{U}_j$ \hspace{1em} \# Adjustment step
\ELSE
\STATE $\bar{U}_{j + 1} := \bar{U}_j$ \hspace{1em} \# Normal step
\ENDIF
\ENDFOR
\end{algorithmic}
\end{algorithm}
\paragraph{Complexity analysis}
To analyze the worst-case complexity of Algorithm~\ref{label_154}, we need to introduce a fundamental property of composite objectives endowed with the quadratic functional growth property in \eqref{label_003}. It is similar to but not the same as the Polyak-\L{}ojasiewicz inequality, which does not hold in this context.
\begin{lemma} \label{label_156}
At every feasible point $\bm{x}$, the squared norm of \emph{any} subgradient taken with coefficient $2/\mu$ constitutes an upper bound on the optimality gap, namely
\begin{equation}
F(\bm{x}) - F^* \leq \frac{2}{\mu} \| \bm{\xi} \|_2^2, \quad \bm{x} \in X, \quad \bm{\xi} \in \delta F(\bm{x}).
\end{equation}
\end{lemma}
\begin{proof}
When $\bm{x} \in X^*$ we have for all $\bm{\xi} \in \delta F(\bm{x})$ that $F(\bm{x}) - F^* = 0 \leq \frac{2}{\mu} \| \bm{\xi} \|_2^2$.
When $\bm{x} \in X \setminus X^*$, from the convexity of $F$ at $\bm{x}$ we have
\begin{equation}
F^* - F(\bm{x)} \geq \langle \bm{\xi}, \bm{o}(\bm{x}) - \bm{x} \rangle .
\end{equation}
Applying the Cauchy–Bunyakovsky–Schwarz inequality we obtain
\begin{equation} \label{label_157}
\| \bm{\xi} \|_2 d(\bm{x}) = \| \bm{\xi} \|_2 \| \bm{x} - \bm{o}(\bm{x}) \|_2 \geq \langle \bm{\xi}, \bm{x} - \bm{o}(\bm{x}))\rangle \geq F(\bm{x}) - F^*.
\end{equation}
Plugging in the QFG definition \eqref{label_003} in \eqref{label_157} yields
\begin{equation}
\frac{2}{\mu} \| \bm{\xi} \|_2^2 \geq \frac{\left( F(\bm{x}) - F^* \right)^2}{\frac{\mu}{2}d^2(\bm{x})} \geq \frac{\left( F(\bm{x}) - F^* \right)^2}{F(\bm{x}) - F^*} = F(\bm{x}) - F^*.
\end{equation}
\end{proof}
The above property enables us to transform localized progress into a global optimality certificate. Note that the following result applies even to one iteration of the Gradient Method and it is thus applicable to instances of $\mathcal{R}$ that were terminated while the convergence guarantee was below the threshold.
\begin{proposition} \label{label_158}
After every wrapper iteration \eqref{label_140} employing a restarted scheme that adheres to the template in Algorithm~\ref{label_141}, we have
\begin{equation}
F(\bm{r}_{j + 1}) - F^* \leq \frac{4 L(\bm{r}_j)}{\mu}\left(1 + \frac{L_f}{L(\bm{r}_j)}\right)^2 \left( F(\bm{r}_j) - F(\bm{r}_{j + 1})\right), \quad j \geq 0,
\end{equation}
even when the stopping criterion $E_j$ triggers an early termination of $\mathcal{R}$.
\end{proposition}
\begin{proof}
Let $s_L(\bm{x}) \overset{\operatorname{def}}{=} g_L(\bm{x}) - f'(\bm{x}) + f'(T_L(\bm{x}))$, $\bm{x} \in \mathbb{R}^n$, $L > 0$.
It is a well-known result~\cite{ref_014} that for all $\bm{x} \in \mathbb{R}^n$, $L > 0$ we have
\begin{equation}
\begin{gathered} \label{label_159}
\| s_L(\bm{x}) \|_2 = \| g_L(\bm{x}) - f'(\bm{x}) + f'(T_L(\bm{x})) \|_2 \leq \| g_L(\bm{x}) \|_2 + \| f'(\bm{x}) - f'(T_L(\bm{x})) \|_2 \\ \leq L \| \bm{x} - T_L(\bm{x}) \|_2 + L_f \| \bm{x} - T_L(\bm{x}) \|_2 = \left( 1 + \frac{L_f}{L} \right) \| g_L(\bm{x}) \|_2.
\end{gathered}
\end{equation}
From \eqref{label_017} we have that $s_L(\bm{x}) \in \delta F(T_L(\bm{x}))$ and we can apply Lemma~\ref{label_156} at $T_{L(\bm{r}_j)}(\bm{r}_j)$. Further using the properties of the scheme $\mathcal{R}$ and our previous result in \eqref{label_159} we obtain
\begin{equation} \label{label_160}
\begin{gathered}
F(\bm{r}_{j + 1}) - F^* \leq F\left(T_{L(\bm{r}_j)}(\bm{r}_j)\right) - F^* \leq \frac{2}{\mu} \| s_{L(\bm{r}_j)}(\bm{r_j}) \|_2^2 \\
\overset{\eqref{label_159}}{\leq} \frac{2}{\mu} \left(1 + \frac{L_f}{L(\bm{r}_j)}\right)^2 \left\| g_{L(\bm{r}_j)} (\bm{r}_j)\right\|^2_2.
\end{gathered}
\end{equation}
The local upper bound property in \eqref{label_019} implies the descent rule~\cite{ref_014} which, together with our assumptions, yields
\begin{equation} \label{label_161}
\frac{1}{2 L(\bm{r}_j)} \| g_{L(\bm{r}_j)}(\bm{r}_j) \|_2^2 \leq F(r_j) - F\left(T_{L(\bm{r}_j)}(\bm{r}_j)\right) \leq F(\bm{r}_j) - F(\bm{r}_{j + 1}) .
\end{equation}
Finally, applying \eqref{label_161} to \eqref{label_160} gives the desired result.
\end{proof}
Now we have all the theoretical tools necessary for our analysis. To proceed, we need to distinguish between the 4 types of wrapper iterations. The first is the \emph{reference} iteration, it occurs exactly once and entails $N_r \leq \bar{n}_1$ iterations of $\mathcal{R}$. As the name implies, its output is used as a reference point in our final results. The heuristic ensures that the progress made in this wrapper iteration is substantial. It also tries to provide an accurate estimate of $\mu$ and thus, as we shall show in the sequel, lowers the bound on the total number of a second type of wrapper iterations, the \emph{adjustment} iterations.

Adjustments occur when the condition in \eqref{label_153} is not satisfied. The $\mathcal{R}$ iterations improve our estimate of the optimal point but cannot be counted towards the asymptotic linear rate. However, we easily observe that the number of such adjustments is bounded independently of the total number wrapper iterations $J$, even when it is infinite. As we have seen, $\bar{U}_j$ can only be increased until it exceeds $U_{\operatorname{max}}$, at which point \eqref{label_153} must hold. Consequently, the total number of adjustments satisfies
\begin{equation} \label{label_162}
b \leq b_{\operatorname{max}} \overset{\operatorname{def}}{=} \left\lceil log_s\left(\frac{U_{\operatorname{max}}}{U_1}\right)\right\rceil = \lceil -log_s(\mu D U_1)\rceil < -log_s(\mu D U_1) + 1 .
\end{equation}
We can thus count the progress made in the $i$th adjustment, indexed as $j_i$, as
\begin{equation} \label{label_163}
F(\bm{r}_{j_i}) - F(\bm{r}_{{j_i} + 1}) \leq \frac{1}{\mu U_1 s^{i - 1} -1} \left(F(\bm{r}_{{j_i} - 1}) - F(\bm{r}_{j_i})\right), \quad i \inrange{b} .
\end{equation}
Moreover, the total number of $\mathcal{R}$ iterations expended during adjustments, denoted by $N_b$, is also bounded by $N_b \leq \displaystyle \sum_{i = 0}^{b - 1} A^{-1}(s^i U_1)$ .
Note that unlike the results obtained in \cite{ref_001}, the adjustment overhead obtained does not depend on a target accuracy, meaning that using our approach one can derive a worst-case rate for an \emph{online} scheme as well.

The third type of wrapper iterations is the dominant one, which we designate as the normal iterations. The progress here is clearly measured as
\begin{equation} \label{label_164}
F(\bm{r}_j) - F(\bm{r}_{j + 1}) \leq \frac{D}{1 - D} \left(F(\bm{r}_{j - 1}) - F(\bm{r}_j)\right), \quad j \geq 1, \quad j \notin \{j_i \mid i \inrange{b}\} .
\end{equation}

Finally, if $E_j$ holds at a given wrapper iteration $j \geq 1$, then the entire wrapper scheme in Algorithm~\ref{label_154} terminates. It occurs at most once and if it does, it always succeeds all other wrapper iterations. In our analysis we can always assume that the termination step occurs by simply ignoring the progress made in the last wrapper iteration beyond one step of the Gradient Method, regardless of whether $E_j$ is triggered or not. Therefore, we consider the progress made here to be governed by Proposition~\ref{label_158} and denote the number of $\mathcal{R}$ iterations expended here by $N_t$.

To obtain $\bm{r}_{j + 1}$ we thus need to perform one reference step, $b \leq b_{\operatorname{max}}$ adjustment steps, $j - b - 1$ normal steps and one termination step. Putting together all the above assumptions as well as the results in \eqref{label_163}, \eqref{label_164} and Proposition~\ref{label_158}, we obtain
\begin{equation}
F(\bm{r}_{j + 1}) - F^* \leq C_b \left(\frac{D}{1 - D} \right)^{j} \left( F(\bm{r}_0) - F(\bm{r}_1) \right), \quad j \geq 0,
\end{equation}
where
\begin{equation} \label{label_165}
C_b \overset{\operatorname{def}}{=} \frac{4 L(\bm{r}_j)}{\mu} \left( 1 + \frac{L_f}{L(\bm{r}_j)} \right)^2 \left(\frac{1 - D}{D} \right)^{b + 1} \prod_{i = 1}^{b} \frac{1}{\mu U_1 s^{i - 1} - 1} .
\end{equation}

In the worst-case, the number of $\mathcal{R}$ iterations at every \emph{normal} wrapper iteration is $\tilde{n} = \lceil A^{-1}\left(\max\{s U_{\operatorname{max}}, U_1\}\right) \rceil$. This also holds for the termination iteration: $N_t \leq \tilde{n}$.
When $A(x)$ is well approximated by $c\cdot x^p$ and the number of $\mathcal{R}$ iterations is large enough to ignore the effects of integer rounding, we can simplify the expression of $\tilde{n}$ as $\tilde{n} = \left( \frac{s}{\mu D c} \right)^{\frac{1}{p}}$. Taking the total number of iterations to be in the worst case $N = N_r + N_b + (j - b - 1) \tilde{n} + N_t = N_o + j \tilde{n}$, with the overhead
\begin{equation} \label{label_166}
N_o \overset{\operatorname{def}}{=} N_r + N_b + N_t - (b + 1) \tilde{n},
\end{equation}
the problem of maximizing the worst-case \emph{asymptotic linear rate} becomes
\begin{equation} \label{label_167}
\begin{gathered}
\arg\min_{D \in \left(0, \frac{1}{2}\right)} \displaystyle \lim_{N \rightarrow \infty} \frac{\ln{\left(C_b \left(\frac{D}{1 - D} \right)^{j}\right)}}{N} = \arg\min_{D \in \left(0, \frac{1}{2}\right)} \displaystyle \lim_{N \rightarrow \infty} \frac{\frac{N - N_o}{\tilde{n}} \ln{\frac{D}{1 - D}} + \ln{C_b} }{N} \\
= \arg\min_{D \in \left(0, \frac{1}{2}\right)} \frac{1}{\tilde{n}}\ln{\frac{D}{1 - D}}
= \arg\min_{D \in \left(0, \frac{1}{2}\right)} D^\frac{1}{p} \ln\left(\frac{D}{1 - D}\right) = \arg\max_{D \in \left(0, \frac{1}{2}\right)} \eta(D, p),
\end{gathered}
\end{equation}
where we define the iteration efficiency $\eta(D, p)$ as
\begin{equation} \label{label_168}
\eta(D, p) \overset{\operatorname{def}}{=} -\frac{e}{p} D^\frac{1}{p} \ln\left(\frac{D}{1 - D}\right) , \quad D \in \left(0, \frac{1}{2} \right), \quad p \geq 1.
\end{equation}
According to our assumptions, we can use the notation $\bm{x}_N = \bm{r}_{j + 1}$. The iteration efficiency allows us to express the worst-case convergence rate in a simple form, namely
\begin{equation}
F(\bm{x}_N) - F^* \leq C_b \exp{-\eta(D, p) \frac{p}{e} \left(\frac{\mu c}{s}\right)^\frac{1}{p} (N - N_o)} \left( F(\bm{r}_0) - F(\bm{r}_1) \right),
\end{equation}
with $C_b$, $\eta(D, p)$ and $N_o$ respectively given by \eqref{label_165}, \eqref{label_168} and \eqref{label_166}. Thus, the ratio between the asymptotic rate for an unknown QFG parameter is $\eta(D, p) s^{-\frac{1}{p}}$ that of the known parameter case in \eqref{label_147}.

Our simulations suggest that at least on most problems the worst-case behavior entails that all backtracks occur during the first wrapper iterations, with the geometric decrease rate being attained only after the threshold exceeds $U_{max}$. This possibility leads to an optimal value of $D = e^{-p}$, the same is in the known growth parameter case (see \eqref{label_146} and \eqref{label_147}). It is worth investigating the relationship between the experimental worst-case rate and our analytical results above. To this end, we list in Table~\ref{label_169} for $p$ in the range $[1,5]$ the iteration efficiency when $D$ takes the values $\frac{1}{e + 1}$ as in \cite{ref_001}, $e^{-p}$ as seen in our experiments and $D^*$ as the numerical solution to \eqref{label_167}, respectively. The case when $p = 2$, corresponding to our AGMM, is marked in bold font.

\begin{table}[H]
\centering
\small
\caption{The iteration efficiency for $p$ in the range $[1,5]$ when considering various values of $D$}
\label{label_169}
\begin{tabular}{ccccccc} \toprule
p & $\frac{1}{e + 1}$ & $e^{-p}$ & $D^*$ & $\eta(\frac{1}{e + 1}, p)$ & $\eta(e^{-p}, p)$ & $\eta(D^*, p)$ \\ \midrule
1.0 & 0.2689414 & 0.3678794 & 0.2178117 & 0.7311 & 0.5413 & 0.7569 \\
1.5 & 0.2689414 & 0.2231302 & 0.1469828 & 0.7550 & 0.8317 & 0.8875 \\
\bd{2.0} & \bd{0.2689414} & \bd{0.1353353} & \bd{0.0981709} & \bd{0.7048} & \bd{0.9273} & \bd{0.9444} \\
2.5 & 0.2689414 & 0.0820850 & 0.0646055 & 0.6430 & 0.9657 & 0.9714 \\
3.0 & 0.2689414 & 0.0497871 & 0.0418457 & 0.5849 & 0.9830 & 0.9849 \\
3.5 & 0.2689414 & 0.0301974 & 0.0267004 & 0.5337 & 0.9912 & 0.9919 \\
4.0 & 0.2689414 & 0.0183156 & 0.0168168 & 0.4894 & 0.9954 & 0.9956 \\
5.0 & 0.2689414 & 0.0067379 & 0.0064795 & 0.4181 & 0.9986 & 0.9987 \\ \bottomrule
\end{tabular}
\end{table}

\subsection{Model overhead} \label{label_076_overhead}

With respect to per-iteration complexity, all accelerated algorithms introduced in this work can be assimilated into AGMM. One reason is that both Algorithms \ref{label_093} and \ref{label_119} have iterations that are identical in terms of computational load. Moreover, the restart strategy in Algorithm~\ref{label_154} applied to Algorithm~\ref{label_093} can be reformulated as Algorithm~\ref{label_093} where at the end of each iteration the restart condition in \eqref{label_153} is additionally evaluated. With adequate caching of past function values, which are scalars and thus not subject to the memory limitation of the bundle, there is no computational complexity difference between the initial run in line \ref{label_154_1} of Algorithm~\ref{label_154} and the subsequent ones.

Following the case of Algorithm~\ref{label_030}, the model overhead for AGMM, in all aforementioned forms, is given by the costs pertaining to the inner problem, followed by middle scheme which iterates a few values of the convergence guarantee $A_{k + 1}$, this time increasing while fit. The inner problem setup is also dominated by the update of $\bm{Q}_{k + 1}$, amounting $\mathcal{O}(p_{k+ 1} n)$ FLOPs. Unlike in GMM, the $\bm{C}_{k + 1}$ update is of negligible complexity, as it can be obtained from $\bm{C}_{k}$ with $\mathcal{O}(n)$ cost. Moreover, for AGMM the middle method does not perform any calls to the oracle functions and the complexity of each step in the middle and inner methods do not depend on the outer problem dimensionality $n$. Actually, we have noticed that a small number $N_N$ of Newton iterations captures almost all performance improvements. Consequently, we consider the middle method, and for that matter the entire model overhead, to have negligible complexity when compared to that of a single oracle call.

\section{Simulations} \label{label_171}

To test the effectiveness of our methods, we have chosen the same collection of composite problems we had used to benchmark the generalized form of the Accelerated Composite Gradient Method in \cite{ref_009}. The five problems come from the areas of statistics, inverse problems and machine learning. Specifically, these are: least absolute shrinkage and selection operator (LASSO)~\cite{ref_020}, non-negative least squares (NNLS), $l_1$-regularized logistic regression (L1LR), ridge regression (RR) and elastic net (EN). The last two are strongly convex (in a strict sense) with known strong convexity parameter. To maximize the performance of the competing methods and to support the modification made to Algorithm~\ref{label_030} in Subsubsection~\ref{label_051}, we transfer all known strong convexity to the regularizer.

The objective function structure for each of the five problems is listed in Table~\ref{label_172}. To simplify notation, we introduce the sum softplus function $\mathcal{I}(\bm{x})$, the element-wise logistic function $\mathcal{L}(\bm{x})$, and the shrinkage operator $\mathcal{T}_{\tau}(\bm{x})$, respectively given for all $\bm{x} \in \mathbb{R}^n$ and $\tau > 0$ by
\begin{gather}
\mathcal{I}(\bm{x}) = \displaystyle \sum_{i = 1}^{m} \log(1 + e^{\bm{x}_i}), \\
\mathcal{L}(\bm{x})_i = \frac{1}{1 + e^{-\bm{x}_i}}, \quad i \inrange{m}, \\
\mathcal{T}_{\tau}(\bm{x})_j = (|\bm{x}_j| - \tau)_{+} \sgn(\bm{x}_j), \quad j \inrange{n}.
\end{gather}

\begin{table*}[h!]
\centering
\small
\caption{Oracle functions of the five test problems}
\label{label_172}
\begin{tabular}{lllll} \toprule
& $f(\bm{x})$ & $\Psi(\bm{x})$ & $\grad{f}(\bm{x})$ & $\operatorname{prox}_{\tau \Psi}(\bm{x})$ \\ \midrule
LASSO & $\frac{1}{2}\|\bm{A} \bm{x} - \bm{b} \|_2^2$ & $\lambda_1 \|\bm{x}\|_1$ & $\bm{A}^T(\bm{A}\bm{x} - \bm{b})$ & $\mathcal{T}_{\tau \lambda_1} (\bm{x})$ \\
NNLS & $\frac{1}{2}\|\bm{A} \bm{x} - \bm{b} \|_2^2$ & $\sigma_{\mathbb{R}_{+}^n}(\bm{x})$ & $\bm{A}^T(\bm{A}\bm{x} - \bm{b})$ & $(\bm{x})_{+}$ \\
L1LR & $\mathcal{I}(\bm{A}\bm{x}) - \bm{y}^T \bm{A} \bm{x}$ & $\lambda_1 \|\bm{x}\|_1$ & $\bm{A}^T(\mathcal{L}(\bm{A}\bm{x}) - \bm{y})$ & $\mathcal{T}_{\tau \lambda_1} (\bm{x})$ \\
RR & $\frac{1}{2}\|\bm{A} \bm{x} - \bm{b} \|_2^2$ & $\frac{\lambda_2}{2} \|\bm{x}\|_2^2$ & $\bm{A}^T(\bm{A}\bm{x} - \bm{b})$ & $\frac{1}{1 + \tau \lambda_2} \bm{x}$ \\
EN & $\frac{1}{2}\|\bm{A} \bm{x} - \bm{b} \|_2^2$ & $\lambda_1 \|\bm{x}\|_1 + \frac{\lambda_2}{2} \|\bm{x}\|_2^2$ & $\bm{A}^T(\bm{A}\bm{x} - \bm{b})$ & $\frac{1}{1 + \tau \lambda_2} \mathcal{T}_{\tau \lambda_1} (\bm{x})$ \\ \bottomrule
\end{tabular}
\end{table*}

The instances of the LASSO, L1LR and RR problems are completely identical to the ones described in \cite{ref_009} but for NNLS and EN it was necessary to alter the parameters. In its original form, NNLS was solved up to machine precision by a well configured gradient method with memory in under $30$ iterations while for EN this was accomplished in under $60$. Such values do not allow the construction of large bundles and leave little room for restarts making it impossible to showcase the advantage of these models and techniques. The parameters for these two problems were altered to ensure that even the fastest method tested would require at least around $100$ iterations converge to a high level of accuracy. For the sake of completeness, we describe in the following the specifics of all five problems. All random values were drawn from independent and identically distributed variables, unless otherwise stated, with $\mathcal{N}(0, 1)$ denoting the standard normal distribution. See also \cite{ref_009} for more details on how the problem instances LASSO, L1LR and RR were generated.

For LASSO, we used a $m = 500$ by $n = 500$ matrix $\bm{A}$ with entries sampled from $\mathcal{N}(0, 1)$, a vector $\bm{b}$ with entries drawn from $\mathcal{N}(0, 9)$ and a regularization parameter $\lambda_1 = 4$. The starting point $\bm{x}_0 \in \mathbb{R}^n$ has all entries sampled from $\mathcal{N}(0, 1)$. In NNLS, $\bm{A}$ is an $m = 1000$ by $n = 1000$ sparse matrix with $1\%$ of entries non-zero. The locations of the non-zero entries were chosen uniformly at random and the entries themselves were drawn from $\mathcal{N}(0, 1)$. Both $\bm{b}$ and $\bm{x}_0$ are dense with entries also drawn from $\mathcal{N}(0, 1)$. For L1LR, $\bm{A}$ is $m = 200 \times n = 1000$ with entries sampled from the standard normal distribution and $\lambda_1 = 5$. The starting point $\bm{x}_0$ has exactly $10$ non-zero entries at locations selected uniformly at random, with each entry drawn from $\mathcal{N}(0, 225)$. The binary labels $\bm{y}_i$ were chosen according to $\mathbb{P}(\bm{Y}_i = 1) = \mathcal{L}(\bm{A} \bm {x})_i$. In RR, we have $m = 500 \times n = 500$. The entries of $\bm{A}$, $\bm{b}$ and $\bm{x}_0$ are sampled from $\mathcal{N}(0, 1)$, $\mathcal{N}(0, 25)$, and $\mathcal{N}(0, 1)$, respectively. We have $\lambda_2 = 10^{-3} (\sigma_{max}(\bm{A}))^2$, with $\sigma_{max}(\bm{A})$ denoting the largest singular value of $\bm{A}$. For EN, $\bm{A}$ has $m = 1000 \times n = 500$ entries sampled from $\mathcal{N}(0, 1)$, $\lambda_1 = 1.5 \sqrt{2 \log(n)}$ (as recommended by \cite{ref_010}) and $\lambda_2 = 10^{-3} (\sigma_{max}(\bm{A}))^2$.

The global Lipschitz constant for each problem with a quadratic smooth part is given by $L_f = (\sigma_{max}(\bm{A}))^2$ whereas for L1LR we have $L_f = \frac{1}{4}(\sigma_{max}(\bm{A}))^2$. The strong convexity parameter in RR and EN is $\mu = \mu_{\Psi} = \lambda_2$ which yields an inverse condition number $q \overset{\operatorname{def}}{=} \mu / (L_f + \mu_{\Psi}) = 1 / 1001$.

To determine an approximate optimal value for each problem along with one approximate optimal point, we have employed AMGS, a variant of Nesterov's Fast Gradient Method introduced in \cite{ref_014}. We use the term Accelerated Multistep Gradient Scheme (AMGS) to distinguish between the original formulation of the Fast Gradient Method and this variant. AMGS was chosen despite its expensive iterations (which require at least two proximal gradient steps each) for its exceptional numerical stability owing to its damped relaxation condition based line-search procedure~\cite{ref_014}. For all problems, we have obtained an optimal point estimate $\hat{\bm{x}}^*$ by running this algorithm for $10000$ iterations with $L_0 = L_f$, $\gamma_u = 2$ and $\gamma_d = 10 / 9$ (as suggested in \cite{ref_003}). We consider that convergence is attained for a certain $\bm{x}_k$ when the relative error $\epsilon_k \overset{\operatorname{def}}{=} \frac{F(\bm{x}_k) - F(\hat{\bm{x}}^*)}{F(\bm{x}_0) - F(\hat{\bm{x}}^*)}$ goes below $\epsilon = 10^{-9}$. In all problems $\bm{x}_0$ is feasible and there is no need to use the proximal map to project it onto the feasible set.

In a preliminary stage, we have tested only the methods introduced in this work, with extensions meant to fill some gaps in the parameter range. The first method in the benchmark is GMM, as listed in Algorithm~\ref{label_030}. We distinguish between GMM-C employing the cyclic replacement strategy (CRS), whereby the oldest bundle entries are removed to make way for the new~\cite{ref_018}, and GMM-M relying on the max-norm replacement strategy (MRS) that removes the entries with the largest composite gradient norm. On the non-strongly convex problems LASSO, NNLS and L1LR we test AGMM in the form given by Algorithm~\ref{label_093} whereas on the strongly-convex problems RR and EN we consider the generalized form in Algorithm~\ref{label_119}. Behavior varies according to the replacement strategy and we likewise consider AGMM-C with CRS and AGMM-M with MRS. Indeed, Algorithm~\ref{label_093} is a particular case of Algorithm~\ref{label_119} and there is no apparent need to explicitly mention the former form separately. However, Algorithm~\ref{label_093} can also take advantage of strong convexity and of the more general QFG by resorting to the restart strategy described in Algorithm~\ref{label_154}. In this case we further distinguish between soft and hard restart modes, and thus have ASRC (Restarting AGMM by having Algorithm~\ref{label_154} call Algorithm~\ref{label_093} with soft restart and CRS as the method $\mathcal{R}$), ASRM (soft restart, MRS) as well as AHRC (hard restart, CRS) and AHRM (hard restart, MRS). In our experiments we found that hard restarting produces very similar, albeit slightly worse, results than soft restarting. Due to space limitations, we do not include the results involving hard restart.

We do however implicitly test against state-of-the-art methods by using the following observations. When the bundle size is one, GMM becomes the Gradient Method (also known as Forward-Backward method in the context of composite problems). Neither variant of AGMM is defined for $m = 1$ but, for consistency, we adopt the convention that AGMM with $m = 1$ actually denotes ACGM~\cite{ref_008,ref_009}. Note that for $m = 2$, AGMM is a generalization and improvement over the Adaptive-Accelerated (AA) Method introduced in \cite{ref_019}, another state-of-the-art method.

A series of tables explains the effectiveness of the six above-mentioned methods as a function of bundle size. We consider bundle sizes given by a power of $2$ within the range $1$ to $128$. The performance measured in iterations until convergence is listed in Table~\ref{label_173} for LASSO, Table~\ref{label_174} for NNLS, Table~\ref{label_175} for L1LR, Table~\ref{label_176} for RR and Table~\ref{label_177} for EN. For each method, the convergence results are presented in two columns. The column on the left gives the number of \emph{outer} iterations until convergence was attained whereas the column on the right contains the average number of \emph{inner} iterations expended per outer iteration.

Performance is also analyzed using wall-clock running times. Again, for each method we provide in a left-hand column the total time expended until convergence (in seconds) and in a right-hand column the average running time of each outer iteration (in miliseconds). The times are listed in Table~\ref{label_178} for LASSO, Table~\ref{label_179} for NNLS, Table~\ref{label_180} for L1LR, Table~\ref{label_181} for RR and Table~\ref{label_182} for EN.

We have set the initial Lipschitz constant estimate $L_0$ for all methods to be $L_f$ whereas the line-search parameters are $r_u = 2$ and $r_d = 0.9$, as recommended by \cite{ref_003}. For GMM these parameters are shared across both the Lipschitz search and step size adjustment procedures. For ASRC and ASRM, the restart parameters are $D = e^{-2}$ and $s = 4$, as recommended by our analysis in Subsubsection~\ref{label_148}.

The inner problem for each gradient method with memory employed was always solved using a constrained version of the Fast Gradient Method~\cite{ref_013}, with the target accuracy set to $\delta = 10^{-9}$. Note that $\delta$ is absolute whereas the outer problem $\epsilon$ is relative. On all five test problems, the inner problem target accuracy is actually at least two orders greater than the outer one in absolute terms.

\begin{table}[t]
\caption{Iteration performance on LASSO. For each method, the left column lists the number of outer iterations until convergence while the right column lists the average number of inner interations per outer iteration.}
\label{label_173}
\centering \footnotesize
\begin{tabular}{|r|rr|rr|rr|rr|rr|rr|} \hline
m & \tcol{GMM-C} & \tcol{GMM-M} & \tcol{AGMM-C} & \tcol{AGMM-M} & \tcol{ASRC} & \tcol{ASRM} \\ \hline
1 & 1024 & 0.00 & 1024 & 0.00 & 368 & 0.00 & 368 & 0.00 & 245 & 0.00 & 245 & 0.00 \\
2 & 628 & 51.43 & 837 & 39.00 & 821 & 21.97 & 821 & 21.97 & 272 & 20.64 & 272 & 20.64 \\
4 & 483 & 55.15 & 709 & 48.96 & 1683 & 21.99 & 1496 & 21.74 & 252 & 21.13 & 254 & 20.94 \\
8 & 479 & 55.32 & 652 & 49.22 & 1581 & 21.97 & 2128 & 21.75 & 251 & 21.21 & 185 & 21.26 \\
16 & 479 & 57.30 & 589 & 54.02 & 2032 & 21.98 & 1288 & 21.97 & 203 & 21.13 & 182 & 21.15 \\
32 & 474 & 61.72 & 506 & 57.88 & 2373 & 21.98 & 981 & 21.98 & 257 & 21.23 & 254 & 21.22 \\
64 & 467 & 71.00 & 500 & 68.36 & 2306 & 21.98 & 1166 & 21.98 & 242 & 21.18 & 241 & 21.36 \\
128 & 479 & 89.72 & 484 & 88.72 & 1475 & 21.98 & 973 & 21.98 & 234 & 21.44 & 235 & 21.44 \\ \hline
\end{tabular}
\vspace{3mm}
\captionof{table}{Iteration performance on NNLS}
\label{label_174}
\centering \footnotesize
\begin{tabular}{|r|rr|rr|rr|rr|rr|rr|} \hline
m & \tcol{GMM-C} & \tcol{GMM-M} & \tcol{AGMM-C} & \tcol{AGMM-M} & \tcol{ASRC} & \tcol{ASRM} \\ \hline
1 & 568 & 0.00 & 568 & 0.00 & 261 & 0.00 & 261 & 0.00 & 180 & 0.00 & 180 & 0.00 \\
2 & 401 & 18.37 & 513 & 17.32 & 601 & 21.78 & 601 & 21.78 & 150 & 19.18 & 150 & 19.18 \\
4 & 340 & 21.72 & 440 & 19.38 & 2047 & 21.70 & 896 & 21.47 & 111 & 20.59 & 150 & 20.22 \\
8 & 359 & 23.19 & 392 & 19.24 & 1990 & 21.72 & 888 & 21.54 & 130 & 20.79 & 119 & 19.24 \\
16 & 341 & 25.59 & 378 & 23.78 & 2031 & 21.61 & 548 & 21.90 & 132 & 20.81 & 115 & 20.63 \\
32 & 350 & 26.82 & 338 & 25.93 & 2191 & 21.89 & 564 & 21.95 & 131 & 20.97 & 134 & 20.99 \\
64 & 354 & 33.66 & 357 & 33.35 & 1117 & 21.97 & 614 & 21.96 & 139 & 21.03 & 138 & 21.02 \\
128 & 351 & 54.10 & 359 & 53.25 & 986 & 21.97 & 534 & 21.95 & 164 & 21.31 & 164 & 21.31 \\ \hline
\end{tabular}
\vspace{3mm}
\captionof{table}{Iteration performance on L1LR}
\label{label_175}
\centering \footnotesize
\begin{tabular}{|r|rr|rr|rr|rr|rr|rr|} \hline
m & \tcol{GMM-C} & \tcol{GMM-M} & \tcol{AGMM-C} & \tcol{AGMM-M} & \tcol{ASRC} & \tcol{ASRM} \\ \hline
1 & 156 & 0.00 & 156 & 0.00 & 159 & 0.00 & 159 & 0.00 & 114 & 0.00 & 114 & 0.00 \\
2 & 135 & 24.04 & 149 & 21.09 & 568 & 21.95 & 568 & 21.95 & 148 & 20.89 & 148 & 20.89 \\
4 & 120 & 28.21 & 161 & 36.22 & 2668 & 21.95 & 788 & 21.97 & 142 & 20.55 & 141 & 21.45 \\
8 & 121 & 28.59 & 138 & 32.09 & 2927 & 21.66 & 771 & 21.97 & 132 & 20.86 & 127 & 21.44 \\
16 & 117 & 58.68 & 148 & 36.12 & 2967 & 21.99 & 606 & 21.96 & 130 & 21.49 & 115 & 21.62 \\
32 & 116 & 64.36 & 135 & 54.24 & 748 & 21.97 & 413 & 21.95 & 133 & 21.50 & 112 & 21.61 \\
64 & 116 & 68.26 & 117 & 66.47 & 652 & 21.97 & 350 & 21.94 & 106 & 21.58 & 108 & 21.59 \\
128 & 116 & 69.07 & 116 & 69.07 & 462 & 21.95 & 349 & 21.94 & 106 & 21.58 & 106 & 21.58 \\ \hline
\end{tabular}
\vspace{3mm}
\captionof{table}{Iteration performance on RR}
\label{label_176}
\centering \footnotesize
\begin{tabular}{|r|rr|rr|rr|rr|rr|rr|} \hline
m & \tcol{GMM-C} & \tcol{GMM-M} & \tcol{AGMM-C} & \tcol{AGMM-M} & \tcol{ASRC} & \tcol{ASRM} \\ \hline
1 & 2248 & 0.00 & 2248 & 0.00 & 228 & 0.00 & 228 & 0.00 & 311 & 0.00 & 311 & 0.00 \\
2 & 1428 & 54.22 & 1482 & 42.27 & 189 & 21.88 & 189 & 21.88 & 314 & 21.18 & 314 & 21.18 \\
4 & 878 & 66.51 & 1299 & 52.40 & 178 & 21.88 & 184 & 21.88 & 257 & 21.49 & 314 & 21.40 \\
8 & 799 & 66.91 & 1281 & 58.40 & 177 & 21.88 & 180 & 21.88 & 301 & 21.42 & 316 & 21.44 \\
16 & 806 & 69.23 & 1287 & 62.65 & 178 & 21.88 & 180 & 21.88 & 297 & 21.41 & 309 & 21.43 \\
32 & 809 & 69.63 & 1072 & 65.88 & 183 & 21.88 & 183 & 21.88 & 304 & 21.57 & 305 & 21.50 \\
64 & 818 & 77.73 & 861 & 74.39 & 190 & 21.88 & 191 & 21.88 & 423 & 21.44 & 308 & 21.57 \\
128 & 798 & 90.76 & 814 & 86.90 & 206 & 21.89 & 206 & 21.89 & 381 & 21.54 & 382 & 21.54 \\ \hline
\end{tabular}
\vspace{3mm}
\captionof{table}{Iteration performance on EN}
\label{label_177}
\centering \footnotesize
\begin{tabular}{|r|rr|rr|rr|rr|rr|rr|} \hline
m & \tcol{GMM-C} & \tcol{GMM-M} & \tcol{AGMM-C} & \tcol{AGMM-M} & \tcol{ASRC} & \tcol{ASRM} \\ \hline
1 & 823 & 0.00 & 823 & 0.00 & 159 & 0.00 & 159 & 0.00 & 197 & 0.00 & 197 & 0.00 \\
2 & 487 & 66.27 & 705 & 59.85 & 145 & 21.85 & 145 & 21.85 & 182 & 21.03 & 182 & 21.03 \\
4 & 401 & 75.91 & 612 & 69.42 & 134 & 21.84 & 141 & 21.84 & 151 & 21.13 & 164 & 21.20 \\
8 & 416 & 74.77 & 532 & 73.85 & 136 & 21.84 & 138 & 21.84 & 148 & 21.26 & 154 & 21.14 \\
16 & 404 & 76.31 & 465 & 76.64 & 145 & 21.85 & 145 & 21.85 & 151 & 21.27 & 152 & 21.28 \\
32 & 408 & 83.79 & 398 & 83.13 & 148 & 21.85 & 148 & 21.85 & 153 & 21.28 & 155 & 21.29 \\
64 & 409 & 95.54 & 426 & 92.49 & 167 & 21.87 & 167 & 21.87 & 167 & 21.34 & 166 & 21.34 \\
128 & 404 & 118.61 & 389 & 118.59 & 170 & 21.87 & 170 & 21.87 & 193 & 21.54 & 193 & 21.54 \\ \hline
\end{tabular}
\end{table}

\begin{table}[t]
\caption{Running time performance on LASSO. For each method, the left column lists the total wall-clock running time until convergence (in seconds) while the right column lists the average running time per outer iteration (in miliseconds).}
\label{label_178}
\centering \footnotesize
\begin{tabular}{|r|rr|rr|rr|rr|rr|rr|} \hline
m & \tcol{GMM-C} & \tcol{GMM-M} & \tcol{AGMM-C} & \tcol{AGMM-M} & \tcol{ASRC} & \tcol{ASRM} \\ \hline
1 & 1.21 & 1.18 & 1.20 & 1.18 & 0.48 & 1.30 & 0.48 & 1.30 & 0.32 & 1.30 & 0.32 & 1.30 \\
2 & 0.76 & 1.21 & 1.00 & 1.20 & 1.08 & 1.31 & 1.08 & 1.31 & 0.36 & 1.31 & 0.36 & 1.31 \\
4 & 0.59 & 1.22 & 0.85 & 1.20 & 2.22 & 1.32 & 1.97 & 1.32 & 0.33 & 1.32 & 0.33 & 1.31 \\
8 & 0.59 & 1.23 & 0.80 & 1.22 & 2.10 & 1.33 & 2.83 & 1.33 & 0.33 & 1.32 & 0.24 & 1.32 \\
16 & 0.60 & 1.26 & 0.74 & 1.25 & 2.74 & 1.35 & 1.74 & 1.35 & 0.27 & 1.33 & 0.24 & 1.34 \\
32 & 0.62 & 1.31 & 0.66 & 1.30 & 3.28 & 1.38 & 1.37 & 1.40 & 0.35 & 1.36 & 0.34 & 1.36 \\
64 & 0.67 & 1.43 & 0.71 & 1.41 & 3.40 & 1.47 & 1.78 & 1.53 & 0.34 & 1.40 & 0.34 & 1.39 \\
128 & 0.80 & 1.68 & 0.81 & 1.68 & 2.64 & 1.79 & 1.83 & 1.89 & 0.34 & 1.45 & 0.34 & 1.46 \\ \hline
\end{tabular}
\vspace{3mm}
\captionof{table}{Running time performance on NNLS}
\label{label_179}
\centering \footnotesize
\begin{tabular}{|r|rr|rr|rr|rr|rr|rr|} \hline
m & \tcol{GMM-C} & \tcol{GMM-M} & \tcol{AGMM-C} & \tcol{AGMM-M} & \tcol{ASRC} & \tcol{ASRM} \\ \hline
1 & 2.79 & 4.92 & 2.79 & 4.90 & 1.42 & 5.45 & 1.42 & 5.45 & 0.98 & 5.43 & 0.98 & 5.43 \\
2 & 2.01 & 5.01 & 2.56 & 4.99 & 3.29 & 5.48 & 3.29 & 5.48 & 0.82 & 5.49 & 0.82 & 5.48 \\
4 & 1.71 & 5.02 & 2.21 & 5.02 & 11.23 & 5.49 & 4.91 & 5.48 & 0.61 & 5.49 & 0.82 & 5.49 \\
8 & 1.80 & 5.03 & 1.97 & 5.02 & 10.95 & 5.50 & 4.88 & 5.50 & 0.71 & 5.46 & 0.65 & 5.49 \\
16 & 1.73 & 5.06 & 1.91 & 5.07 & 11.23 & 5.53 & 3.03 & 5.52 & 0.73 & 5.51 & 0.63 & 5.49 \\
32 & 1.80 & 5.14 & 1.72 & 5.10 & 12.24 & 5.59 & 3.15 & 5.59 & 0.72 & 5.51 & 0.74 & 5.53 \\
64 & 1.86 & 5.25 & 1.87 & 5.24 & 6.42 & 5.75 & 3.53 & 5.75 & 0.78 & 5.59 & 0.77 & 5.61 \\
128 & 1.93 & 5.51 & 1.98 & 5.51 & 6.05 & 6.14 & 3.27 & 6.13 & 0.93 & 5.65 & 0.93 & 5.65 \\ \hline
\end{tabular}
\vspace{3mm}
\captionof{table}{Running time performance on L1LR}
\label{label_180}
\centering \footnotesize
\begin{tabular}{|r|rr|rr|rr|rr|rr|rr|} \hline
m & \tcol{GMM-C} & \tcol{GMM-M} & \tcol{AGMM-C} & \tcol{AGMM-M} & \tcol{ASRC} & \tcol{ASRM} \\ \hline
1 & 0.15 & 0.99 & 0.15 & 0.96 & 0.19 & 1.17 & 0.17 & 1.05 & 0.12 & 1.05 & 0.12 & 1.05 \\
2 & 0.14 & 1.01 & 0.14 & 0.97 & 0.61 & 1.07 & 0.61 & 1.07 & 0.16 & 1.06 & 0.16 & 1.06 \\
4 & 0.12 & 1.01 & 0.16 & 0.97 & 2.89 & 1.08 & 0.85 & 1.08 & 0.15 & 1.06 & 0.15 & 1.07 \\
8 & 0.12 & 1.02 & 0.14 & 0.99 & 3.21 & 1.10 & 0.84 & 1.09 & 0.14 & 1.08 & 0.14 & 1.07 \\
16 & 0.13 & 1.07 & 0.15 & 1.02 & 3.35 & 1.13 & 0.68 & 1.13 & 0.14 & 1.09 & 0.13 & 1.09 \\
32 & 0.13 & 1.14 & 0.15 & 1.09 & 0.90 & 1.20 & 0.49 & 1.20 & 0.15 & 1.13 & 0.13 & 1.13 \\
64 & 0.14 & 1.24 & 0.14 & 1.20 & 0.89 & 1.37 & 0.47 & 1.34 & 0.12 & 1.17 & 0.13 & 1.17 \\
128 & 0.15 & 1.29 & 0.15 & 1.25 & 0.78 & 1.70 & 0.57 & 1.63 & 0.13 & 1.19 & 0.13 & 1.19 \\ \hline
\end{tabular}
\vspace{3mm}
\captionof{table}{Running time performance on RR}
\label{label_181}
\centering \footnotesize
\begin{tabular}{|r|rr|rr|rr|rr|rr|rr|} \hline
m & \tcol{GMM-C} & \tcol{GMM-M} & \tcol{AGMM-C} & \tcol{AGMM-M} & \tcol{ASRC} & \tcol{ASRM} \\ \hline
1 & 2.71 & 1.20 & 2.70 & 1.20 & 0.30 & 1.30 & 0.30 & 1.30 & 0.41 & 1.30 & 0.41 & 1.30 \\
2 & 1.74 & 1.22 & 1.79 & 1.21 & 0.25 & 1.31 & 0.25 & 1.31 & 0.41 & 1.31 & 0.41 & 1.31 \\
4 & 1.09 & 1.24 & 1.59 & 1.22 & 0.23 & 1.32 & 0.24 & 1.31 & 0.34 & 1.31 & 0.41 & 1.32 \\
8 & 1.00 & 1.25 & 1.59 & 1.24 & 0.23 & 1.32 & 0.24 & 1.32 & 0.40 & 1.32 & 0.42 & 1.32 \\
16 & 1.03 & 1.28 & 1.63 & 1.27 & 0.24 & 1.34 & 0.24 & 1.33 & 0.40 & 1.34 & 0.41 & 1.34 \\
32 & 1.08 & 1.33 & 1.42 & 1.32 & 0.25 & 1.36 & 0.25 & 1.36 & 0.41 & 1.36 & 0.41 & 1.36 \\
64 & 1.19 & 1.46 & 1.25 & 1.45 & 0.27 & 1.40 & 0.27 & 1.40 & 0.59 & 1.40 & 0.43 & 1.40 \\
128 & 1.38 & 1.73 & 1.41 & 1.73 & 0.31 & 1.49 & 0.31 & 1.49 & 0.56 & 1.47 & 0.57 & 1.48 \\ \hline
\end{tabular}
\vspace{3mm}
\captionof{table}{Running time performance on EN}
\label{label_182}
\centering \footnotesize
\begin{tabular}{|r|rr|rr|rr|rr|rr|rr|} \hline
m & \tcol{GMM-C} & \tcol{GMM-M} & \tcol{AGMM-C} & \tcol{AGMM-M} & \tcol{ASRC} & \tcol{ASRM} \\ \hline
1 & 4.17 & 5.07 & 4.11 & 4.99 & 0.86 & 5.42 & 0.86 & 5.43 & 1.07 & 5.44 & 1.07 & 5.44 \\
2 & 2.47 & 5.06 & 3.54 & 5.02 & 0.79 & 5.47 & 0.79 & 5.47 & 1.00 & 5.48 & 1.00 & 5.48 \\
4 & 2.04 & 5.08 & 3.09 & 5.04 & 0.73 & 5.47 & 0.77 & 5.47 & 0.83 & 5.48 & 0.90 & 5.48 \\
8 & 2.13 & 5.11 & 2.70 & 5.08 & 0.74 & 5.47 & 0.75 & 5.46 & 0.81 & 5.47 & 0.84 & 5.47 \\
16 & 2.08 & 5.15 & 2.39 & 5.14 & 0.79 & 5.48 & 0.79 & 5.48 & 0.83 & 5.51 & 0.84 & 5.50 \\
32 & 2.14 & 5.24 & 2.08 & 5.23 & 0.82 & 5.53 & 0.82 & 5.53 & 0.85 & 5.53 & 0.86 & 5.56 \\
64 & 2.22 & 5.42 & 2.30 & 5.40 & 0.94 & 5.61 & 0.94 & 5.61 & 0.93 & 5.56 & 0.92 & 5.57 \\
128 & 2.34 & 5.80 & 2.25 & 5.79 & 0.97 & 5.70 & 0.97 & 5.70 & 1.10 & 5.69 & 1.10 & 5.69 \\ \hline
\end{tabular}
\end{table}
\clearpage

Eliminating $\delta$ from the convergence analysis allows us to better manage the auxiliary problem overhead by placing bound on the number of inner iterations. For GMM, we chose the limit $N_{\operatorname{inner}} = 1000$. The number of middle iterations was unrestricted for GMM but for all variants of AGMM we have decided that the middle method must employ at most $N_N = 2$ Newton iterations, with the number of inner iterations limited to $N_{\operatorname{inner}} = 11$, thus targeting an average of $10$ inner iterations per middle iteration. We found that increasing either $N_N$ or $N_{\operatorname{inner}}$ beyond this point did not significantly increase performance.

On all test problems, GMM benefits most substantially from the increase in bundle size. For GMM-M, this dependence is very clear, with larger bundle sizes almost always leading to a decreased number of outer iterations until convergence is reached (thereafter only referred to as iterations). This improvement is confirmed by the running times, although for very high bundle sizes the performance benefits diminish due to the high cost of setting up the inner problem. For GMM-C, most gains, both in iterations and in running time, are obtained when the bundle size is in the range $4$ to $32$. These gains are quite remarkable, with the number of iterations and running times halved for LASSO, RR and EN and reduced by around $40\%$ on NNLS. Not only is GMM-C more efficient in terms of bundle size than GMM-M, but the number of outer iterations, again reflected in running times, was consistently lower than for GMM-M on all tested problems. This is rather unexpected, as it contrasts with the preliminary results pertaining to the original GMM on smooth problems~\cite{ref_018}.

More surprisingly, AGMM without restarts does not benefit from using the bundle on the non-strongly convex problems LASSO, NNLS and L1LR. In fact, for AGMM-C small bundles may increase the number of iterations as well as running times by almost an order of magnitude as compared to ACGM and up to $5$-fold when compared to AA. The performance penalty is smaller for AGMM-M, but remains nonetheless prohibitive. Quite the opposite occurs on the strongly convex problems RR and EN, where AGMM-C attains an iteration and running time decrease of around $15\%$ for small bundles, particularly for the size $8$. AGMM-M performs slightly poorer. These results are in concordance with the ones obtained on Exact Gradient Methods with Memory in \cite{ref_005}.

Restart strategies not only mitigate the penalties but actually lead to substantial performance improvements on all problems. On the non-strongly convex ones, the best performance is attained with MRS, with the exception is L1LR, where no more that two restarts are triggered for all strategies.
Therefore, this problem can be considered an outlier. On the strongly convex problems, Restarting AGMM (R-AGMM) lags behind AGMM. This is however due to the former not having access to the strong convexity parameter, which it instead needs to estimate from oracle calls. Under these conditions, the performance of R-AGMM is actually very good. We also notice that similarly to AGMM, R-AGMM converges faster with CRS than with MRS.

In terms of per-iteration complexity, increasing the bundle up to the size $16$ has a negligible impact (less than $5\%$ in running time) on all test problems both for the accelerated methods, where the number of inner iterations per outer iteration is severely restricted, as well as for the fixed-point schemes. This is in perfect agreement with our analysis in Subsections \ref{label_075} and \ref{label_076_overhead}. Bundle sizes beyond $32$ increase the per outer iteration cost, mostly due to the inner problem setup overhead, and partially negate the performance gains made by the bundle. Interestingly, the per iteration running time of the fixed-point schemes is around $10\%$ lower than the corresponding accelerated methods. The main reason is that fixed-point schemes reuse gradient information during backtracks. Despite employing two line-search procedures, this low backtrack overhead manages to compensate for the higher frequency of backtracks.

Based on these preliminary simulation results, we can select a collection of parameters that is small enough to allow the display of convergence rate plots of our methods alongside state-of-the-art schemes. We plot the convergence behavior of our GMM-C, AGMM-M and ASRM, as well as that of FISTA~\cite{ref_002}, AMGS~\cite{ref_014}, ACGM~\cite{ref_008,ref_009} and R-ACGM (Algorithm~\ref{label_154} applied to ACGM) on the non-strongly convex problems LASSO, NNLS and L1LR. On RR and EN, we replace AGMM-M with AGMM-C and ASRM with ASRC. The most consistent performance gains were noticed for $m = 16$ in our previous simulations and we choose this bundle size to plot all methods with memory in our benchmark. The plots are shown for LASSO, NNLS and L1LR in Figure~\ref{label_183} whereas for RR, EN they are shown in Figure~\ref{label_184} .

\begin{figure*}[t] \centering \footnotesize
\begin{minipage}[t]{0.43\linewidth} \centering
\includegraphics[width=\textwidth]{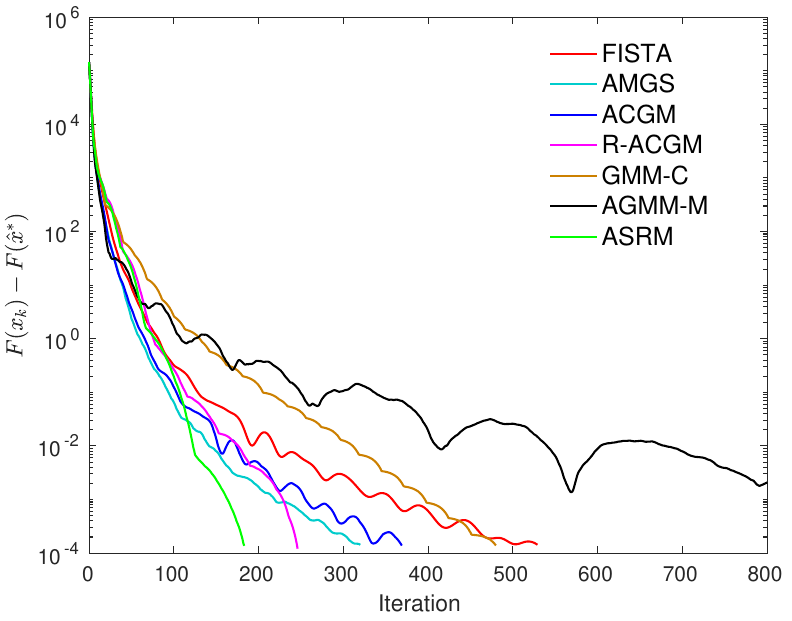}
(a) Convergence rates in iterations on LASSO
\end{minipage}
\hspace{2mm}
\begin{minipage}[t]{0.43\linewidth} \centering
\includegraphics[width=\textwidth]{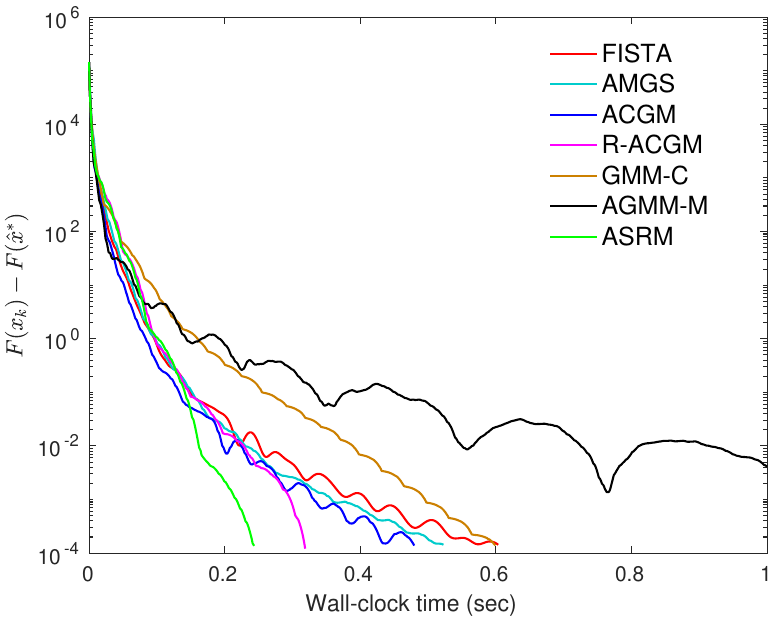}
(b) Convergence rates in wall-clock time on LASSO
\end{minipage}
\vspace{2mm}
\begin{minipage}[t]{0.43\linewidth} \centering
\includegraphics[width=\textwidth]{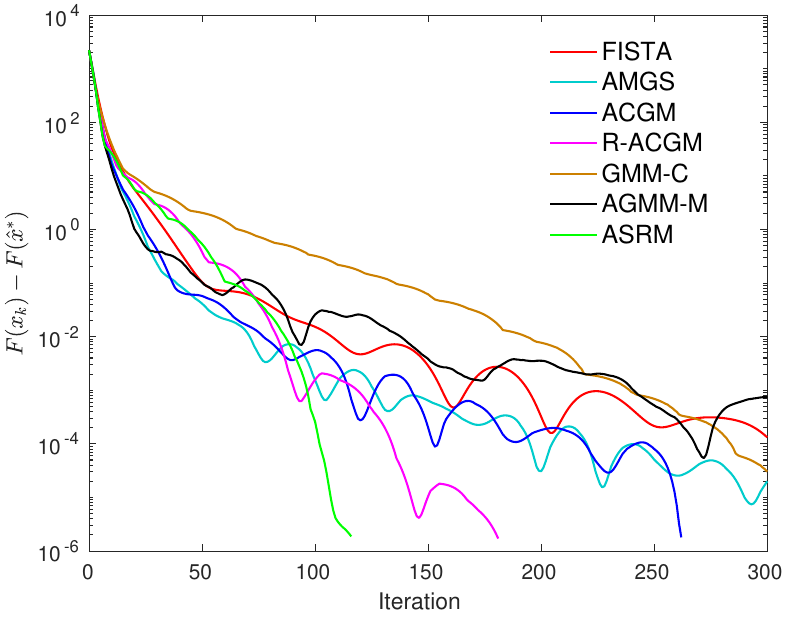}
(c) Convergence rates in iterations on NNLS
\end{minipage}
\hspace{2mm}
\begin{minipage}[t]{0.43\linewidth} \centering
\includegraphics[width=\textwidth]{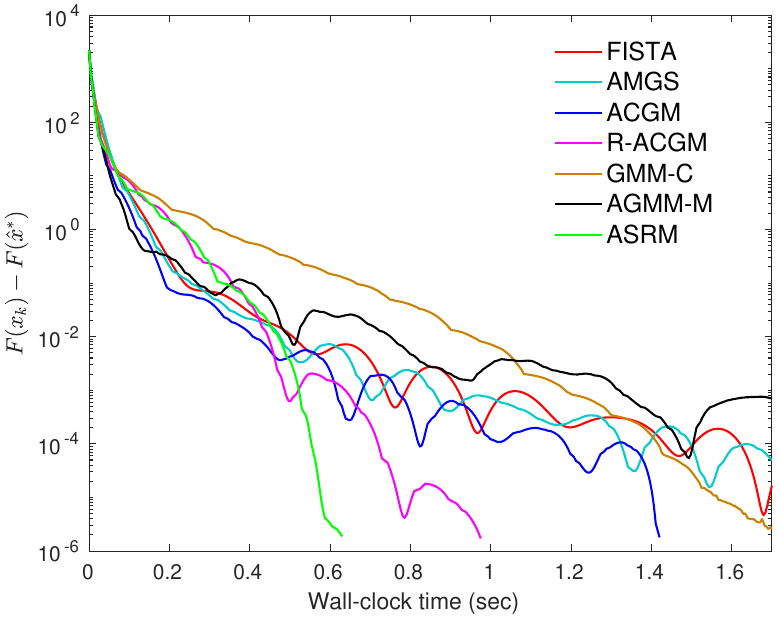}
(d) Convergence rates in wall-clock time on NNLS
\end{minipage}
\vspace{2mm}
\begin{minipage}[t]{0.43\linewidth} \centering
\includegraphics[width=\textwidth]{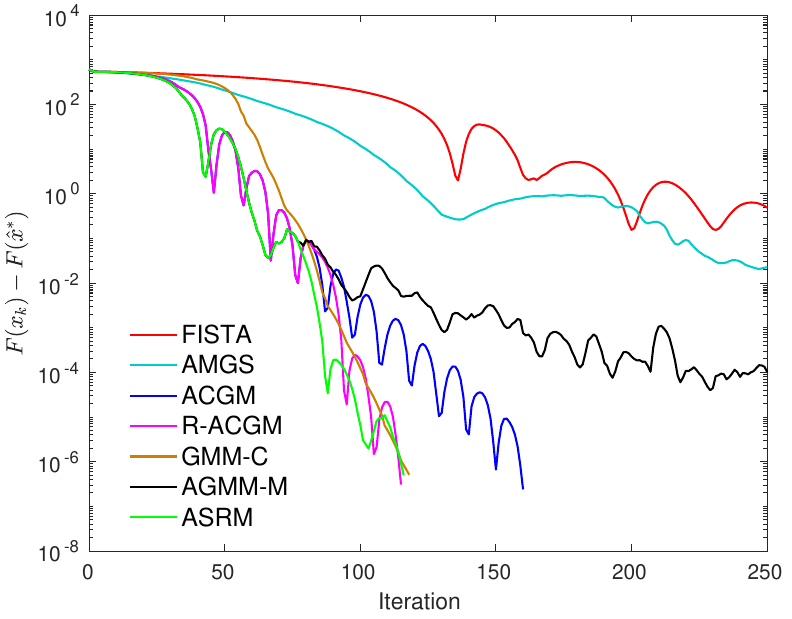}
(e) Convergence rates in iterations on L1LR
\end{minipage}
\hspace{2mm}
\begin{minipage}[t]{0.43\linewidth} \centering
\includegraphics[width=\textwidth]{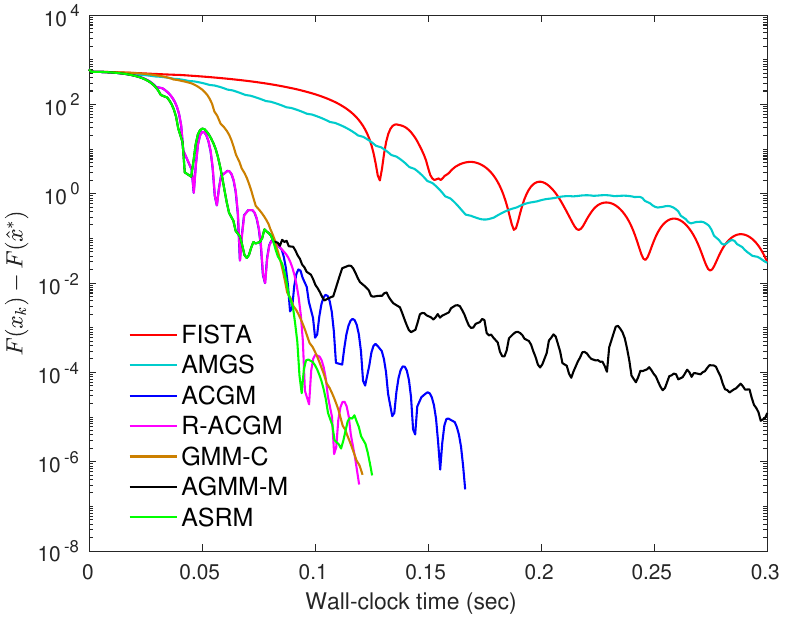}
(f) Convergence rates in wall-clock time on L1LR
\end{minipage}
\caption{Convergence results of GMM-C, AGMM-M and ASRM (all with $m = 16$), as well as those of FISTA, AMGS, ACGM and R-ACGM on the non-strongly convex test problems}
\label{label_183}
\end{figure*}

\begin{figure*}[t] \centering \footnotesize
\begin{minipage}[t]{0.43\linewidth} \centering
\includegraphics[width=\textwidth]{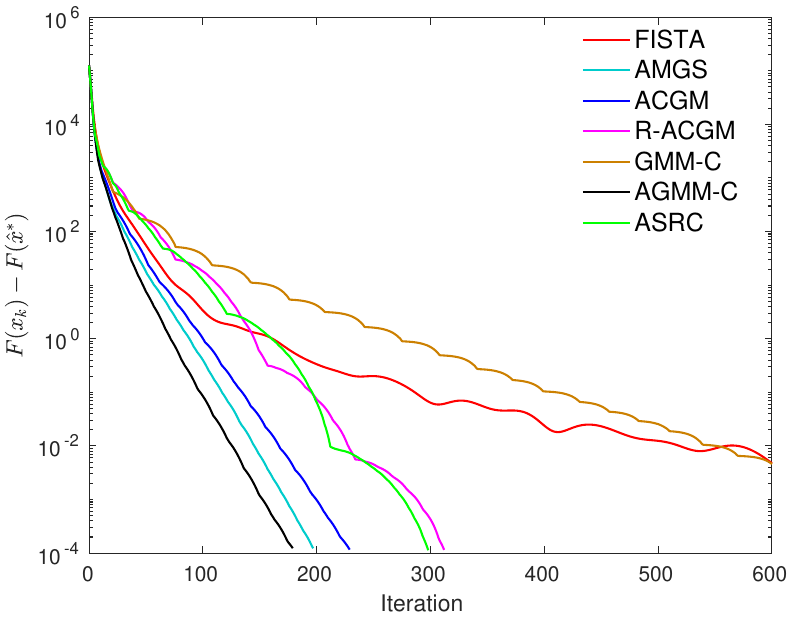}
(a) Convergence rates in iterations on RR
\end{minipage}
\hspace{2mm}
\begin{minipage}[t]{0.43\linewidth} \centering
\includegraphics[width=\textwidth]{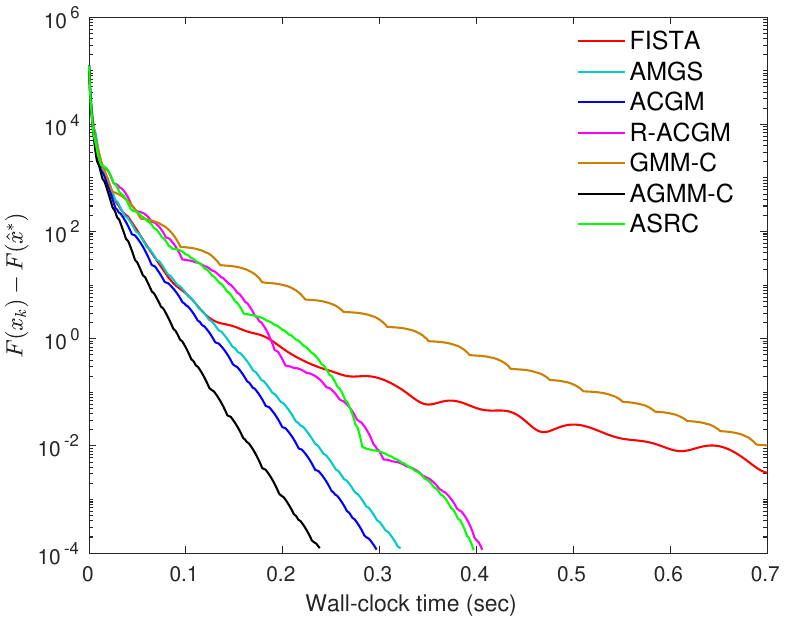}
(b) Convergence rates in wall-clock time on RR
\end{minipage}
\vspace{2mm}
\begin{minipage}[t]{0.43\linewidth} \centering
\includegraphics[width=\textwidth]{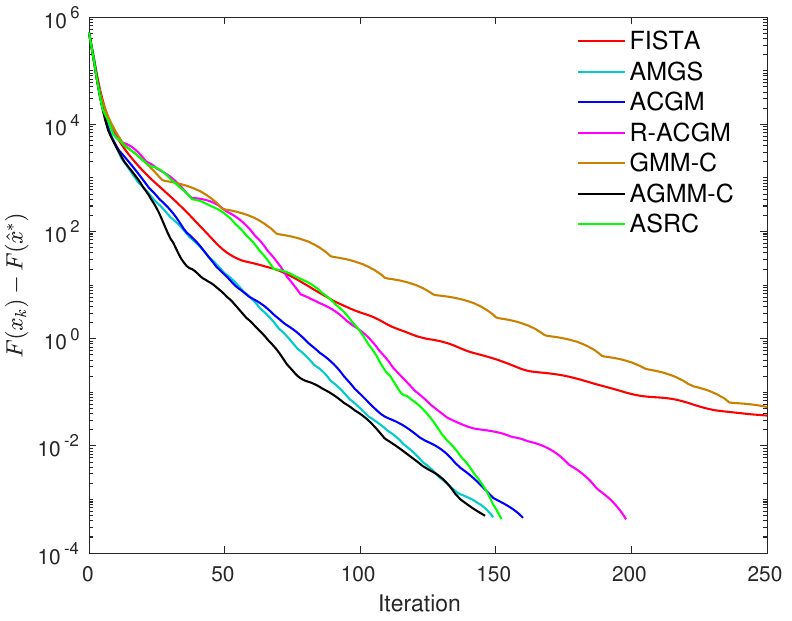}
(c) Convergence rates in iterations on EN
\end{minipage}
\hspace{2mm}
\begin{minipage}[t]{0.43\linewidth} \centering
\includegraphics[width=\textwidth]{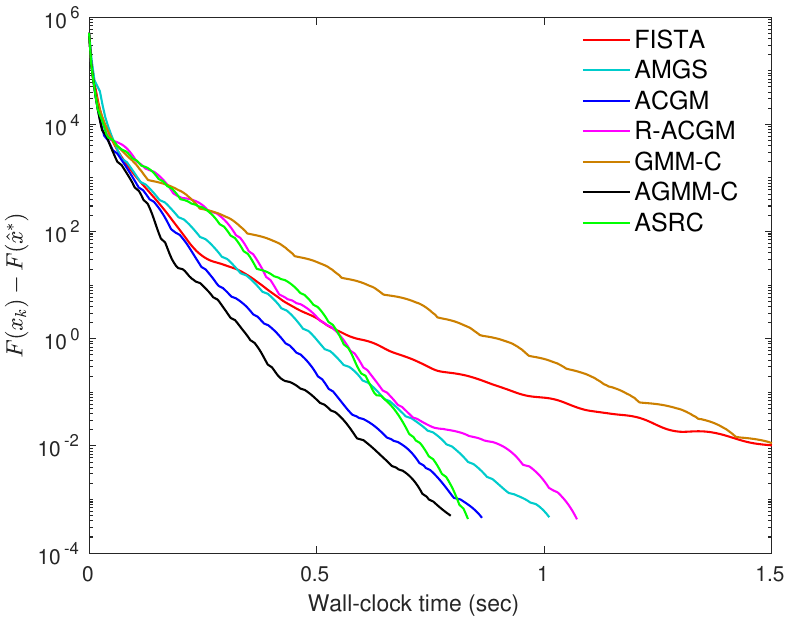}
(d) Convergence rates in wall-clock time on EN
\end{minipage}
\caption{Convergence results of GMM-C, AGMM-C and ASRC (all with $m = 16$) as well as those of FISTA, AMGS, ACGM and R-ACGM on the strongly convex test problems}
\label{label_184}
\end{figure*}

On non-strongly convex problems LASSO and NNLS, ASRM dominates all other methods, followed by R-ACGM. L1LR is an outlier problem where restarts are rarely triggered and the bundle benefits are limited, which is why both ASRC and R-ACGM have a similar convergence behavior. Restarting frequently does seem to be effective on L1LR by looking at the remarkable performance of GMM-C. Performing a restart at every iteration by design, GMM-C is able to match ASRC and R-ACGM on this problem, both in iterations and in running time. GMM-C even manages to compete with FISTA on LASSO and NNLS. By contrast, AGMM-M converges very slowly on every problem it is employed, lagging behind its memory-less counterpart ACGM, which is a top performer among the state-of-the-art.

On the strongly convex problems, AGMM-C is the most effective method. Employing the bundle enables it to clearly surpass ACGM. ASRC performs very well considering that it needs to estimate the strong convexity parameter and on EN it almost matches AGMM-C in performance. ASRC also evidently surpasses the memory-less R-ACGM. GMM-C however is not able to compete here. It is able to take advantage of strong convexity but the rate of convergence is much poorer compared to the accelerated methods. GMM-C surpasses only FISTA, which is unable to utilize strong convexity in any way to increase performance. AMGS performs very well in iterations on all problems except L1LR, approaching ASRM on LASSO and ASRC on RR and EN, but its expensive iterations cause it to lag behind these methods when measuring performance using wall-clock running times.

\section{Discussion} \label{label_185}

In this work we have introduced several efficient Gradient Methods with Memory, specifically designed to address composite problems. All are multi-level schemes, having in common an outer (main) scheme indirectly calling an inner method that solves the auxiliary problem generated by the bundle. The efficiency of all these methods is driven by the model design and initialization, which ensure that the auxiliary (inner) problem can be solved without making any calls to the oracle and that the effects of inexactly solving the inner problem are eliminated from the convergence analysis. The removal of oracle calls from the inner problem has been achieved by adding composite gradients to the bundle. The gradients used in the original GMM~\cite{ref_018} include information only on the smooth part $f$ but the composite gradients also capture information on $\Psi$. The composite gradient data collection imposes no oracle overhead and the bundle can be constructed solely from information already collected to advance the algorithm. Moreover, if the inner method is initialized with the solution obtained by the outer scheme without the bundle, the inner scheme will produce a solution no worse than the starting point thus negating the effects of the solution's inexactness on the outer scheme.

Using this efficient configuration, we have proposed a fixed-point scheme in Algorithm~\ref{label_030}, which we also denote as the Gradient Method with Memory (GMM) for composite objectives. Because the quality of the inner problem solution is no worse (and in practice significantly better) than the one of its memory-less counterpart GM, we can increase the step size to a much higher level than that of GM. On the other hand, the model contains composite gradients that also incorporate a Lipschitz constant estimate. We thus distinguish two distinct line-search procedures: an initial Lipschitz search and a subsequent step size search. The step size search constitutes a middle scheme that repeatedly calls the inner scheme until its stopping criterion is satisfied. The evaluation of the stopping criterion for the middle scheme does however require oracle calls, a drawback inherited from the original GMM formulation in \cite{ref_018}.

On the other hand, with inner solution inexactness negated, we can apply acceleration to GMM to obtain an efficient Accelerated Gradient Method with Memory (AGMM), as listed in Algorithm~\ref{label_093}, that does not exhibit error accumulation~\cite{ref_004}. By incorporating the bundle in the estimate functions, we attain the optimal worst-case rate on the problem class along with the ability to increase the convergence guarantees beyond this level. The convergence guarantee increase procedure in Algorithm~\ref{label_086} becomes a middle method for AGMM, taking on the role of the step size search procedure in GMM. Unlike GMM, the middle problem in AGMM can be solved using Newton's rootfinding algorithm without making any calls to the oracle. Newton's method is further parameter-free and its efficiency dramatically lowers the number of inner iterations required to capture the performance benefit induced by the bundle. Thus, apart from the complexity of the inner problem setup, occurring exactly once per outer iteration, the overhead introduced by the bundle can be shown to be virtually non-existent.

We have also studied how our methods can be altered to take advantage of additional problem structure, in particular strong convexity and relaxations of this condition such as the quadratic functional growth (QFG) property. GMM can exploit both strong convexity and QFG by design, with only minor modifications needed to account for strong convexity in the regularizer (see \eqref{label_067} through \eqref{label_068}). The rates obtained for strong convexity also apply to a larger class of problems where the bundle can be shown to satisfy \eqref{label_057} in Proposition~\ref{label_033_sc}. This property resembles quasi-strong convexity~\cite{ref_011}. Determining its exact scope and finding relevant practical non-strongly convex applications is left as a topic for future research. Despite GMM's flexibility, the worst-case rates on the entire class, as well as the subclasses studied, remain far below the optimal ones.

Much better rates can be obtained for AGMM on these subclasses but require more substantial alterations. A generalized version of AGMM that makes use of known strong convexity, listed in Algorithm~\ref{label_119}, incorporates strongly convex lower bounds in its estimate functions, following a procedure previously described in the state-of-the-art method ACGM~\cite{ref_008,ref_009}. However, to be able to increase the convergence guarantee beyond that of ACGM, as shown in Algorithm~\ref{label_086-sc-method}, we must restrict ourselves to unaugmented estimate functions. Interestingly, the main convergence result in Theorem~\ref{label_127} only requires known strong convexity in the regularizer, which can be readily attained due to the simplicity of $\Psi$. The lower bounds aggregated into the estimate functions need not be global, and the convergence rate obtained in Subsubsection~\ref{label_120} remains valid when $f$ has the more general quasi-strong convexity property, provided that all iterates project onto the same optimal point. Remark~1 in \cite{ref_011} describes a subclass of problems that are not strongly convex but satisfy the above conditions.

To deal with QFG, we have turned to restart strategies. The idea of restarting is very general and can be applied to any optimization scheme with a sublinear worst-case rate that satisfies \eqref{label_139}. The need for considering values of the exponent $p$ other than $2$ arises when dealing with universal methods~\cite{ref_015} or future methods dealing with narrower problem classes. For instance, recent developments on higher-order methods in \cite{ref_016,ref_017} may lead to the emergence of new super-fast methods applicable to subclasses of composite problems defined in terms of higher-order conditions. We have shown that for any such algorithm, the best rate can be obtained if the restart occurs when the distance to the optimum measured in function value decreases by a fixed factor $D$. In the known growth parameter (KGP) case, it is straightforward to determine the value of $D$ that gives the best worst-case rate. When the growth parameter is not known, we propose an adaptive restart strategy in Algorithm~\ref{label_154}.

One would expect that the ratio between the number of outer iterations of the restarted scheme necessary for the adaptive scheme to reach the same level of objective accuracy as in the KGP case, a ratio that we denote as the iteration efficiency, to be very small. However, when considering the asymptotic rates, the iteration efficiency for values of $D$ in the range $D^*$ (the optimal value according to our otherwise conservative analysis) to $e^{-p}$ (the optimal value for KGP case derived in \cite{ref_011} that produces good results in our experiments) is over $92\%$ for all $p \geq 2$ when the search parameter $s$ is close to $1$. This increases to $99\%$ for $p \geq 3.5$. In other words, the computational resource penalty of not knowing the growth parameter can be rendered \emph{less than $8\%$} in the long term for \emph{every} restarted every scheme that manages to take advantage of composite problem structure, which is quite remarkable.

Whereas adaptive restart has already been proposed in \cite{ref_001}, our approach has a number of important improvements. It does not impose conditions on the optimization scheme being restarted, aside from sublinear convergence guarantees. We can thus restart methods that are able to dynamically estimate other problem parameters such as the Lipschitz constant of the smooth part gradient. Our analysis does not rely on a target accuracy, being thus applicable to online methods, and acknowledges that the number of adjustments of the wrapper scheme is always bounded, attaining an astoundingly efficient asymptotic rate. Note that we are also able to remove the redundant proximal gradient steps utilized in \cite{ref_001} to place a bound on the composite gradient, resulting in a slightly better overall convergence rate.

When the restart strategy in Algorithm~\ref{label_154} is applied to strong convexity agnostic AGMM in Algorithm~\ref{label_093}, we obtain a Restarting AGMM (R-AGMM). Whereas GMM and AGMM are three-level schemes, R-AGMM becomes a four-level scheme, with the outermost layer (Algorithm~\ref{label_154}) denoted as a wrapper. Its complexity renders R-AGMM a fully adaptive method, being able to simultaneously estimate both the growth parameter as well as the Lipschitz constant. It also possesses one of the highest known worst-case rates on its class, as well as on the entire composite problem class (even when no restarts are necessary) and requires approximately $3$ times (when the parameter $s$ is set very close to $1$) more outer iterations to attain the asymptotic rate of AGMM (Algorithm~\ref{label_119}) on strongly convex objectives. R-AGMM is thus a very robust scheme that is able to achieve state-of-the-art performance on a multitude of important problem classes not just by automatically and dynamically adjusting key geometry parameters, but also by adaptively increasing the convergence guarantees.

We have tested GMM, AGMM and R-AGMM with a variety of parameter choices on a collection of synthetic problems spanning the areas of statistics, inverse problems and machine learning. GMM benefited the most from the bundle having an all-round stable performance, far superior to the worst-case guarantees. Discarding older information (CRS) proved more effective than keeping the composite gradients with the smallest norm (MRS). Not only was the number of iterations required for convergence noticeably small, but the average iteration running time was slightly smaller than its accelerated counterparts. Despite having two separate line-search procedures, fixed-point backtracking kept the line-search overhead at a low level.

On the non-strongly convex problems, AGMM suffered a severe performance penalty. This is most likely due to some form of interference between the bundle based acceleration and the estimate sequence acceleration. The CRS seems to contribute to this penalty more than MRS. The interference is mitigated by restarting, with R-AGMM generally outperforming its memory-less counterpart R-ACGM. Restarting also preserves the beneficial effects of MRS. Soft restarting was found to perform slightly better than hard restarting, but the difference was not found to be significant.

No interference was noticed on the strongly convex problems, with AGMM leading all other methods, partly due to its knowledge of the strong convexity parameter $\mu$. The benefits of employing the bundle were substantial in this case. Restating proved very effective when $\mu$ is not known, R-AGMM being only slightly behind AGMM.

Small bundles captured most performance benefits on all problems for both AGMM and R-AGMM, with the values $m = 8$ and $m = 16$ leading to the fastest convergence. For R-AGMM, our experimental results confirm that the $D = e^{-2}$ is better than the choice $D = 1 / (e + 1)$ made in \cite{ref_001}, with the selection of $s$ being problem specific. Simulations also suggest that when no information on the problem is available beforehand, R-AGMM is most likely to perform well with soft restarting and MRS.

Overall, R-AGMM has proved to be the best method, both theoretically due to its robustness, adaptivity and excellent worst-case rate as well as in practice, converging very quickly on all the test problems, regardless of their structure. Whereas GMM restarts at every iteration, thereby implicitly estimating strong convexity or a relaxed condition, but does not employ acceleration, and AGMM uses the estimate sequence to accelerate but lacks restarts, R-AGMM constitutes a hybrid between the two, cumulating their strengths to attain outstanding efficacy.

\section*{Acknowledgements}

We thank Yurii Nesterov for very helpful discussions and suggestions.

\section*{Funding}

This project has received funding from the European Research Council (ERC) under the European Union's Horizon 2020 research and innovation programme (grant agreement No. 788368).

\end{document}